\documentclass[11pt,a4paper]{article}
\usepackage{amsfonts}
\usepackage{amsmath,amssymb}
\usepackage{mathrsfs}
\usepackage{hyperref}
\usepackage{graphicx}
\usepackage{float}
\usepackage{amsfonts}
\usepackage{bbm}
\usepackage{amsmath,amscd,amsbsy,amssymb,latexsym,url,bm,amsthm}
\usepackage{mathrsfs, euscript}
\usepackage{lineno}

\newtheorem{thm}{Theorem}[section]
\newtheorem{cor}[thm]{Corollary}
\newtheorem{lem}[thm]{Lemma}
\newtheorem{prop}[thm]{Proposition}
\newtheorem{defn}[thm]{Definition}

\newtheorem{rem}[thm]{Remark}

\makeatletter \renewenvironment{proof}[1][Proof]
{\par\pushQED{\qed}\normalfont\topsep6\p@\@plus6\p@\relax\trivlist\item[\hskip\labelsep\bfseries#1\@addpunct{.}]\ignorespaces}{\popQED\endtrivlist\@endpefalse} \makeatother
\makeatletter
\numberwithin{equation}{section}\allowdisplaybreaks
\def\leq{\leqslant}

\def\eps{\varepsilon}
\def\leq{\leqslant}
\def\geq{\geqslant}

\def\R{{\mathbb{R}}}

\def\FF{{\mathscr{F}^{-1}}}
\def\F{{\mathscr{F}}}
\def\Z{{\mathbb{Z}}}
\renewcommand{\S}{\mathcal{S}}

\newcommand{\ddt}{\frac{d\ }{dt}}

\begin{document}

\title{\large\bf  Global Well-posedness for the Fourth-order Nonlinear Schr\"odinger Equation}

\author{\normalsize \bf Mingjuan Chen$^{a}$,  \bf Yufeng Lu$^{b,}\footnote{Corresponding author.}$\ ,\   and \bf Yaqing Wang$^{a}$\\
\footnotesize
\it $^a$ Department of Mathematics, Jinan University, Guangzhou 510632, P.R. China, \\ \footnotesize
\it $^b$ School of Science, Jimei University, Xiamen 361021, P.R. China, \\ \footnotesize
\it E-mails: mjchen@jnu.edu.cn, lyf@jmu.edu.cn \\
} \maketitle

\thispagestyle{empty}
\begin{abstract}
The local and global well-posedness for the one dimensional fourth-order nonlinear Schr\"odinger equation are established in
the modulation space $M^{s}_{2,q}$ for $s\geq \frac12$ and $2\leq q <\infty$. The local result is based on the $U^p-V^p$ spaces and crucial bilinear estimates. The key ingredient to obtain the global well-posedness is that we achieve a-priori estimates of the solution in modulation  spaces by utilizing the power series expansion of the perturbation determinant introduced by Killip-Visan-Zhang for completely integrable PDEs. \\


\end{abstract}

\section{Introduction}
%%%%%%%%%%%%%%%%%%%%%%%%%%%%%%%%%%%%%%%%%%%%%%%%%%%%%%%%%%%%%%%%%%%%%%%%%%%%%%%%%%%%%%%%%%%%%%%%%%%%%%%%%%%%%%%%%%%%%%%%%%%%%%

In this paper, we consider the Cauchy problem for the fourth order nonlinear Schr\"odinger equation (4NLS) on the line $\mathbb{R}$:
\begin{equation}\label{4NLS}
\left\{\begin{array}{l}
	iu_{t}+u_{xxxx}+8u_{xx}|u|^{2}+2\overline{u}_{xx}u^{2}+6u_{x}^{2}\overline{u}+4u|u_{x}|^{2}+6u|u|^{4}=0,  \\
	u(x,0) =u_{0}(x) ,
\end{array}\right.
\end{equation}
where $u$ is a complex valued function of $(x,t)\in \mathbb{R} \times  \mathbb{R}$, and the fixed coefficients in the
nonlinearities could ensure the complete integrable structure. 4NLS \eqref{4NLS}, as a member of the NLS integrable hierarchy, can be expressed as the compatibility conditions of linear eigenvalue equations. For a more general class of the 4NLS:
\begin{equation}\label{4NLS2}
	iu_{t}+u_{xx}+2u|u|^{2}+\gamma(u_{xxxx}+8u_{xx}|u|^{2}+2\overline{u}_{xx}u^{2}+6u_{x}^{2}\overline{u}+4u|u_{x}|^{2}+6u|u|^{4})=0,  \end{equation}
which can be used to describe the propagation of ultrashort optical pulses in long-distance high-speed optical fiber transmission systems \cite{A09,D01}, or the three dimensional motion of an isolated vortex filament embedded in an inviscid incompressible fluid \cite{Da1906}. In magnetic mechanics, it can describe the nonlinear spin excitations in one-dimensional isotropic biquadratic Heisenberg ferromagnetic spin with the octuple-dipole interaction \cite{D99}.

It is worth noting that 4NLS \eqref{4NLS} is endowed with two important features. One is the complete integrability: in the sense that it admits Lax pair formulations, thus one can show that the solution exists globally in time for any Schwartz initial data. The other one is the scaling symmetry: for any solution $u(x,t)$ to \eqref{4NLS} with initial data $u_{0}(x)$, then the $\lambda$-scaled function
$ u_{\lambda } (x,t)=\lambda u(\lambda x,\lambda^{4}t)$ is also a solution to \eqref{4NLS} with $\lambda$-scaled initial data $u_{0,\lambda } (x)=\lambda u_0(\lambda x)$. The Sobolev space $\dot{H}^{s_c}$ is called the scaling-critical space if it satisfies $\|u_{0,\lambda}\|_{\dot{H}^{s_c}}\sim\|u_0\|_{\dot{H}^{s_c}}.$ By a simple calculation, one can know that $s_c=-1/2$. We reasonably predict that \eqref{4NLS} is ill-posed in $H^s$ when $s<-1/2$.

The Cauchy problem for the 4NLS has been studied extensively by many researchers. Segata \cite{segata2003well} established the local-wellposedness of \eqref{4NLS2} excluding the nonlinearity  $u_{xx}|u|^{2}$ in the Sobolev space $H^s(\mathbb{R})$ for $s\geq 1/2$ under the condition $\gamma<0$. Then Huo and Jia \cite{huo2005cauchy} relaxed the restriction $\gamma<0$ while still not incorporating $u_{xx}|u|^{2}$. For the complete \eqref{4NLS2} which includes the nonlinearity  $u_{xx}|u|^{2}$, Segata \cite{segata2004remark} obtained the the local-wellposedness in $H^s(\mathbb{R})$ for $s> 7/12$ under the condition $\gamma<0$. Huo and Jia \cite{huo2007refined} eliminated the requirement $\gamma<0$ and improved the regularity to $s> 1/2$ by employing
a crucial maximal function estimate. In addition, it is known that \eqref{4NLS2} is ill-posed in $H^s$ for $s\in (-1/2,1/2)$ in the sense that the data-to-solution map fails to be uniformly continuous, see \cite{maeda2011existence,segata2010well}.  There are also many studies focusing on the 4NLS with different nonlinearities or in the probabilistic setting, see, for example, \cite{kwak2018periodic,oh2018global,segata2012refined} on the torus, \cite{hao2007well,pausader2009cubic,pausader2010mass,ruzhansky2016global,seong2021well} on the real line, \cite{chen2020random,dinh2021probabilistic,oh2020solving} in the probabilistic setting.

In this work, we study the sharp global well-posedness of the 4NLS \eqref{4NLS} in modulation spaces $M^s_{2,q}(\mathbb{R})$ for $s
\geq 1/2$ and $2\leq q<\infty$. The sharpness arises from the ill-posed theory in \cite{maeda2011existence,segata2010well}, which implies that the 4NLS can not be solved in the Sobolev space $H^s(\mathbb{R})$ with $s<1/2$ by a contraction mapping principle. Many researchers have directed their attention towards solving nonlinear PDEs in modulation spaces, see \cite{Gu16,GRW21,KaKoIt14,OhWang20,Wa13} and the references therein. We recall the definition of the modulation spaces $M^s_{p,q}(\mathbb{R})$, which was introduced by Feichtinger \cite{Fei2}. Let $\psi \in \S(\R)$ such that
\begin{align*}
{\rm supp}\ \psi \subset [-1, 1],
\qquad  \psi_n(\xi) := \psi(\xi - n) \qquad \text{and} \qquad \sum_{n \in \Z} \psi_n(\xi) \equiv 1,
\end{align*}
then the modulation spaces $M^s_{p,q}(\mathbb{R})$ are defined as the collection of all tempered distributions
$f\in\S'(\R)$ such that $\|f\|_{M^s_{p,q}}<\infty$, where
the $M^s_{p,q}$-norm can be equivalently defined in the following way (cf.
\cite{WaHu07,WaHaHuGu11}):
\begin{align}
\|f\|_{M^s_{p,q}(\mathbb{R})}= \big\| \langle n \rangle^{s}
\|\Box_n f\|_{L^p(\mathbb{R})} \big\|_{\ell^q_n(\mathbb{Z})},  \label{mod-space1}
\end{align}
where $\Box_n = \mathscr{F}^{-1} \psi_n(\xi) \mathscr{F}$, $\mathscr{F}$ ($\mathscr{F}^{-1}$) denotes the (inverse) Fourier transform on $ \mathbb{R}$, and  $\langle n \rangle=(1+|n|^2)^{1/2}$.

By incorporating the $U^p$ and $V^p$ spaces into modulation spaces, and applying Strichartz-type estimates, bilinear estimates and $L^4$ estimates, we derive nonlinear estimates and achieve local well-posedness in modulation spaces. In order to upgrade local well-posedness to global well-posedness, we need to develop the specific properties of the 4NLS \eqref{4NLS} and obtain the global bounds of the solutions. The main idea is from Killip-Visan-Zhang \cite{KVZ18} where they exploited the complete integrable structure and obtained the conservation laws on the $H^s$-norm of solutions for KdV, NLS and mKdV equations. The 4NLS \eqref{4NLS} belongs to the classical ZS-AKNS system and admits the following Lax pair formulation \cite{Zhang09}:
$$
\frac{d\ }{dt}  L(t;\kappa) = [P(t;\kappa), L(t;\kappa)] \quad\text{with}\quad
	L(t;\kappa) = \begin{bmatrix}-\partial_x+\kappa & u(x) \\ -\overline{u}(x) & -\partial_x-\kappa \end{bmatrix}
$$
and some operator pencil $P(t;\kappa)$ whose precise form will play no role in this paper. We define the following perturbation determinant
\begin{align}\label{alpha}
\alpha(\kappa;u) := {\rm Re} \sum_{\ell=1}^\infty \frac{1}{\ell} {\rm tr}\Bigl\{\bigl[(\kappa-\partial_x)^{-1/2} u (\kappa+\partial_x)^{-1} \overline{u} (\kappa-\partial_x)^{-1/2} \bigr]^\ell\Bigr\}
\end{align}
which formally represents
\begin{align}
-{\rm Re}\log\det\left( \begin{bmatrix} (-\partial_x+\kappa)^{-1} & 0 \\ 0 & (-\partial_x-\kappa)^{-1} \end{bmatrix}
	\begin{bmatrix}-\partial_x+\kappa & u(x) \\ - \overline{u}(x) & -\partial_x-\kappa \end{bmatrix} \right).
\end{align}
Here, the trace of an operator $A$ on $L^{2}(\mathbb{R} )$ (${\rm tr} A$) is defined in Definition \ref{tr}.
We will show that $\alpha(\kappa;u)$ is conserved under the 4NLS flow \eqref{4NLS} in Proposition \ref{conser0}. The $M^s_{2,q}(\mathbb{R})$-norm of solutions can be related to the size of the leading term in the series $\alpha(\kappa;u(t))$, therefore it is almost conserved and can be controlled  by the norm of the initial data. Finally, the local well-posedness can be extended to global well-posedness.

In this work, we shall prove the following theorems:
\begin{thm}\label{theorem1}
Let $2\leq q <\infty$, $s\geq \frac12$, and $u_0 \in  M^{s}_{2,q}$. Then there exists a time $T>0$ such that 4NLS \eqref{4NLS} is locally well posed in $C([0,T];  M^{s}_{2,q}) \cap X^{s}_{q,S}([0,T])$, where $X^{s}_{q,S}$ is defined in Definition \ref{Xsq}.
\end{thm}

\begin{thm}\label{theorem2}
	Fix parameters $s$ and $q$ conforming to either of the following restrictions: $q=2$ and $-\frac12<s<1$, or  $2< q< \infty $ and $0\leq s<1-\frac{1}{q}$. Then there exists a nondecreasing function $C_{s,q}:[0,\infty)\rightarrow [0,\infty)$ such that
\begin{align}\label{Result}
\sup_{t\in\mathbb{R}}\| u(t) \|_{M^s_{2,q}} \lesssim  C_{s,q}(\| u_0 \|_{M^s_{2,q}}).
\end{align}
for any Schwartz class solution $u$ to \eqref{4NLS} on $\mathbb{R} $.
\end{thm}

As a consequence of Theorem \ref{theorem1} and Theorem \ref{theorem2}, we can establish the subsequent global well-posedness.
\begin{thm}\label{theorem3}
	Let $2\leq q <\infty$, $s\geq \frac12$, and $u_0 \in  M^{s}_{2,q}$. Then 4NLS \eqref{4NLS} is globally well-posed.
\end{thm}

\begin{rem}
For $q=2$ and $\frac12\leq s<1$, or  $2< q< \infty $ and $\frac12\leq s<1-\frac{1}{q}$, Theorem \ref{theorem3} is a corollary to Theorem \ref{theorem1} and Theorem \ref{theorem2}. For larger $s$, a persistence of regularity argument
with the global bounds on the $M^{1/2}_{2,q}$-norms of solutions implies the corresponding global well-posedness.
\end{rem}
\begin{rem}
Local well-posedness in Theorem \ref{theorem1} is achieved by using PDE techniques and the contraction principle, without requiring the complete integrable structure. On the other hand, for the general 4NLS operator $i\partial_t+\Delta^2\pm\Delta$, we can combine Biharmonic estimates in low frequency with Strichartz estimates to obtain nonlinear estimates. Therefore, Theorem \ref{theorem1} is available for the following general 4NLS equation:
\begin{equation*}
	iu_{t}+u_{xx}+\gamma u_{xxxx}+a_1u|u|^{2}+a_2u_{xx}|u|^{2}+a_3\overline{u}_{xx}u^{2}+a_4u_{x}^{2}\overline{u}+a_5u|u_{x}|^{2}+a_6u|u|^{4}=0,  \end{equation*}
which improves the local results in \cite{huo2005cauchy,huo2007refined,segata2003well,segata2004remark}.
\end{rem}
\begin{rem}
For the spectral parameter $\kappa\in \mathbb{C}$ in \eqref{alpha}, its real part and imaginary part correspond to dilation and translation transformations, respectively. Therefore, the advantage of the imaginary part of $\kappa$ can be fully developed in the modulation space. Combining with the scaling symmetry of the 4NLS equation \eqref{4NLS} itself, we can obtain Theorem \ref{theorem2} for $2< q< \infty $ and $\frac12\leq s<1-\frac{1}{q}$. If the equation does not enjoy the scaling symmetry, the result shall need the small initial condition. However, for $q=2$ and $\frac12\leq s<1$ ($M^s_{2,2}=B^s_{2,2}=H^s$), we can choose ${\rm Re} \kappa$ large enough to remedy the absence of scaling symmetry and achieve the result for large initial data in Besov space, see \cite{Chen23}.
\end{rem}
Therefore, we have the following corollary, which is the sharp global well-posedness via the contraction principle method, since the data-to-solution map fails to be uniformly continuous when $s<1/2$, see \cite{segata2010well}.
\begin{cor}
 The 4NLS equation \eqref{4NLS2} is globally well-posed in $H^s$ for $s\geq 1/2$.
\end{cor}

{\bf Notations.} Let $c < 1$, $C>1$  denote positive universal constants, which can
be different at different places. We use the notation $a \lesssim b$ if there is an independently constant $C$ such that $a \leqslant Cb$. Denote $a\sim b$ if $a\lesssim b$ and $b\lesssim a$. $a\approx b$ means that $|a-b|\leq C$. Denote $a\vee b= \max \{a, b\}, a\wedge b=\min \{a,b\}.$ $p'$ is the dual number of $p \in
[1,\infty]$, i.e., $1/p+1/p'=1$.
The (inverse) Fourier transform $(\FF) \F $ in $\mathbb{R}$ can be defined as follows:
\begin{align*}
	(\F f)(\xi) = \hat{f}(\xi) = \frac{1}{\sqrt{2\pi}}\int_{\mathbb{R}} f(x) e^{-ix\xi} \ d x,\quad
	(\FF f)(x) = \check{f}(x)=\frac{1}{\sqrt{2\pi}} \int_{\mathbb{R}}f(\xi) e^{ix\xi}\ d\xi.
\end{align*}
Let $I\subset \mathbb{R}$ be an interval of finite length, we denote
$ u_\lambda = \Box_\lambda u, \ \  u_I = \sum_{\lambda\in I \cap \mathbb{Z}} u_\lambda.$

\section{Preliminaries}

In this section, we recall some function spaces utilized in establishing the well-posedness theory for nonlinear dispersive equations.  $U^p$ and $V^p$ spaces were initially introduced by Koch and Tataru \cite{CHT12,KoTa05,KoTa07,KoTa12} in the study of NLS and mKdV. They have played a crucial role in establishing the well-posedness of various equations, including the KP-II equation \cite{HaHeKo09}, the cubic NLS equation \cite{Gu16,KoTa07,KoTa12}, the DNLS equation \cite{GRW21}, and the mKdV equation \cite{Chen20}.

\subsection{Function spaces}
\begin{defn}\label{tr}
Let $A$ be an operator on $L^{2}(\mathbb{R} )$ with continuous integral kernel $K(x,y)$. The trace of operator $A$ is defined by
$$\mathrm{tr}(A):=\int K(x,x)dx.$$
If $A$ is a Hilbert-schmidt operator with integral kernel $K(x,y)\in L^{2}(\mathbb{R}\times \mathbb{R} )$, then
$$\mathrm{tr}(A^{2} )=\iint_{\mathbb{R}^{2} } K(x,y)K(y,x)dxdy,$$
and
\begin{align}
\left \| A \right \| _{\mathfrak{I}_{2}  }^{2} :=\mathrm{tr}(AA^{*} )=\iint_{\mathbb{R}^{2} }\left | K(x,y) \right | ^{2}dxdy.\label{A}
\end{align}
\end{defn}

Let $\mathcal{Z}$ be the set of finite partitions $-\infty= t_0 <t_1<...< t_{K-1} < t_K =\infty$. In the following, we consider functions taking values in $L^2:=L^2(\mathbb{R}^d;\mathbb{C})$, but in the general $L^2$ can be replaced by an arbitrary Hilbert space or general Banach space.

\begin{defn}
Let $1\leq p <\infty$. For any $\{t_k\}^K_{k=0} \in \mathcal{Z}$ and $\{\phi_k\}^{K-1}_{k=0} \subset L^2$ with $\sum^{K-1}_{k=0} \|\phi_k\|^p_2=1$, $\phi_0=0$. A step function $a: \mathbb{R}\to L^2$ given by
$$
a= \sum^{K}_{k=1} \chi_{[t_{k-1}, t_k)} \phi_{k-1}
$$
is said to be a $U^p$-atom. The set of all $U^p$ atoms is denoted by $\mathcal{A}(U^p)$.   The $U^p$ space is
$$
U^p:=\left\{u= \sum^\infty_{j=1} c_j a_j : \ a_j \in \mathcal{A}(U^p), \ \ c_j \in \mathbb{C}, \ \ \sum^\infty_{j=1} |c_j|<\infty  \right\}
$$
for which the norm is given by
$$
\|u\|_{U^p}:= \inf \left\{\sum^\infty_{j=1} |c_j| : \ \ u= \sum^\infty_{j=1} c_j a_j , \ \  \ a_j \in \mathcal{A}(U^p), \ \ c_j \in \mathbb{C} \right\}.
$$
\end{defn}

\begin{defn}
Let $1\leq p <\infty$. We define $V^p$ as the normed space of all functions $v: \mathbb{R} \to L^2$ such that $\lim_{t\to \pm \infty} v(t)$ exist and for which the norm
$$
\|v\|_{V^p} := \sup_{\{t_k\}^K_{k=0} \in \mathcal{Z}} \left( \sum^K_{k=1} \|v(t_k)-v(t_{k-1})\|^p_{L^2}\right)^{1/p}
$$
is finite, where we use the convention that $v(-\infty) = \lim_{t\to \infty} v(t)$ and $v(\infty)=0$ (here $v(\infty)$ and $\lim_{t\to  \infty} v(t)$ are different notations). Likewise, we denote by $V^p_-$ the subspace of all $v\in V^p$ so that $v(-\infty) =0$. Moreover, we define the closed subspace $V^p_{rc}$ $(V^p_{-,rc})$ as all right continuous functions in $V^p$ $(V^p_-)$.
\end{defn}

\begin{defn}
We define
$$
U^p_{S} := e^{\cdot i \partial_x^4} U^p, \ \ \|u\|_{U^p_{S}} = \|e^{ -it \partial_x^4} u \|_{U^p},
$$
$$
V^p_{S} := e^{\cdot i \partial_x^4} V^p, \ \ \|u\|_{V^p_{S}} = \|e^{ -it \partial_x^4} u \|_{V^p},
$$
and similarly for the definition of $V^p_{rc, S}$, $V^p_{-, S}$, $V^p_{-, rc, S}$.
\end{defn}

\begin{defn}
 Besov type Bourgain's spaces $\dot X^{s, b, q}$ are defined by
$$
\|u\|_{\dot X^{s,b,q}} := \left\| \|\chi_{|\tau-\xi^4|\in [2^{j-1}, 2^j)} |\xi|^{s} |\tau-\xi^4|^{b} \widehat{u}(\tau,\xi) \|_{L^2_{\xi,\tau}}  \right\|_{\ell^q_{j\in \mathbb{Z}}}.
$$
\end{defn}

\begin{defn}\label{Xsq}
The frequency-uniform localized $U^2$-spaces $X^s_{q}(I)$ and $V^2$-spaces $Y^s_{q} (I)$  are defined by
\begin{align}
 \|u\|_{X^s_q(I)} & = \left(\sum_{\lambda \in I \cap \mathbb{Z}} \langle \lambda\rangle^{sq}\|\Box_\lambda u\|^q_{U^2}\right)^{1/q}, \quad  X^s_q:= X^s_q(\mathbb{R}),  \label{mod-space2}\\
 \|v\|_{Y^s_q(I)} & = \left(\sum_{\lambda \in I \cap \mathbb{Z}} \langle \lambda\rangle^{sq}\|\Box_\lambda v\|^q_{V^2}\right)^{1/q}, \quad  Y^s_q:= Y^s_q(\mathbb{R}),  \label{mod-space2a} \\
\|u\|_{X^s_{q, S}}: & = \| e^{-it \partial_x^4}  u\|_{X^s_q}, \ \ \ \|v\|_{Y^s_{q, S}}: = \| e^{-it \partial_x^4}  v\|_{Y^s_q}.
\end{align}
\end{defn}

\subsection{Known results}

\begin{prop}[Scaling properties] Let $s\geq 0$, $q\geq 2$, then
	  \begin{align}
\|f(\lambda \cdot)\|_{M^{s}_{2, q}(\mathbb{R}) }\lesssim \lambda^{-1/q'} &\|f(\cdot)\|_{M^{s}_{2, q}(\mathbb{R}) }, \qquad \qquad \rm{for} \qquad 0<\lambda\leq 1;\label{Scaling1}\\
\|f(\lambda \cdot)\|_{M^{s}_{2, q}(\mathbb{R}) }\lesssim \lambda^{s-1/2} &\|f(\cdot)\|_{M^{s}_{2, q}(\mathbb{R}) }, \qquad \qquad \rm{for} \qquad \lambda> 1. \label{Scaling2}
	  \end{align}
\end{prop}

\begin{prop}{\rm \cite{Chen23}}\label{lemchen}
	Suppose that the operator $A$ is given on the Fourier side by
	  \begin{align*}
	  \widehat{Af}(\xi )=\int m(\xi ,\eta )\hat{f}(\eta ) d\eta ,
	  \end{align*}
	then the following results hold
	  \begin{align}
	  \widehat{A^{*} f}(\xi )=\int &\overline{m(\eta  ,\xi )} \hat{f}(\eta ) d\eta,\\
	  \mathrm{tr}(A)=\frac{1}{2\pi}\int m(\xi ,\xi )d\xi ,\quad \quad &\left \| A \right \| _{\mathfrak{I}_{2}  }^{2} =\frac{1}{2\pi}\iint_{\mathbb{R}^{2} }\left | m(\xi ,\eta ) \right | ^{2}d\xi d\eta. \nonumber
	  \end{align}
Moreover, if $A_1, A_2, \cdots A_n$ are Hilbert-Schmidt operators with Fourier kernels\\ $m_1, m_2, \cdots m_n$, then
\begin{align}
{\rm tr}(A_1A_2\cdots A_n)=\frac{1}{2\pi}\int_{\mathbb{R}^n} m_1(\xi_1, \xi_2)\cdots m_n(\xi_n, \xi_1)\ d\xi_1 \cdots d\xi_n.\label{trace}
\end{align}
\end{prop}

The following known results about $U^p$ and $V^p$ can be found in \cite{Gu16,HaHeKo09,KoTa05,KoTa12}.
\begin{prop} \label{UVprop1}
{\rm (Embedding)} Let $1\leq p <q <\infty$. We have the following results.
\begin{itemize}
 \item[\rm (i)]  $U^p$ and $V^p$, $V^p_{rc}$, $V^p_{-}$, $V^p_{rc, -}$ are Banach spaces.

\item[\rm (ii)] $U^p\subset V^p_{rc, -} \subset U^q \subset L^\infty (\mathbb{R}, L^2)$. Every $u\in U^p$ is right continuous on $t\in \mathbb{R}$.

\item[\rm (iii)] $V^p \subset V^q$,   $V^p_{-} \subset V^q_{-} $,   $V^p_{rc} \subset V^q_{rc} $,  $V^p_{rc, -} \subset V^q_{rc, -} $.

\item[\rm (iv)] $\dot X^{0, 1/2, 1} \subset U^2_{S} \subset V^2_{S} \subset \dot X^{0, 1/2, \infty}$.
\end{itemize}
\end{prop}

\begin{prop} \label{UVprop2}
{\rm (Interpolation)} Let $1\leq p <q <\infty$.  There exists a positive constant $\epsilon(p,q)>0$, such that for any $u\in V^p $ and $M>1$,  there exists a decomposition $u=u_1+u_2$ satisfying
\begin{align}
 \frac{1}{M} \|u_1\|_{U^p} + e^{\epsilon M}  \|u_2\|_{U^q} \lesssim \|u\|_{V^p}. \label{interp}
\end{align}
\end{prop}

\begin{prop} \label{orthogonality}{\rm(Orthogonality in $V^2$)} Take an interval $I\subset \mathbb{R}$, then for $u\in V^2$ the following orthogonality holds:
\begin{align}
\|u_I\|_{V^2}\leq \bigg(\sum_{\lambda\in I\cap\mathbb{Z}} \| u_\lambda\|_{V^2}^2\bigg)^{1/2}.
\end{align}
\end{prop}

\begin{prop} \label{UVprop3}
{\rm (Duality)} Let $1\leq p   <\infty$, $1/p+1/p'=1$.  Then $(U^p)^* = V^{p'}$ in the sense that
\begin{align}
T: V^{p'}  \to (U^p)^*; \ \ T(v)=B(\cdot,v), \label{dual}
\end{align}
is an isometric mapping.  The bilinear form $B: U^p\times V^{p'}$ is defined as follows: For a partition $\mathrm{t}:= \{t_k\}^K_{k=0} \in \mathcal{Z}$, we define
 \begin{align}
B_{\mathrm{t}} (u,v) = \sum^K_{k=1} \langle u(t_{k-1}), \ v(t_k)-v(t_{k-1})\rangle. \label{dual2}
\end{align}
Here $\langle \cdot, \cdot \rangle$ denotes the inner product on $L^2$. For any $u\in U^p$, $v\in V^{p'}$, there exists a unique number $B(u,v)$ satisfying the following property. For any $\varepsilon>0$, there exists a partition $\mathrm{t}$ such that
$$
|B(u,v)- B_{\mathrm{t}'} (u,v)| <\varepsilon, \ \ \forall \  \mathrm{t}'    \supset \mathrm{t}.
$$
Moreover,
$$
|B(u,v)| \leq \|u\|_{U^p} \|v\|_{V^{p'}}.
$$
In particular, let $u\in V^1_{-}$ be absolutely continuous on compact interval, then for any $v\in V^{p'}$,
$$
 B(u,v) =\int \langle u'(t), v(t)\rangle dt.
$$
\end{prop}

\begin{prop}{\rm \cite{GRW21}} \label{UVprop4} {\rm (Duality)}
 Let $1\leq q   <\infty$.  Then $(X^s_q)^* = Y^{-s}_{q'} $ in the sense that
\begin{align}
T: Y^{-s}_{q'}   \to (X^s_q)^* ; \ \ T(v)=B(\cdot,v), \label{dualprop4}
\end{align}
is an isometric mapping, where the bilinear form $B(\cdot,\cdot)$ is defined in  Proposition \ref{UVprop3}. Moreover, we have
$$
|B(u,v)| \leq   \|u\|_{X^s_q} \|v\|_{Y^{-s}_{q'}}.
$$
\end{prop}

\section{Basic Estimates}

%It is known that the dispersive modulation of the 4NLS equation \eqref{4NLS} is $|\tau-\xi^4|$.  By the last inclusion of (iv) in Proposition \ref{UVprop1}, we see that

%\begin{lem}[\rm Dispersion Modulation Decay]
%Suppose that the dispersion modulation $|\tau-\xi^4| \gtrsim \mu$ for a function $u\in L^2_{x,t}$, then we have
%\begin{align}
%  \|u \|_{L^2_{x,t} }  \lesssim \mu^{-1/2} \|u\|_{V^2_S}. \label{dispersiondecay}
%\end{align}
%\end{lem}

\begin{lem}{\rm \cite{KPV91}(Strichartz-type Estimates)} \label{strichartz}
Let $2\leq p,q\leq\infty$, a pair $(p, q)$ is called biharmonic admissible or Strichartz admissible if
   \[
\frac{4}{p}+\frac{1}{q}=\frac{1}{2}\ \ \    \hbox{or}\ \ \    \frac{2}{p}+\frac{1}{q}=\frac{1}{2}.
    \]
\begin{itemize}
\item[\rm (i)] Let $(p,q)$ be biharmonic admissible, then
\begin{align}\label{biharmonic}
\|e^{it\partial^4_x}\phi\|_{L^p_tL^q_x}\lesssim \|\phi\|_{L^2}.
\end{align}
\item[\rm (ii)]Let $(p,q)$ be Strichartz admissible, then
\begin{align}\label{strichartza}
\|D_x^{2/p}e^{it\partial^4_x}\phi\|_{L^p_tL^q_x}\lesssim \|\phi\|_{L^2}.
\end{align}
In particular, for $N\geq1$,
\begin{align}\label{strichartzb}
 \|P_{N}e^{it\partial^4_x}\phi\|_{L^p_tL^q_x}\lesssim \langle N\rangle^{-2/p} \|\phi\|_{L^2}.
\end{align}
By testing atoms in $U^p_S$ space, we obtain
\begin{align}\label{strichartzc}
 \|P_{N}u\|_{L^p_tL^q_x}\lesssim \langle N\rangle^{-2/p} \|u\|_{U^p_S}.
\end{align}
\end{itemize}
\end{lem}

\begin{lem}[\rm Bilinear Estimates]
Suppose that $\widehat{u_0}, \widehat{v_0}$ are localized in some compact intervals $I_1,I_2$ with $dist(I_1, I_2)\gtrsim \lambda>0$. For any $\xi_1\in\text{supp}\ \widehat{u_0},\ \xi_2\in\text{supp}\ \widehat{v_0}$,
\begin{itemize}
\item[\rm (i)] if $2|\xi_1|\leq|\xi_2|\sim N$, then
\begin{align}
 \|e^{it\partial_x^4}u_0e^{it\partial_x^4} v_0 \|_{L^2_{x, t}} \lesssim  N^{-3/2} \|u_0\|_{L^2} \|v_0\|_{L^2},\label{0bilinear1}\\
 \|e^{it\partial_x^4}u_0\overline{e^{it\partial_x^4} v_0} \|_{L^2_{x, t}} \lesssim  N^{-3/2} \|u_0\|_{L^2} \|v_0\|_{L^2}. \label{0bilinear2}
\end{align}
\item[\rm (ii)] if $|\xi_1|\sim N_1$, $|\xi_2|\sim N_2$, then
\begin{align}
  \|e^{it\partial_x^4}u_0e^{it\partial_x^4} v_0 \|_{L^2_{x, t}} \lesssim (\lambda \max\{N_1^2,N_2^2\})^{-1/2}\|u_0\|_{L^2} \|v_0\|_{L^2},\label{0bilinear3}\\
 \|e^{it\partial_x^4}u_0\overline{e^{it\partial_x^4} v_0} \|_{L^2_{x, t}} \lesssim (\lambda \max\{N_1^2,N_2^2\})^{-1/2}\|u_0\|_{L^2} \|v_0\|_{L^2}.\label{0bilinear4}
\end{align}
\end{itemize}
\end{lem}
\begin{proof}
For \eqref{0bilinear3}, by duality, it suffices to prove that
\begin{align}\label{bidual}
     \sup_{\|g\|_{L_{x,t}^2}=1}  \left | \int_{\mathbb{R}^{2}} e^{it\partial_x^4}u_0e^{it\partial_x^4} v_0 \cdot \overline{g}\ dxdt \right | \lesssim (\lambda \max\{N_1^2,N_2^2\})^{-1/2}\|u_0\|_{L^2} \|v_0\|_{L^2}.
\end{align}
From the Plancherel theorem and the definition of Fourier transform, we know that
\begin{align}
\text{LHS}\eqref{bidual}&=  \sup_{\|g\|_{L_{x,t}^2}=1} \left |  \int_{\mathbb{R}^{2}} \int e^{it((\xi-\xi_1)^4+\xi_1^4)} \widehat{u}_0(\xi_1) \widehat{v}_0(\xi-\xi_1) d\xi_1 \cdot\overline{\widehat{g}}(\xi,t)d\xi dt\right |\nonumber\\
&=\sup_{\|g\|_{L_{x,t}^2}=1}\left | \int_{\mathbb{R}^{2}} \overline{(\mathscr{F}_{x,t}{g})(\xi_1+\xi_2,\xi_1^4+\xi_2^4)}\widehat{u}_0(\xi_1) \widehat{v}_0(\xi_2) d\xi_1 d\xi_2\right |
 \label{bilinearest1}
\end{align}
Let $\xi=\xi_1+\xi_2$ and $\tau=\xi_1^4+\xi_2^4$. A simple calculation shows that $d\xi d\tau= |J|d\xi_1 d\xi_2$,$|J|=4|\xi_1^3-\xi_2^3|\gtrsim \lambda\cdot \max\{N_1^2,N_2^2\} $.  Then by using H\"older's inequality and the change of variables, we get
\begin{align}
\text{LHS}\eqref{bidual}&\lesssim \bigg(\int_{\mathbb{R}^{2}}\big|\widehat{u}_0\big(\xi_1(\xi,\tau)\big) \widehat{v}_0\big(\xi_2(\xi,\tau)\big)\big|^2 |J|^{-2}d\xi d\tau\bigg)^{1/2}\nonumber\\
&\lesssim \bigg(\int_{\mathbb{R}^{2}}\big|\widehat{u}_0\big(\xi_1\big) \widehat{v}_0\big(\xi_2\big)\big|^2 |J|^{-1}d\xi_1 d\xi_2\bigg)^{1/2}\nonumber\\
&\lesssim (\lambda\cdot \max\{N_1^2,N_2^2\})^{-1/2}\|u_0\|_{L^2} \|v_0\|_{L^2}.
\end{align}
A similar argument gives \eqref{0bilinear1} as long as $|J|$ is replaced by $4|\xi_1^3-\xi_2^3|\sim N^3$. \eqref{0bilinear2} and \eqref{0bilinear4} can be derived through analogous discussions.
\end{proof}
Next by testing atoms in the $U^2$ space and applying the interpolation in Proposition \ref{UVprop2}, we get the bilinear estimates as follows.
\begin{cor}[\rm Bilinear Estimates] \label{V2decay2}
 Let $0<T\leq 1$. Suppose that $\widehat{u}, \widehat{v}$ are localized in some compact intervals $I_1,I_2$ with $dist(I_1, I_2)\gtrsim \lambda>0$. For any $0<\varepsilon \ll 1$, $\xi_1\in\text{supp}\ \widehat{u},\ \xi_2\in\text{supp}\ \widehat{v}$,
\begin{itemize}
\item[\rm (i)] if $2|\xi_1|\leq|\xi_2|\sim N$, then
\begin{align}
 \|u v \|_{L^2_{x, t\in[0,T]}} \lesssim T^{\varepsilon/4} N^{-3/2 + \varepsilon} \|u\|_{V^2_S} \|v\|_{V^2_S},\label{bilinear1}\\
 \|u \bar{v} \|_{L^2_{x, t\in[0,T]}} \lesssim T^{\varepsilon/4} N^{-3/2 + \varepsilon} \|u\|_{V^2_S} \|v\|_{V^2_S}. \label{bilinear2}
\end{align}
\item[\rm (ii)] if $|\xi_1|\sim N_1$, $|\xi_2|\sim N_2$, then
\begin{align}
 \|u v \|_{L^2_{x, t\in[0,T]}} \lesssim T^{\varepsilon/4} (\lambda\cdot \max\{N_1^2,N_2^2\})^{-1/2 + \varepsilon} \|u\|_{V^2_S} \|v\|_{V^2_S},\label{bilinear3}\\
  \|u \bar{v} \|_{L^2_{x, t\in[0,T]}} \lesssim T^{\varepsilon/4} (\lambda\cdot \max\{N_1^2,N_2^2\})^{-1/2 + \varepsilon} \|u\|_{V^2_S} \|v\|_{V^2_S}.\label{bilinear4}
\end{align}
\end{itemize}
\end{cor}

\begin{lem}\label{L4} {\rm ($L^4$ Estimates)}
Let $I\subset [0, +\infty)$ or $(-\infty, 0]$ with $|I|<\infty$.  For  any  $\theta\in (0,1)$, $\beta>0$, we have
\begin{align}
 \|u_I\|^2_{L^4_{x,t\in [0,T]}} \lesssim (T^{1/4}+ T^{(1-\theta)/4}|I|^{2\beta+(1-\theta)/2 }) \|u\|^2_{X^{-1/4}_{4,S}(I)}. \label{lebesgue4}
\end{align}
In particular, if $1 \lesssim  |I| <\infty$, $0<T<1$, then for any $0< \varepsilon \ll 1$, $4\leq q\leq\infty$
\begin{align}
 \|u_I\|_{L^4_{x,t\in [0,T]}} \lesssim   T^{\varepsilon/4}|I|^{1/4-1/q + \varepsilon} \max_{\lambda \in I} \langle \lambda\rangle^{-3/4} \|u\|_{X^{1/2}_{q,S}(I)}.  \label{lebesgue4a}
\end{align}
\end{lem}

\begin{proof}
Without loss of generality, we assume $I\subset [0, +\infty)$.
\begin{align*}
\| u_I \|_{L^4([0,T]\times \mathbb{R})}^2 &= \| (u_I)^2 \|_{L^2([0,T]\times \mathbb{R})}=\bigg\| \sum_{m, n\in I\cap \mathbb{Z}} u_m u_n\bigg\|_{L^2([0,T]\times \mathbb{R})} \\
	&\leq \sum_{k\in \mathbb{N}}\bigg\| \sum_{m-n\sim2^k}u_m u_n\bigg\|_{L^2([0,T]\times \mathbb{R})}.
\end{align*}
When $k=0$, we utilize the orthogonality in $L^2$ and Strichartz estimates in $L_t^8L_x^4$ to derive the conclusions.  When $k>0$, by applying the interpolation in Proposition \ref{UVprop2}, bilinear estimates in Lemma \ref{V2decay2}, and Strichartz estimates, we can complete the proof. The detailed proof follows a similar approach to the $L^4$ estimates in \cite{Chen20,Gu16}, and is therefore omitted here.
\end{proof}

\begin{lem}\label{V2toX}
Let $I\subset \mathbb{R}$ with $1\lesssim  |I|<\infty$, $2\leq q \leq\infty$, we have
\begin{align}
\|u_I\|_{L^\infty_t L^2_{x} \cap V^2_S} & \lesssim    |I|^{1/2-1/q} \max_{\lambda \in I} \langle \lambda\rangle^{-1/2} \|u\|_{X^{1/2}_{q,S}(I)}.  \label{V2}
\end{align}
\end{lem}

\begin{proof}
By employing the embedding $V^2_S \subset L^\infty_t L^2_x$, the orthogonality in $V^2$ and H\"older's inequality, we can derive the following results step by step:
\begin{align*}
\|u_I\|_{L^\infty_t L^2_{x} \cap V^2_S} &\lesssim  \bigg(\sum_{\lambda\in I}\|u_\lambda\|^2_{V^2_S}\bigg)^{1/2}
  \lesssim\max_{\lambda\in I}\langle\lambda\rangle^{-1/2}\bigg(\sum_{\lambda\in I}\big(\langle\lambda\rangle^{1/2}\|u_\lambda\|\big)^2_{V^2_S}\bigg)^{1/2}\nonumber\\
  &\lesssim|I|^{1/2-1/q} \max_{\lambda \in I} \langle \lambda\rangle^{-1/2} \|u\|_{X^{1/2}_{q,S}(I)}.
\end{align*}
\end{proof}

%%%%%%%%%%%%%%%%%%%%%%%%%%%%%%%%%%%%%%%%%%%%%%%%%%%%%%%%%%%%%%%%%%%%%%%%%%%%%%%%%%%%%%%%%%%%%%%%%%%%%%%%%%%%%%%%%%%%%%%%%%%%%%

\section{Local well-posedness}

 It is known that 4NLS \eqref{4NLS} is equivalent to the following integral equation:
\begin{align}
u(x,t)=e^{it\partial^{4}_ x }u_0+i\int_{0}^{t}e^{i(t-\tau )\partial^{4} _x }F(u(\tau ))d\tau \label{integralEQ},
\end{align}
where
\begin{align*}
F(u)=8u_{xx}\left | u \right | ^2+2 \bar{u} _{xx}u^2+6u_x^2\bar{u}+4u\left | u_x \right | ^2+6u\left | u \right | ^4.
\end{align*}
\subsection{Trilinear estimates}

\begin{lem}
    Let $q\geq 2$ and $0<\varepsilon\ll 1$, for any $u\in X_{q,S}^{\frac{1}{2}}$, we have
\begin{align}
&\left \| \int_{0}^{t}e^{i(t-\tau )\partial _{x}^{4} }\left ( u_{xx}\left | u \right |^{2}  \right )(\tau )d\tau \right \|_{X_{q,S}^{\frac{1}{2}}}
\lesssim T^{\varepsilon }\left \| u \right \|_{X_{q,S}^{\frac{1}{2}}}^{3}; \label{uxxuudual}\\
&\left \| \int_{0}^{t}e^{i(t-\tau )\partial _{x}^{4} }\left ( \overline{u}_{xx}u^{2}  \right )(\tau )d\tau \right \|_{X_{q,S}^{\frac{1}{2}}}
\lesssim T^{\varepsilon }\left \| u \right \|_{X_{q,S}^{\frac{1}{2}}}^{3}; \label{uxxdualuu}\\
&\left \| \int_{0}^{t}e^{i(t-\tau )\partial _{x}^{4} }\left ( u_{x}^{2}\overline{u}   \right )(\tau )d\tau \right \|_{X_{q,S}^{\frac{1}{2}}}
\lesssim T^{\varepsilon }\left \| u \right \|_{X_{q,S}^{\frac{1}{2}}}^{3};\label{uxuxudual}\\
&\left \| \int_{0}^{t}e^{i(t-\tau )\partial _{x}^{4} }\left ( u\left | u_{x} \right |^{2}    \right )(\tau )d\tau \right \|_{X_{q,S}^{\frac{1}{2}}}
\lesssim T^{\varepsilon }\left \| u \right \|_{X_{q,S}^{\frac{1}{2}}}^{3}.\label{uuxuxdual}
\end{align}
\end{lem}

\begin{proof}
 The most crucial and challenging estimate is \eqref{uxxuudual}. The estimates of \eqref{uxxdualuu}, \eqref{uxuxudual} and \eqref{uuxuxdual} can be deduced in a similar way as \eqref{uxxuudual}. Therefore, for the sake of brevity, we only prove \eqref{uxxuudual}.

 By Propositions \ref{UVprop3} and \ref{UVprop4}, we see that, for supp $v\subset \mathbb{R} \times [0,T]$,
\begin{align}
\left \| \int_{0}^{t}e^{i(t-\tau )\partial^{4} _x }F(u(\tau ))d\tau \right \|_{X_{q,S}^{\frac{1}{2} } }
&=\mathrm{sup} \left \{ \left | B\left ( \int_{0}^{t}e^{-i\tau \partial^{4} _x }F(u(\tau ))d\tau ,v \right )  \right |:\left \|  v \right \| _{Y_{q'}^{-1/2 } }\leqslant 1   \right \}  \nonumber\\
&\leqslant  \underset{\left \|  v \right \| _{Y_{q'}^{-1/2} }\leqslant 1}{\mathrm{sup}}\left | \int_{[0,T]}\left \langle F(u(\tau )), e^{i\tau \partial^{4} _x }v(\tau )\right \rangle d\tau  \right |  \nonumber\\
&\leqslant  \underset{\left \|  v \right \| _{Y_{q',S}^{-1/2} }\leqslant 1}{\mathrm{sup}}\left | \int_{[0,T]}\left \langle F(u(\tau )), v(\tau )\right \rangle d\tau  \right | \label{vie1}.
\end{align}
Thus it suffices to show that
\begin{align}
\left | \int_{\mathbb{R}\times \left [ 0,T \right ]  }\overline{v}(\partial _{x}^{2}u)u\overline{u}  dxdt  \right |
\lesssim T^{\varepsilon }\left \| u \right \|_{X_{q,S}^{\frac{1}{2}}}^{3}\left \| v \right \|_{Y_{q',S}^{-\frac{1}{2}}}\label{2}
\end{align}
We perform a frequency-uniform decomposition with $u,v$ in the left-hand side of (\ref{2}), it suffices to prove that
\begin{align}
\left | \sum_{\lambda _{0},\dots,\lambda _{3}  } \left \langle \lambda _{0} \right \rangle^{1/2}  \int_{\mathbb{R}\times \left [ 0,T \right ]  }\overline{v}_{\lambda _{0}} (\partial _{x}^{2}u_{\lambda _{1}})u_{\lambda _{2}}\overline{u}_{\lambda _{3}}  dxdt  \right |
\lesssim T^{\varepsilon }\left \| u \right \|_{X_{q,S}^{1/2}}^{3}\left \| v \right \|_{Y_{q',S}^{0}}.\label{4}
\end{align}

In order to keep the left-hand side of (\ref{4}) nonzero, we have the frequency constraint condition (FCC)
\begin{align}
\lambda _{1}+\lambda _{2}-\lambda _{3}-\lambda _{0}\approx 0. \label{FCC}
\end{align}
We only need to consider the cases that $\lambda _{0}$ is minimal or secondly minimal number in $\lambda _{0},\dots ,\lambda _{3}$ (Otherwise one can replace $\lambda _{0},\dots ,\lambda _{3}$ with $-\lambda _{0},\dots ,-\lambda _{3}$).

$\mathbf{Step\:1.}$ $\lambda _{0}=\underset{0\leqslant k\leqslant 3}{\mathrm{min}  }\lambda _{k}$. We separate the proof into two subcases $\lambda _{0}\geqslant  0$ and $\lambda _{0}<0$.

$\mathbf{Step\:1.1.}$ $\lambda _{0}\geqslant  0$. Let us denote $I_0=[0,1),\;I_j=[2^{j-1},2^j),\;j\geqslant 1.$ We decompose $\lambda _{0}+[0,\infty )$ by dyadic decomposition, $\mathrm{i.e.}$,
$$\lambda _{k}\in \lambda _{0}+[0,\infty )=\bigcup_{j_{k}\geqslant 0 }(\lambda _{0}+I_{j_{k} } ),\quad  k=1,2,3. $$
By FCC (\ref{FCC}), we see that $j_{3}\approx (j_1\vee j_2 )$. Then $\lambda _{0},\dots ,\lambda _{3}$ have the following two orders:
\begin{align*}
\mathrm{Order\;1:}& \quad \lambda _{0}\leqslant \lambda _{2} \leqslant \lambda _{1}\leqslant \lambda _{3}\\
\mathrm{Order\;2:}& \quad \lambda _{0}\leqslant \lambda _{1} \leqslant \lambda _{2}\leqslant \lambda _{3}
\end{align*}
We will take $\mathrm{Order\;1}$ for example because $\mathrm{Order\;2}$ is similar and even easier (notice that the derivatives are located in $u_{\lambda _{1}}$).

$\mathbf{Order\;1:}  \quad \lambda _{0}\leqslant \lambda _{2} \leqslant \lambda _{1}\leqslant \lambda _{3}$. In view of FCC (\ref{FCC}), we see that $j_2 \leqslant j_1 \approx j_3$. It turns out that
\begin{align*}
\mathscr{L}^+(u,v):=&\sum_{0\leqslant \lambda _{0}\leqslant \lambda _{2} \leqslant \lambda _{1}\leqslant \lambda _{3} } \left \langle \lambda _{0} \right \rangle^{1/2}  \int_{\mathbb{R}\times \left [ 0,T \right ]  }\overline{v}_{\lambda _{0}} (\partial _{x}^{2}u_{\lambda _{1}})u_{\lambda _{2}}\overline{u}_{\lambda _{3}}  dxdt \nonumber\\
\lesssim & \left ( \sum_{\lambda _{0}\geq0,\: 0\leqslant j_2\leqslant j_1\approx j_3\lesssim 1}+\sum_{\lambda _{0},j_2\geqslant 0,\:j_1\approx j_3\gg 1} \right ) \nonumber\\
&\quad\quad\left \langle \lambda _{0} \right \rangle^{1/2}  \int_{\mathbb{R}\times \left [ 0,T \right ]  }\overline{v}_{\lambda _{0}} (\partial _{x}^{2}u_{\lambda _{0}+I_{j_1}})u_{\lambda _{0}+I_{j_2}}\overline{u}_{\lambda _{0}+I_{j_3}}  dxdt\nonumber\\
:=& \mathscr{L}^+_l(u,v)+\mathscr{L}^+_h(u,v).
\end{align*}

In $\mathscr{L}^+_{l}(u,v)$, we have $\lambda_0 \thickapprox \lambda_1 \thickapprox \lambda_2  \thickapprox \lambda_3$. So, by H\"older's inequality, Strichartz estimate and $U^2_S \subset V^2_S \subset U^8_S$,  we have
 \begin{align}\label{Ll}
|\mathscr{L}^+_{l}(u,v)|
 & \lesssim T^{1/2}\sum_{0\leq\lambda_0 \thickapprox \lambda_1 \thickapprox \lambda_2  \thickapprox \lambda_3} \left\langle\lambda_0\right\rangle^{5/2} \|v_{\lambda_0}\|_{L^{8}_{t}L^{4}_{x}} \|u_{\lambda_1}\|_{L^{8}_{t}L^{4}_{x}} \|u_{\lambda_2}\|_{L^{8}_{t}L^{4}_{x}} \|u_{\lambda_3}\|_{L^{8}_{t}L^{4}_{x}} \nonumber \\
%&\lesssim T^{1/2} \sum_{\lambda _{0}\geq0} \left\langle\lambda_0\right\rangle^{5/2}\|v_{\lambda_0}\|_{L^{8}_{t}L^{4}_{x}}\|u_{\lambda_0}\|_{L^{8}_{t}L^{4}_{x}}^3 \nonumber \\
%&\lesssim T^{1/2} \sum_{\lambda _{0}\geq0} \left\langle\lambda_0\right\rangle^{3/2}\|v_{\lambda_0}\|_{U^8_S}\|u_{\lambda_0}\|_{U^8_S}^3 \nonumber \\
&\lesssim T^{1/2} \sum_{\lambda _{0}\geq0}\|v_{\lambda_0}\|_{V^2_S}\cdot \left ( \left\langle\lambda_0\right\rangle^{1/2}\|u_{\lambda_0}\|_{U^2_S} \right )^3 \lesssim T^{1/2}\|v\|_{Y^0_{q', S}}  \|u\|^{3}_{X^{1/2}_{q, S}}.
\end{align}

In $\mathscr{L}^+_{h} (u,v)$,  notice that the frequency of $v_{\lambda_0}$ and $u_{\lambda_0+I_{j_1}}$ have transversality. Thus we can use bilinear estimate (Lemma \ref{V2decay2}), $L^4$ estimate (Lemma \ref{L4}) and $V^2_S$ estimate (Lemma \ref{V2toX}) to obtain that
\begin{align}
|\mathscr{L}^+_{h} (u,v)|
& \lesssim  \sum_{\lambda _{0},\:j_2\geqslant 0,j_1\approx j_3\gg 1}  \langle\lambda_0 \rangle^{1/2}\|{\overline{v}_{\lambda_0}} {\partial_x^2 u_{\lambda_0+I_{j_1}}}\|_{L^2_{x,t}}
\|{u_{\lambda_0+I_{j_2}}} \|_{L^4_{x,t}}\|{\overline{u}_{\lambda_0+I_{j_3}}} \|_{L^4_{x,t}}   \notag\\
& \lesssim \sum_{\lambda _{0},\:j_2\geqslant 0,\:j_1\approx j_3\gg 1}  \langle\lambda_0 \rangle^{1/2}\cdot\langle \lambda_0+ 2^{j_1}\rangle^2  T^{\varepsilon/4}  \left [ 2^{j_1}\langle \lambda_0+ 2^{j_1}\rangle^2 \right ]^{-1/2+\varepsilon}  \|{v_{\lambda_0}} \|_{V^2_S}\|{u_{\lambda_0+I_{j_1}}}\|_{V^2_S}\nonumber\\
&\qquad\qquad\cdot T^{\varepsilon/2}2^{(j_2+j_3)(1/4-1/q+\varepsilon)}\langle \lambda_0+ 2^{j_2}\rangle^{-3/4}\langle \lambda_0+ 2^{j_3}\rangle^{-3/4} \| u  \|^2_{X^{1/2}_{q,S}} \nonumber\\
%&\lesssim  T^{3\varepsilon/4}\sum_{\lambda _{0}, \:j_2\geqslant 0,\:j_1\gg 1}\langle\lambda_0 \rangle^{1/2}\langle \lambda_0+ 2^{j_1}\rangle^{1/4+2\varepsilon}\langle \lambda_0+ 2^{j_2}\rangle^{-3/4}2^{j_1(-1/4-1/q+2\varepsilon)+j_2(1/4-1/q+\varepsilon)}\nonumber\\
%&\qquad\qquad\cdot 2^{j_1(1/2-1/q)}\langle \lambda_0+ 2^{j_1}\rangle^{-1/2} \|{v_{\lambda_0}} \|_{V^2_S}\| u  \|_{X^{1/2}_{q,S}(\lambda_0+I_{j_1})}\| u  \|^2_{X^{1/2}_{q,S}}  \nonumber\\
&\lesssim  T^{3\varepsilon/4}\sum_{\lambda _{0}, \:j_2\geqslant 0,\:j_1\gg 1}\langle\lambda_0 \rangle^{1/2}\langle \lambda_0+ 2^{j_1}\rangle^{-1/4+2\varepsilon}\langle \lambda_0+ 2^{j_2}\rangle^{-3/4}\nonumber\\
&\qquad\qquad\qquad\qquad2^{j_1(1/4-2/q+2\varepsilon)+j_2(1/4-1/q+\varepsilon)}\cdot \|{v_{\lambda_0}}  \|_{V^2_S}\| u  \|_{X^{1/2}_{q,S}(\lambda_0+I_{j_1})}\| u  \|^2_{X^{1/2}_{q,S}}  \nonumber\\
&\lesssim  T^{3\varepsilon/4}\sum_{\lambda _{0}, \:j_2\geqslant 0,\:j_1\gg 1}2^{j_1(-2/q+4\varepsilon)+j_2(-1/q+\varepsilon)}\|{v_{\lambda_0}} \|_{V^2_S}\| u  \|_{X^{1/2}_{q,S}(\lambda_0+I_{j_1})}\| u  \|^2_{X^{1/2}_{q,S}}\nonumber.
\end{align}
Making the summation on $j_2$, then applying H\"older's inequality on $\lambda_0$, and finally making the summation on $j_1$, we know for $0<\varepsilon <1/2q$,
\begin{align*}
|\mathscr{L}^+_{h} (u,v)|\lesssim  T^{3\varepsilon/4}\sum_{j_1\gg 1}2^{j_1(-2/q+4\varepsilon)}\|{v} \|_{Y^0_{q',S}}\| u  \|^3_{X^{1/2}_{q,S}}\lesssim  T^{3\varepsilon/4}\|v \|_{Y^0_{q',S}}\| u  \|^3_{X^{1/2}_{q,S}}.
\end{align*}

$\mathbf{Step\:1.2.}$ $\lambda _{0}<0$. We can assume that $\lambda _{0}\ll 0$. For short, we use the following notations:
$$
\left\{
\begin{array}{l}
\lambda_k \in h_- \Leftrightarrow \lambda_k\in [\lambda_0, 3\lambda_0/4]; \\
\lambda_k \in h_+ \Leftrightarrow \lambda_k\in [-3\lambda_0/4, -\lambda_0];\\
\lambda_k \in l_- \Leftrightarrow \lambda_k\in [3\lambda_0/4, 0];\\
\lambda_k \in l_+ \Leftrightarrow \lambda_k\in [0,- 3 \lambda_0/4].
\end{array}
\right.
$$
We need to consider the following three cases $\lambda_1 \in h_-$, $\lambda_1\in l_-$ and $\lambda_1\in [0, \infty)$ separately.

{\bf Case A.} $\lambda_1\in h_-$.  We consider the following four subcases as in Table \ref{caseA}.
\begin{table}[h]
\begin{center}

\begin{tabular}{|c|c|c|c|}
\hline
 ${\rm Case\ A}$   &  $\lambda_1\in $ & $\lambda_2\in $ & $\lambda_3\in $    \\
\hline
$h_-h_-h_-$  & $ [\lambda_0, 3\lambda_0/4)$  & $ [\lambda_0, 3\lambda_0/4)$  &  $[\lambda_0,\lambda_0/2+C )$  \\
\hline
$h_-l_-l_{h-}$ & $ [\lambda_0, 3\lambda_0/4)$  & $ [3\lambda_0/4, \ 0)$ & $ [3\lambda_0/4-C, \lambda_0/4)$  \\
\hline
$h_-l_- l_{l\pm} $ & $ [\lambda_0, 3\lambda_0/4)$  & $ [3\lambda_0/4,0)$  & $ [\lambda_0/4,-\lambda_0/4+C)$\\
\hline
$h_-a_+a_+$ & $ [\lambda_0, 3\lambda_0/4)$  & $ [0, \infty)$  & $ [-C,\infty)$\\
\hline
\end{tabular}
\end{center}
\caption{$\lambda _{0}=\underset{0\leqslant k\leqslant 3}{\mathrm{min}  }\lambda _{k},\ \ \lambda_0\ll 0,\ \ \lambda_1\in h_-$.}\label{caseA}
\end{table}

{\it Case $h_-h_-h_-$.} We decompose $\lambda _k$ in the following way:
$$
\lambda_k\in \lambda_0 +[0, - \frac{\lambda_0}{4} ) = \bigcup_{j_k\geq 0} (\lambda_0 + I_{j_k}),\  k=1,2;\ \ \lambda_3 \in \lambda_0+ [0,-\frac{\lambda_0}{2} +C) = \bigcup_{j_3\geqslant 0} (\lambda_0+ I_{j_3}),
$$
where we can assume that dyadic interval $I_{j_1}$ satisfies $I_{j_1} = [2^{j-1}, 2^j) \cap [0, -\lambda_0/4)$. In the following, we will all explain it in this way. By FCC \eqref{FCC}, we see that $j_3\approx (j_1\vee j_2)$ and $0\leqslant j_k\lesssim \mathrm{ln} \left \langle \lambda _{0}  \right \rangle,\:k=1,2,3 $, we have
\begin{align*}
\mathscr{L}^-_{h_-h_-h_-} (u,v):=&\sum_{\lambda _{0}\ll 0,\:0\leqslant j_3\approx (j_1\vee j_2)\lesssim \mathrm{ln} \left \langle \lambda _{0}  \right \rangle } \left \langle \lambda _{0} \right \rangle^{1/2}  \int_{\mathbb{R}\times \left [ 0,T \right ]  }\overline{v}_{\lambda _{0}} (\partial _{x}^{2}u_{\lambda _{0}+I_{j_1}})u_{\lambda _{0}+I_{j_2}}\overline{u}_{\lambda _{0}+I_{j_3}}  dxdt\nonumber\\
\lesssim & \left ( \sum_{\lambda _{0}\ll 0,\:0\leqslant j_3\approx (j_1\vee j_2)\lesssim 1}+\sum_{\lambda _{0}\ll 0,\:1\ll j_3\approx (j_1\vee j_2)\lesssim \mathrm{ln} \left \langle \lambda _{0}  \right \rangle } \right ) \nonumber\\
&\quad\quad\left \langle \lambda _{0} \right \rangle^{1/2}  \int_{\mathbb{R}\times \left [ 0,T \right ]  }\overline{v}_{\lambda _{0}} (\partial _{x}^{2}u_{\lambda _{0}+I_{j_1}})u_{\lambda _{0}+I_{j_2}}\overline{u}_{\lambda _{0}+I_{j_3}}  dxdt\nonumber\\
:= &\mathscr{L}^{-,l}_{h_-h_-h_-} (u,v)+\mathscr{L}^{-,h}_{h_-h_-h_-} (u,v).
\end{align*}

In $\mathscr{L}^{-,l}_{h_-h_-h_-} (u,v)$, we have $\lambda _0\approx \lambda _1\approx \lambda _2\approx \lambda _3$. Using the same method as in  $\mathscr{L}^+_{l}(u,v)$ \eqref{Ll}, we can get the result and the details are omitted.

In $\mathscr{L}^{-,h}_{h_-h_-h_-} (u,v)$, we have $\langle\lambda_0 \rangle\sim \langle\lambda_1 \rangle\sim\langle\lambda_2 \rangle\sim\langle\lambda_3 \rangle$. By applying bilinear estimate \eqref{bilinear3}, $L^4$ estimate \eqref{lebesgue4a}, H\"older's inequality and Lemma \ref{V2toX}, we obtain that
\begin{align}
&|\mathscr{L}^{-,h}_{h_-h_-h_-} (u,v)|\lesssim \sum_{1\ll j_3\approx (j_1\vee j_2)\lesssim \mathrm{ln} \left \langle \lambda _{0} \right \rangle } \langle\lambda_0 \rangle^{1/2}\|{\overline{v}_{\lambda_0}\overline{u}_{\lambda_0+I_{j_3}}} \|_{L^2_{x,t}}\|{\partial_x^2 u_{\lambda_0+I_{j_1}}} \|_{L^4_{x,t}}\|{u_{\lambda_0+I_{j_2}}} \|_{L^4_{x,t}} \nonumber\\
\lesssim& \sum_{1\ll j_3\approx (j_1\vee j_2)\lesssim \mathrm{ln} \left \langle \lambda _{0} \right \rangle } \langle\lambda_0 \rangle^{1/2}\langle\lambda_0 \rangle^{2}\cdot T^{\varepsilon/4}\left [ 2^{j_3}\langle\lambda_0 \rangle^{2} \right ]^{-1/2+\varepsilon } \|{v_{\lambda_0}} \|_{V^2_S}\|{u_{\lambda_0+I_{j_3}}}\|_{V^2_S}\nonumber\\
&\qquad\qquad\qquad\qquad \cdot T^{\varepsilon/2}2^{(j_1+j_2)(1/4-1/q+\varepsilon)}\langle\lambda_0\rangle^{-3/2}\| u \|^2_{X^{1/2}_{q,S}}\nonumber\\
\lesssim&T^{3\varepsilon/4}\sum_{1\ll j_3\approx (j_1\vee j_2)\lesssim \mathrm{ln} \left \langle \lambda _{0} \right \rangle } \langle\lambda_0 \rangle^{2\varepsilon}2^{(j_1+j_2)(1/4-1/q+\varepsilon)+j_3(-1/2+\varepsilon)}\nonumber\\
&\qquad\qquad\qquad\qquad\cdot 2^{j_3(1/2-1/q)}\langle \lambda_0\rangle^{-1/2}\|{v_{\lambda_0}} \|_{V^2_S}\| u \|^3_{X^{1/2}_{q,S}}\nonumber\\
\lesssim&T^{3\varepsilon/4}\sum_{1\ll j_3\approx (j_1\vee j_2)\lesssim \mathrm{ln} \left \langle \lambda _{0} \right \rangle } \langle\lambda_0 \rangle^{-1/2+2\varepsilon }2^{(j_1+j_2)(1/4-1/q+\varepsilon)+j_3(-1/q+\varepsilon)}\|{v_{\lambda_0}}  \|_{V^2_S}\| u \|^3_{X^{1/2}_{q,S}}\nonumber\\
\lesssim&T^{3\varepsilon/4}\sum_{\lambda _{0}\ll 0} \langle\lambda_0 \rangle^{-2/q+4\varepsilon }\|{v_{\lambda_0}} \|_{V^2_S}\| u \|^3_{X^{1/2}_{q,S}}\nonumber\\
\lesssim&T^{3\varepsilon/4}\|{v_{\lambda_0}}\|_{Y^0_{q',S}}\| u \|^3_{X^{1/2}_{q,S}}.\label{pm2}
\end{align}
Here we first make the summation on $j_1,j_2,j_3$ (as long as $0<\varepsilon <1/q$), then apply H\"older's inequality on $\lambda_0$.

{\it Case $h_-l_-l_{h-}$.} By FCC \eqref{FCC}, we see that $\lambda_2 \in [ 3\lambda_0/4, \lambda_0/4+C)$. Then we have $ \langle\lambda_0\rangle\sim \langle\lambda_1\rangle \sim  \langle\lambda_2 \rangle\sim  \langle\lambda_3 \rangle$. Using the same method as in  $\mathscr{L}^{-,h}_{h_-h_-h_-}(u,v)$, we can get the result as desired.

{\it Case $h_-l_- l_{l\pm} $.} We decompose $\lambda _k$:
$$
\lambda_1\in \lambda_0 +[0, - \frac{\lambda_0}{4} ) = \bigcup_{j_1\geq 0} (\lambda_0 + I_{j_1});\ \ \lambda_2 \in [ \frac{3\lambda_0}{4}, 0) = \bigcup_{j_2\geq 0} -I_{j_2};\ \ \lambda_3 \in [ \frac{\lambda_0}{4}, -\frac{\lambda_0}{4}+C) = \bigcup_{j_3\geq 0} \pm I_{j_3}.
$$

Noticing that $\left | \lambda _{0}-\lambda _{2} \right |\geqslant \langle \lambda _{0}\rangle ,\ \ \left | \lambda _{1}-\lambda _{3} \right |\geqslant \langle \lambda _{0}\rangle $, so we can use twice bilinear estimate \eqref{bilinear4} and Lemma \ref{V2toX} to obtain that
\begin{align}
\mathscr{L}^{-}_{h_-l_- l_{l\pm}} (u,v)&:=\sum_{\lambda _{0}\ll 0,\:0\leqslant j_1,j_2,j_3\lesssim \mathrm{ln} \left \langle \lambda _{0} \right \rangle } \langle \lambda_0 \rangle^{\frac{1}{2}} \int_{\mathbb{R}\times [0,T]} \overline{v}_{\lambda_0} \partial_x^2 u_{\lambda_0+I_{j_1}} u_{-I_{j_2}}\overline{u}_{\pm I_{j_3}}dxdt\nonumber\\
&\lesssim\sum_{0\leqslant j_1,j_2,j_3\lesssim \mathrm{ln} \left \langle \lambda _{0} \right \rangle } \langle\lambda_0 \rangle^{1/2}\|{\overline{v}_{\lambda_0}u_{-I_{j_2}}} \|_{L^2_{x,t}}\|{(\partial_x^2 u_{\lambda_0+I_{j_1}})\overline{u}_{\pm I_{j_3}}} \|_{L^2_{x,t}} \nonumber\\
&\lesssim T^{\varepsilon/2}\sum_{0\leqslant j_1,j_2,j_3\lesssim \mathrm{ln}\langle \lambda _{0}\rangle } \langle\lambda_0 \rangle^{5/2}\left[\langle \lambda _{0}\rangle^3\right]^{-1/2+\varepsilon }2^{j_2(-1/q)}\nonumber\\
&\qquad\qquad\qquad\cdot\left[\langle \lambda _{0} \rangle ^3 \right ]^{-1/2+\varepsilon }
2^{j_1(1/2-1/q)+j_3(-1/q)}\langle \lambda _{0}\rangle ^{-1/2} \|{v_{\lambda_0}}\|_{V^2_S}\| u \|^3_{X^{1/2}_{q,S}} \nonumber\\
&\lesssim T^{\varepsilon/2}\sum_{0\leqslant j_1,j_2,j_3\lesssim \mathrm{ln} \left \langle \lambda _{0} \right \rangle }\left \langle \lambda _{0}  \right \rangle^{-1+6\varepsilon }2^{j_1(1/2-1/q)+(j_2+j_3)(-1/q)}\|{v_{\lambda_0}} \|_{V^2_S}\| u \|^3_{X^{1/2}_{q,S}}  \nonumber\\
&\lesssim T^{\varepsilon/2}\sum_{\lambda _{0}\ll 0}\left \langle \lambda _{0}  \right \rangle^{-1/2-1/q+6\varepsilon }\|{v_{\lambda_0}} \|_{V^2_S}\| u \|^3_{X^{1/2}_{q,S}} \nonumber\\
&\lesssim T^{\varepsilon/2}\|{v_{\lambda_0}} \|_{Y^0_{q',S}}\| u \|^3_{X^{1/2}_{q,S}}\label{pm4}.
\end{align}
Here the summations on $j_2,j_3$ are convergent to a constant, the summation on $j_1$ is controlled by $\langle \lambda _{0}\rangle^{1/2-1/q}$, and we finally use  H\"older's inequality on $\lambda_0$.

{\it Case $h_-a_+a_+$.} We decompose $\lambda _k$:\
\begin{align*}
&\lambda_1\in \lambda_0 +[0, - \frac{\lambda_0}{4} ) = \bigcup_{j_1\geqslant  0} (\lambda_0 + I_{j_1});\ \ \lambda_2 \in [ 0, \infty ) = \bigcup_{j_2\geqslant 0}I_{j_2};\\
&\lambda_3 \in [ -C,\infty) = \bigcup_{j_3\geqslant -1}I_{j_3},\ \ I_{-1}=[-C,0).
\end{align*}
In this case we have $\left | \lambda _{0}-\lambda _{2} \right |\gtrsim  \left \langle \lambda _{0} \right \rangle  ,\ \ \left | \lambda _{1}-\lambda _{3} \right |\gtrsim  \left \langle \lambda _{0} \right \rangle $ and $j_1\lesssim \mathrm{ln} \left \langle \lambda _{0} \right \rangle $. Noticing that we only use $j_1\lesssim \mathrm{ln} \left \langle \lambda _{0} \right \rangle $ and do not use $j_2,j_3\lesssim \mathrm{ln} \left \langle \lambda _{0} \right \rangle $ in \eqref{pm4}, so we can use the same method as in \eqref{pm4} to estimate this case and the details are omitted.

{\bf Case B.} $\lambda_1\in l_-$. We need to consider the following subcases as in Table \ref{caseB}.
\begin{table}[h]
\begin{center}

\begin{tabular}{|c|c|c|c|}
\hline
${\rm Case\ B}$   &  $\lambda_1\in $ & $\lambda_2\in $ & $\lambda_3\in $    \\
\hline
$l_-h_-l_{h-}$  & $ [3\lambda_0/4,0 )$  & $ [\lambda_0, 3\lambda_0/4)$  &  $[3\lambda_0/4-C,\lambda_0/4 )$  \\
\hline
$l_-h_-l_{l\pm}$ & $ [3\lambda_0/4,0 )$  & $ [\lambda_0,3\lambda_0/4)$ & $ [\lambda_0/4, -\lambda_0/4+C)$  \\
\hline
$l_-l_{h-}l_{h-}$ & $ [3\lambda_0/4,0 )$  & $ [3\lambda_0/4,\lambda_0/4)$  & $ [\lambda_0/2-C,\lambda_0/8)$\\
\hline
$l_-l_{h-}l_{l-}$ & $ [3\lambda_0/4,0 )$  & $ [3\lambda_0/4,\lambda_0/4)$  & $ [\lambda_0/8,0)$\\
\hline
$l_-l_{h-}l_+$ & $ [3\lambda_0/4,0 )$  & $ [3\lambda_0/4,\lambda_0/4)$  & $ [0,-3\lambda_0/4+C)$\\
\hline
$l_-l_{l-}l_+$ & $ [3\lambda_0/4,0 )$  & $ [\lambda_0/4,0)$  & $ [-C,-\lambda_0/2)$\\
\hline
$l_-l_{l-}h_+$ & $ [3\lambda_0/4,0 )$  & $ [\lambda_0/4,0)$  & $ [-\lambda_0/2,-\lambda_0+C)$\\
\hline
$l_-a_+a_+$ & $ [3\lambda_0/4,0 )$  & $ [0,\infty )$  & $ [-\lambda_0/4-C,\infty)$\\
\hline
\end{tabular}
\end{center}
\caption{$\lambda_0\ll 0$, $\lambda_1\in l_-$, dividing the range of $\lambda_2$.}\label{caseB}
\end{table}

For {\it Case $l_-h_-l_{h-}$} and {\it Case $l_-h_-l_{l\pm}$}, we can use the same method as for {\it Case $h_-l_-l_{h-}$} and {\it Case $h_-l_-l_{l\pm}$} in {\it Case $A$}, respectively. In fact, these two cases are easier because the derivative at low frequency is simpler than at high frequency, so we omit it.

{\it Case $l_-l_{h-}l_{h-}$.} By FCC \eqref{FCC}, we see that $\lambda_1<3\lambda_0/8+C$. Then we have $ \langle \lambda _k\rangle \sim \langle \lambda_0 \rangle$ for $k=1,2,3$.  We decompose $\lambda_k$:
\begin{align*}
&\lambda_1\in \frac{3\lambda_0}{4}  +[0, - \frac{3\lambda_0}{8}+C ) = \bigcup_{j_1\geqslant  0} (\frac{3\lambda_0}{4} + I_{j_1});\ \ \lambda_2 \in \frac{3\lambda_0}{4}  +[0 , -\frac{\lambda_0}{2}) = \bigcup_{j_2\geqslant 0}(\frac{3\lambda_0}{4}+I_{j_2});\\
&\lambda_3 \in \frac{\lambda_0}{2}+[-C, -\frac{3\lambda_0}{8} )= \bigcup_{j_3\geqslant -1}(\frac{\lambda_0}{2}+I_{j_3}),\ \ I_{-1}=[-C,0).
\end{align*}
Noticing that $\left | \lambda _{0}-\lambda _{2} \right |\gtrsim  \left \langle \lambda _{0} \right \rangle$ and $j_1,j_2,j_3\lesssim  \mathrm{ln}\langle \lambda _{0} \rangle $. It follows from the bilinear estimate \eqref{bilinear4}, $L^4$ estimates \eqref{lebesgue4a} and Lemma \ref{V2toX} that
\begin{align*}
\mathscr{I}^{-}_{l_-l_{h-}l_{h-}} (u,v):=&\sum_{\lambda _{0}\ll 0,\:j_1,j_2,j_3\lesssim \mathrm{ln} \left \langle \lambda _{0} \right \rangle } \langle\lambda_0 \rangle^{1/2}\|{\overline{v}_{\lambda_0}u_{3\lambda_0/4+I_{j_2}}} \|_{L^2_{x,t}}\|{\partial_x^2 u_{3\lambda_0/4+I_{j_1}}} \|_{L^4_{x,t}}\|{\overline{u}_{\lambda_0/2+I_{j_3}}} \|_{L^4_{x,t}}  \nonumber\\
\lesssim& T^{3\varepsilon/4}\sum_{\lambda _{0}\ll 0,\:j_1,j_2,j_3\lesssim \mathrm{ln} \left \langle \lambda _{0} \right \rangle } \langle\lambda_0 \rangle^{1/2}\cdot \left [ \langle\lambda_0 \rangle^3 \right ]^{-1/2+\varepsilon }\cdot 2^{j_2(1/2-1/q)}\langle\lambda_0 \rangle^{-1/2}  \nonumber\\
&\quad\quad\cdot \langle\lambda_0 \rangle^2 2^{(j_1+j_3)(1/4-1/q+\varepsilon)}\langle\lambda_0 \rangle^{-3/4}\langle\lambda_0 \rangle^{-3/4}\|{v_{\lambda_0}} \|_{V^2_S}\| u \|^3_{X^{1/2}_{q,S}} \nonumber\\
\lesssim& T^{3\varepsilon/4}\sum_{\lambda _{0}\ll 0} \langle\lambda_0 \rangle^{-3/q+5\varepsilon }\|{v_{\lambda_0}}\|_{V^2_S}\| u \|^3_{X^{1/2}_{q,S}}  \nonumber\\
\lesssim& T^{3\varepsilon/4}\|v\|_{Y^0_{q',S}}\| u \|^3_{X^{1/2}_{q,S}}.
\end{align*}

{\it Case $l_-l_{h-}l_{l-}$.} By FCC \eqref{FCC}, we see that $\lambda_1<\lambda_0/4+C$. We decompose $\lambda_k$:
\begin{align*}
&\lambda_1\in \frac{3\lambda_0}{4}  +[0, - \frac{\lambda_0}{2}+C ) = \bigcup_{j_1\geqslant  0} (\frac{3\lambda_0}{4} + I_{j_1});\ \ \lambda_2 \in \frac{\lambda_0}{4}  +[\frac{\lambda_0}{2} , 0) = \bigcup_{j_2\geqslant 0}(\frac{\lambda_0}{4}-I_{j_2});\\
&\lambda_3 \in [\frac{\lambda_0}{8}, 0 )= \bigcup_{j_3\geqslant 0}-I_{j_3}.
\end{align*}
Noticing that $\left | \lambda _{0}-\lambda _{1} \right |,\ \left | \lambda _{2}-\lambda _{3} \right |\gtrsim  \left \langle \lambda _{0} \right \rangle$ and $j_1,j_2,j_3\lesssim  {\rm ln}\left \langle \lambda _{0} \right \rangle $. It follows from the bilinear estimate \eqref{bilinear4} and Lemma \ref{V2toX} that
\begin{align*}
\mathscr{I}^{-}_{l_-l_{h-}l_{l-}} (u,v)\lesssim
&\sum_{\lambda _{0}\ll 0,\:j_1,j_2,j_3\lesssim \mathrm{ln} \left \langle \lambda _{0} \right \rangle } \langle\lambda_0 \rangle^{1/2}\|{\overline{v}_{\lambda_0}\partial_x^2 u_{3\lambda_0/4+I_{j_1}}} \|_{L^2_{x,t}}\|{u_{\lambda_0/4-I_{j_2}}\overline{u}_{-I_{j_3}}} \|_{L^2_{x,t}} \nonumber\\
\lesssim& T^{\varepsilon/2}\sum_{\lambda _{0}\ll 0,\:j_1,j_2,j_3\lesssim \mathrm{ln} \left \langle \lambda _{0} \right \rangle } \langle\lambda_0 \rangle^{1/2}\langle\lambda_0 \rangle^2\cdot \left [ \langle\lambda_0 \rangle^3 \right ]^{-1/2+\varepsilon }\cdot 2^{j_1(1/2-1/q)}\langle\lambda_0 \rangle^{-1/2}\nonumber\\
&\quad\quad\cdot \left [ \langle\lambda_0 \rangle^3 \right ]^{-1/2+\varepsilon } 2^{j_2(1/2-1/q)+j_3(-1/q)}\langle\lambda_0 \rangle^{-1/2}\|{v_{\lambda_0}} \|_{V^2_S}\| u \|^3_{X^{1/2}_{q,S}}\nonumber\\
\lesssim& T^{\varepsilon/2}\sum_{\lambda _{0}\ll 0} \langle\lambda_0 \rangle^{-1/2-2/q+6\varepsilon }\|{v_{\lambda_0}}\|_{V^2_S}\| u \|^3_{X^{1/2}_{q,S}} \nonumber\\
\lesssim& T^{\varepsilon/2}\|v\|_{Y^0_{q',S}}\| u \|^3_{X^{1/2}_{q,S}},
\end{align*}
where we use the fact $\langle \lambda _{1}\rangle\sim \langle \lambda _{2}\rangle\sim\langle \lambda _{0}\rangle$.

{\it Case $l_-l_{h-}l_+$.} We collect $\lambda _k$ in the following dyadic version:
\begin{align*}
&\lambda_1\in [\frac{3\lambda_0}{4}, 0) = \bigcup_{j_1\geqslant  0} - I_{j_1};\ \ \lambda_2 \in \lambda_0+[-\frac{\lambda_0}{4} ,-\frac{3\lambda_0}{4}) = \bigcup_{j_2\lesssim \mathrm{ln}\left \langle \lambda _{0}  \right \rangle  }(\lambda_0+I_{j_2});\\
&\lambda_3 \in [0, -\frac{3\lambda_0}{4}+C )= \bigcup_{j_3\geqslant 0}I_{j_3}.
\end{align*}
Notice that the frequency of $(\lambda_0,\lambda_1)$ and $(\lambda_2,\lambda_3)$ have  transversality and $j_1,j_2,j_3\lesssim  \ln \left \langle \lambda _{0} \right \rangle $. It follows from the bilinear estimate \eqref{bilinear4} and Lemma \ref{V2toX} that
\begin{align}\label{lll}
\mathscr{I}^{-}_{l_-l_{h-}l_+} (u,v)
%:=& \sum_{\lambda _{0}\ll 0,\:j_1,j_2,j_3\lesssim \mathrm{ln} \left \langle \lambda _{0} \right \rangle } \left \langle \lambda _{0} \right \rangle^{1/2}  \int_{\mathbb{R}\times \left [ 0,T \right ]  }\left |\overline{v}_{\lambda _{0}} (\partial _{x}^{2}u_{\lambda _{0}-I_{j_1}})u_{\lambda_0+I_{j_2}}\overline{u}_{I_{j_3}} \right |dxdt \nonumber\\
\lesssim&\sum_{\lambda _{0}\ll 0,\:j_1,j_2,j_3\lesssim \mathrm{ln} \left \langle \lambda _{0} \right \rangle } \langle\lambda_0 \rangle^{1/2}\|{\overline{v}_{\lambda_0}\partial_x^2 u_{-I_{j_1}}} \|_{L^2_{x,t}}\|{u_{\lambda_0+I_{j_2}}\overline{u}_{I_{j_3}}} \|_{L^2_{x,t}}\nonumber\\
\lesssim&T^{\varepsilon/2}\sum_{\lambda _{0}\ll 0,\:j_1,j_2,j_3\lesssim \mathrm{ln} \left \langle \lambda _{0} \right \rangle } \langle\lambda_0 \rangle^{1/2}\cdot 2^{2j_1}\left [ \langle\lambda_0 \rangle^3 \right ]^{-1/2+\varepsilon }2^{j_1(-1/q)}   \nonumber\\
&\quad\quad\cdot \left [ \langle\lambda_0 \rangle^3 \right ]^{-1/2+\varepsilon }2^{j_2(1/2-1/q)+j_3(-1/q)}\langle\lambda_0 \rangle^{-1/2}\|{v_{\lambda_0}} \|_{V^2_S}\| u \|^3_{X^{1/2}_{q,S}} \nonumber\\
\lesssim&T^{\varepsilon/2}\sum_{\lambda _{0}\ll 0} \langle\lambda_0 \rangle^{-1/2-2/q+6\varepsilon }\|{v_{\lambda_0}}\|_{V^2_S}\| u \|^3_{X^{1/2}_{q,S}}\nonumber\\
\lesssim& T^{\varepsilon/2}\|v\|_{Y^0_{q',S}}\| u \|^3_{X^{1/2}_{q,S}}.
\end{align}

{\it Case $l_-l_{l-}l_+$.} By FCC \eqref{FCC}, we see that $\lambda_1<\lambda_0/4+C$. We decompose $\lambda_k$:
\begin{align*}
&\lambda_1\in \lambda_0+[-\frac{\lambda_0}{4}, -\frac{3\lambda_0}{4}+C) = \bigcup_{j_1\lesssim \mathrm{ln}\left \langle \lambda _{0}  \right \rangle  } (\lambda_0+ I_{j_1});\ \ \lambda_2 \in [\frac{\lambda_0}{4} ,0) = \bigcup_{j_2\geqslant 0}(-I_{j_2});\\
&\lambda_3 \in [-C, -\frac{\lambda_0}{2} )= \bigcup_{j_3\geq -1}I_{j_3}.
\end{align*}
Notice that the frequency of $(\lambda_0,\lambda_2)$ and $(\lambda_1,\lambda_3)$ have  transversality and $j_1,j_2,j_3\lesssim  \ln \left \langle \lambda _{0} \right \rangle $. It follows from the bilinear estimate \eqref{bilinear4} and Lemma \ref{V2toX} that
\begin{align*}
\mathscr{I}^{-}_{l_-l_{l-}l_+} (u,v):=& \sum_{\lambda _{0}\ll 0,\:j_1,j_2,j_3\lesssim \mathrm{ln} \left \langle \lambda _{0} \right \rangle } \left \langle \lambda _{0} \right \rangle^{1/2}  \int_{\mathbb{R}\times \left [ 0,T \right ]  }\left |\overline{v}_{\lambda _{0}} (\partial _{x}^{2}u_{\lambda _{0}+I_{j_1}})u_{-I_{j_2}}\overline{u}_{I_{j_3}} \right |dxdt \nonumber\\
\lesssim&\sum_{\lambda _{0}\ll 0,\:j_1,j_2,j_3\lesssim \mathrm{ln} \left \langle \lambda _{0} \right \rangle } \langle\lambda_0 \rangle^{1/2}\|{\overline{v}_{\lambda_0} u_{-I_{j_2}}} \|_{L^2_{x,t}}\|{(\partial_x^2u_{\lambda_0+I_{j_1}})\overline{u}_{I_{j_3}}} \|_{L^2_{x,t}} \nonumber\\
\lesssim&T^{\varepsilon/2}\sum_{\lambda _{0}\ll 0,\:j_1,j_2,j_3\lesssim \mathrm{ln} \left \langle \lambda _{0} \right \rangle } \langle\lambda_0 \rangle^{1/2}\cdot \left [ \langle\lambda_0 \rangle^3 \right ]^{-1/2+\varepsilon }2^{j_2(-1/q)}\nonumber\\
&\quad\quad\cdot \langle\lambda_0 \rangle^2\left [ \langle\lambda_0 \rangle^3 \right ]^{-1/2+\varepsilon }2^{j_1(1/2-1/q)+j_3(-1/q)}\langle\lambda_0 \rangle^{-1/2}\|{v_{\lambda_0}} \|_{V^2_S}\| u \|^3_{X^{1/2}_{q,S}}\nonumber\\
\lesssim&T^{\varepsilon/2}\sum_{\lambda _{0}\ll 0} \langle\lambda_0 \rangle^{-1/2-1/q+6\varepsilon }\|{v_{\lambda_0}}\|_{V^2_S}\| u \|^3_{X^{1/2}_{q,S}}\nonumber\\
\lesssim& T^{\varepsilon/2}\|v\|_{Y^0_{q',S}}\| u \|^3_{X^{1/2}_{q,S}},
\end{align*}
where we use the same summation order as in \eqref{pm4}.

{\it Case $l_-l_{l-}h_+$.} We decompose $\lambda_k$ in the following way:
\begin{align*}
&\lambda_1\in [\frac{3\lambda_0}{4},0) = \bigcup_{j_1\geqslant 0} -I_{j_1};\ \ \lambda_2 \in [\frac{\lambda_0}{4} ,0) = \bigcup_{j_2\geqslant 0}-I_{j_2};\\
&\lambda_3 \in -\lambda_0+[\frac{\lambda_0}{2} ,C) = \bigcup_{j_3\geqslant -1}(-\lambda_0-I_{j_3}),\ \ I_{-1}=[-C,0).
\end{align*}
Notice that the frequency of $(\lambda_0,\lambda_1$) and $(\lambda_2,\lambda_3)$ have  transversality, and $j_1,j_2,j_3\lesssim  \ln \left \langle \lambda _{0} \right \rangle $.  Using a similar way as in \eqref{lll}, we get that
\begin{align*}
\mathscr{I}^{-}_{l_-l_-h_+} (u,v):=& \sum_{\lambda _{0}\ll 0,\:j_1,j_2,j_3\lesssim \mathrm{ln} \left \langle \lambda _{0} \right \rangle } \left \langle \lambda _{0} \right \rangle^{1/2}  \int_{\mathbb{R}\times \left [ 0,T \right ]  }\left |\overline{v}_{\lambda _{0}} (\partial _{x}^{2}u_{-I_{j_1}})u_{-I_{j_2}}\overline{u}_{-\lambda _{0}-I_{j_3}} \right |dxdt \nonumber\\
\lesssim&\sum_{\lambda _{0}\ll 0,\:j_1,j_2,j_3\lesssim \mathrm{ln} \left \langle \lambda _{0} \right \rangle } \langle\lambda_0 \rangle^{1/2}\|{\overline{v}_{\lambda_0} \partial_x^2u_{-I_{j_1}}} \|_{L^2_{x,t}}\|{u_{-I_{j_2}}\overline{u}_{-\lambda_0-I_{j_3}}} \|_{L^2_{x,t}} \nonumber\\
\lesssim&T^{\varepsilon/2}\sum_{\lambda _{0}\ll 0,\:j_1,j_2,j_3\lesssim \mathrm{ln} \left \langle \lambda _{0} \right \rangle } \langle\lambda_0 \rangle^{1/2}\cdot 2^{2j_1}\left [ \langle\lambda_0 \rangle^3 \right ]^{-1/2+\varepsilon }2^{j_1(-1/q)}\nonumber\\
&\quad\quad \cdot\left [ \langle\lambda_0 \rangle^3 \right ]^{-1/2+\varepsilon }2^{j_2(-1/q)+j_3(1/2-1/q)}\langle\lambda_0 \rangle^{-1/2}\|{v_{\lambda_0}}\|_{V^2_S}\| u \|^3_{X^{1/2}_{q,S}}\nonumber\\
\lesssim&T^{\varepsilon/2}\sum_{\lambda _{0}\ll 0} \langle\lambda_0 \rangle^{-1/2-2/q+6\varepsilon }\|{v_{\lambda_0}}\|_{V^2_S}\| u \|^3_{X^{1/2}_{q,S}}\nonumber\\
\lesssim& T^{\varepsilon/2}\|v\|_{Y^0_{q',S}}\| u \|^3_{X^{1/2}_{q,S}}.
\end{align*}

{\it Case $l_-a_+a_+$.} We decompose $\lambda_k$ in the following way:
\begin{align*}
&\lambda_1\in \frac{3\lambda_0}{4}+[0, -\frac{3\lambda_0}{4}) = \bigcup_{j_1\geqslant 0} (\frac{3\lambda_0}{4}+ I_{j_1});\ \ \lambda_2 \in [0,\infty ) = \bigcup_{j_2\geqslant 0}I_{j_2};\\
&\lambda_3 \in -\frac{\lambda_0}{4}+[-C, \infty )= \bigcup_{j_3\geqslant -1}(-\frac{\lambda_0}{4}+I_{j_3}),\ \ I_{-1}=[-C,0).
\end{align*}
Notice that the frequency of $(\lambda_0,\lambda_2)$ and $(\lambda_1,\lambda_3)$ have  transversality and $j_1\lesssim {\rm ln} \left \langle \lambda _{0} \right \rangle $. Using the embedding $\|{u_{3\lambda _{0}/4+I_{j_1}}}\|_{V^2_S}\lesssim \| u \|_{X^{1/2}_{q,S}}$ for any fixed $j_1$, it follows from the bilinear estimate \eqref{bilinear4} and Lemma \ref{V2toX} that
\begin{align*}
\mathscr{I}^{-}_{l_-a_+a_+} (u,v):=& \sum_{j_1\lesssim \mathrm{ln} \left \langle \lambda _{0} \right \rangle ,j_2,j_3} \left \langle \lambda _{0} \right \rangle^{1/2}  \int_{\mathbb{R}\times \left [ 0,T \right ]  }\left |\overline{v}_{\lambda _{0}} (\partial _{x}^{2}u_{3\lambda _{0}/4+I_{j_1}})u_{I_{j_2}}\overline{u}_{-\lambda _{0}/4+I_{j_3}} \right |dxdt \nonumber\\
\lesssim&\sum_{j_1\lesssim \mathrm{ln} \left \langle \lambda _{0} \right \rangle ,j_2,j_3} \langle\lambda_0 \rangle^{1/2}\|{\overline{v}_{\lambda_0}u_{I_{j_2}}} \|_{L^2_{x,t}}\|{ (\partial_x^2u_{3\lambda_0/4+I_{j_1}})\overline{u}_{-\lambda_0/4+I_{j_3}}} \|_{L^2_{x,t}}\nonumber\\
\lesssim&T^{\varepsilon/2}\sum_{j_1\lesssim \mathrm{ln} \left \langle \lambda _{0} \right \rangle ,j_2,j_3} \langle\lambda_0 \rangle^{5/2}\cdot \left [ \langle\lambda_0 \rangle^3 \right ]^{-1/2+\varepsilon }2^{j_2(-1/q)}\left [ \langle\lambda_0 \rangle (\langle\lambda_0 \rangle+2^{j_3}\right )^2]^{-1/2+\varepsilon }\nonumber\\
&\qquad\qquad\qquad \cdot2^{j_3(1/2-1/q)}(\left \langle \lambda_0 \right \rangle+2^{j_3} )^{-1/2}\|{v_{\lambda_0}}\|_{V^2_S}\| u \|^3_{X^{1/2}_{q,S}}\nonumber\\
\lesssim&T^{\varepsilon/2}\sum_{j_1\lesssim \mathrm{ln} \left \langle \lambda _{0} \right \rangle,j_3}\langle\lambda_0 \rangle^{1/2+4\varepsilon }(\left \langle \lambda_0 \right \rangle+2^{j_3} )^{-3/2+2\varepsilon }2^{j_3(1/2-1/q)}\|{v_{\lambda_0}}\|_{V^2_S}\| u \|^3_{X^{1/2}_{q,S}},
\end{align*}
Here we make the summation on $j_2$. Then one can use the inequalities $$(\langle \lambda _0  \rangle+2^{j_3} )^{-3/2+2\varepsilon }\lesssim \langle \lambda _0  \rangle ^{-1-1/q+3\varepsilon }2^{j_3(-1/2+1/q-\varepsilon)}, \sum_{j_1\lesssim \mathrm{ln} \langle \lambda _{0}\rangle}1\lesssim \langle \lambda _0  \rangle ^{\varepsilon},$$ and H\"older's inequality on $\lambda_0$ to obtain that
\begin{align*}
\mathscr{I}^{-}_{l_-a_+a_+} (u,v)\lesssim&T^{\varepsilon/2}\sum_{\lambda _{0}\ll 0,j_3\geq -1} \langle\lambda_0 \rangle^{-1/2-1/q+8\varepsilon }2^{j_3(-\varepsilon)}\|{v_{\lambda_0}}\|_{V^2_S}\| u \|^3_{X^{1/2}_{q,S}}\nonumber\\
\lesssim& T^{\varepsilon/2}\|v\|_{Y^0_{q',S}}\| u \|^3_{X^{1/2}_{q,S}}.
\end{align*}

{\bf Case C.} $\lambda_1\in [0,\infty)$. We need to consider the following subcases as in Table \ref{caseC}.
\begin{table}[h]
\begin{center}

\begin{tabular}{|c|c|c|c|}
\hline
${\rm Case\ C}$   &  $\lambda_1\in $ & $\lambda_2\in $ & $\lambda_3\in $    \\
\hline
$a_+h_-a_+$  & $ [0,\infty )$  & $ [\lambda_0, \lambda_0/2)$  &  $[-C,\infty )$  \\
\hline
$a_+l_-a_+$ & $ [0,\infty )$  & $ [\lambda_0/2,0)$ & $ [-\lambda_0/2-C, \infty)$  \\
\hline
%$a_+a_+a$ & $ [0,\infty )$  & $ [0,\infty )$  & $ [-\lambda_0-C,\infty)$\\
$l_+a_+a_+$ & $ [0,-\lambda_0/2)$  & $ [0,\infty )$  & $ [-\lambda_0-C,\infty)$\\
\hline
$a_+l_+a_+$ & $ [-\lambda_0/2,\infty)$  & $ [0,-\lambda_0/2)$  & $ [-3\lambda_0/2-C,\infty)$\\
\hline
$a_+a_+a_+$ & $ [-\lambda_0/2,\infty)$  & $ [-\lambda_0/2,\infty)$  & $ [-2\lambda_0-C,\infty)$\\
\hline
\end{tabular}
\end{center}
\caption{$\lambda _{0}=\underset{0\leqslant k\leqslant 3}{\mathrm{min}  }\lambda _{k},\ \ \lambda_0\ll 0,\ \ \lambda_1\in [0,\infty)$.}\label{caseC}
\end{table}

{\it Case $a_+h_-a_+$.} We decompose $\lambda_k$ in the following way:
\begin{align*}
&\lambda_1\in [0, \infty ) = \bigcup_{j_1\geqslant 0} I_{j_1};\ \ \lambda_2 \in \lambda_0+[0,-\frac{\lambda_0}{2}) = \bigcup_{j_2\geqslant 0}(\lambda_0+I_{j_2});\\
&\lambda_3 \in [-C, \infty )= \bigcup_{j_3\geqslant -1}I_{j_3},\ \ I_{-1}=[-C,0).
\end{align*}

Notice that the frequency of $(\lambda_0,\lambda_1)$ and $(\lambda_2,\lambda_3)$ have  transversality, $j_3\approx (j_1\vee j_2) $ by FCC \eqref{FCC} and $j_2\lesssim \mathrm{ln} \left \langle \lambda _{0} \right \rangle$. It follows from the bilinear estimate \eqref{bilinear4} and Lemma \ref{V2toX} that
\begin{align}\label{aha}
\mathscr{X}^{-}_{a_+h_-a_+} (u,v):= &\sum_{\lambda _{0}\ll 0,\:j_3\approx (j_1\vee j_2),\ j_2\lesssim \mathrm{ln} \left \langle \lambda _{0} \right \rangle} \left \langle \lambda _{0} \right \rangle^{1/2}  \int_{\mathbb{R}\times \left [ 0,T \right ]  }\left |\overline{v}_{\lambda _{0}} (\partial _{x}^{2}u_{I_{j_1}})u_{\lambda _{0}+I_{j_2}}\overline{u}_{I_{j_3}} \right |dxdt \nonumber\\
\lesssim&\sum_{\lambda _{0}\ll 0,\:j_3\approx (j_1\vee j_2),\ j_2\lesssim \mathrm{ln} \left \langle \lambda _{0} \right \rangle} \left \langle \lambda _{0} \right \rangle^{1/2}\|{\overline{v}_{\lambda_0}\partial_x^2u_{I_{j_1}}} \|_{L^2_{x,t}}\|{ u_{\lambda_0+I_{j_2}}\overline{u}_{I_{j_3}}} \|_{L^2_{x,t}}\nonumber\\
\lesssim&T^{\varepsilon/2}\sum_{\lambda _{0}\ll 0,\:j_3\approx (j_1\vee j_2),\ j_2\lesssim \mathrm{ln} \left \langle \lambda _{0} \right \rangle} \left \langle \lambda _{0} \right \rangle^{1/2}\cdot 2^{2j_1}\left [ \langle\lambda_0 \rangle\cdot  2^{2j_1} \right ] ^{-1/2+\varepsilon }2^{j_1(-1/q)}\nonumber\\
&\quad\quad \cdot  \left [ \langle\lambda_0 \rangle \cdot 2^{2j_3} \right ] ^{-1/2+\varepsilon }2^{j_2(1/2-1/q)+j_3(-1/q)}\langle\lambda_0 \rangle^{-1/2}\|{v_{\lambda_0}}\|_{V^2_S}\| u \|^3_{X^{1/2}_{q,S}}\nonumber\\
\lesssim&T^{\varepsilon/2}\sum_{\lambda _{0}\ll 0,\:j_3} \left \langle \lambda _{0} \right \rangle^{-1/2-1/q+2\varepsilon }2^{j_3(-2/q+4\varepsilon )}\|{v_{\lambda_0}}\|_{V^2_S}\| u \|^3_{X^{1/2}_{q,S}}\nonumber\\
\lesssim& T^{\varepsilon/2}\|v\|_{Y^0_{q',S}}\| u \|^3_{X^{1/2}_{q,S}}.
\end{align}
Here we first use $j_1\lesssim j_3,\ \ j_2\lesssim \mathrm{ln} \left \langle \lambda _{0} \right \rangle$ to make the summation on $j_1,j_2$, then make the summation on $j_3$ for $0<\varepsilon <1/2q$, finally apply H\"older's inequality on $\lambda_0$.

{\it Case $a_+l_-a_+$.} We decompose $\lambda_k$ in the following way:
\begin{align*}
&\lambda_1\in [0, \infty ) = \bigcup_{j_1\geqslant 0} I_{j_1};\ \ \lambda_2 \in \lambda_0+[-\frac{\lambda_0}{2},-\lambda_0) = \bigcup_{j_2\lesssim \mathrm{ln} \left \langle \lambda _{0}\right \rangle}(\lambda_0+I_{j_2});\\
&\lambda_3 \in [-\frac{\lambda_0}{2}-C, \infty )= \bigcup_{j_3\gtrsim  \mathrm{ln} \left \langle \lambda _{0} \right \rangle}I_{j_3}.
\end{align*}

Notice that the frequency of $(\lambda_0,\lambda_1$) and $(\lambda_2,\lambda_3)$ have transversality, $j_3\approx (j_1\vee j_2) $ by FCC \eqref{FCC} and $j_2\lesssim  \mathrm{ln} \left \langle \lambda _{0} \right \rangle$. Using the embedding $\|{u_{\lambda _{0}+I_{j_2}}}\|_{V^2_S}\lesssim \| u \|_{X^{1/2}_{q,S}}$ for any fixed $j_2$, it follows from the bilinear estimate \eqref{bilinear4} and Lemma \ref{V2toX} that
\begin{align*}
\mathscr{X}^{-}_{a_+l_-a_+} (u,v):= &\sum_{\lambda _{0}\ll 0,\:j_3\approx (j_1\vee j_2)\gtrsim \mathrm{ln} \left \langle \lambda _{0} \right \rangle} \left \langle \lambda _{0} \right \rangle^{1/2}  \int_{\mathbb{R}\times \left [ 0,T \right ]  }\left |\overline{v}_{\lambda _{0}} (\partial _{x}^{2}u_{I_{j_1}})u_{\lambda _{0}+I_{j_2}}\overline{u}_{I_{j_3}} \right |dxdt \nonumber\\
\lesssim&\sum_{\lambda _{0}\ll 0,\:j_3\approx (j_1\vee j_2)\gtrsim \mathrm{ln} \left \langle \lambda _{0} \right \rangle} \left \langle \lambda _{0} \right \rangle^{1/2}\|{\overline{v}_{\lambda_0}\partial_x^2u_{I_{j_1}}} \|_{L^2_{x,t}}\|{ u_{\lambda_0+I_{j_2}}\overline{u}_{I_{j_3}}} \|_{L^2_{x,t}}\nonumber\\
\lesssim&T^{\varepsilon/2}\sum_{\lambda _{0}\ll 0,\:j_3\approx (j_1\vee j_2)\gtrsim \mathrm{ln} \left \langle \lambda _{0} \right \rangle} \left \langle \lambda _{0} \right \rangle^{1/2}\cdot 2^{2j_1}\left [( \langle\lambda_0 \rangle+2^{j_1})\cdot  2^{2j_1} \right ] ^{-1/2+\varepsilon }2^{j_1(-1/q)}  \nonumber\\
&\quad\quad \cdot  \left [ \langle\lambda_0 \rangle \cdot 2^{2j_3} \right ] ^{-1/2+\varepsilon }2^{j_3(-1/q)}\|{v_{\lambda_0}}\|_{V^2_S}\| u \|^3_{X^{1/2}_{q,S}}  \nonumber\\
\lesssim&T^{\varepsilon/2}\sum_{\lambda _{0}\ll 0,\:j_3\gtrsim \mathrm{ln} \left \langle \lambda _{0} \right \rangle} \left \langle \lambda _{0} \right \rangle^{2\varepsilon} 2^{j_3(-1/2-2/q+5\varepsilon )}\|{v_{\lambda_0}}\|_{V^2_S}\| u \|^3_{X^{1/2}_{q,S}} \nonumber\\
\lesssim& T^{\varepsilon/2}\|v\|_{Y^0_{q',S}}\| u \|^3_{X^{1/2}_{q,S}},
\end{align*}
where we use the same summation order as in \eqref{aha}.

{\it Case $l_+a_+a_+$.} We collect $\lambda_k$ in the following dyadic version:
\begin{align*}
&\lambda_1\in [0,-\frac{\lambda_0}{2}) = \bigcup_{j_1\geqslant 0} I_{j_1};\ \ \lambda_2\in [0,\infty )= \bigcup_{j_2\geqslant 0}I_{j_2};\\
&\lambda_3 \in -\lambda_0+[-C,\infty )= \bigcup_{j_3\geqslant -1}(-\lambda_0+I_{j_3}),\ \ I_{-1}=[-C,0).
\end{align*}

Notice that the frequency of $(\lambda_0,\lambda_2$) and $(\lambda_1,\lambda_3)$ have transversality and $j_1\lesssim \mathrm{ln} \left \langle \lambda _{0} \right \rangle$. It follows from the bilinear estimate \eqref{bilinear4} and Lemma \ref{V2toX} that
\begin{align*}
\mathscr{X}^{-}_{l_+a_+a_+} (u,v):= &\sum_{\lambda _{0}\ll 0,\:j_1\lesssim \mathrm{ln} \left \langle \lambda _{0} \right \rangle,\ j_2,j_3} \left \langle \lambda _{0} \right \rangle^{1/2}  \int_{\mathbb{R}\times \left [ 0,T \right ]  }\left |\overline{v}_{\lambda _{0}} (\partial _{x}^{2}u_{I_{j_1}})u_{I_{j_2}}\overline{u}_{-\lambda _{0}+I_{j_3}} \right |dxdt \nonumber\\
\lesssim&\sum_{\lambda _{0}\ll 0,\:j_1\lesssim \mathrm{ln} \left \langle \lambda _{0} \right \rangle,\ j_2,j_3} \left \langle \lambda _{0} \right \rangle^{1/2}\|{\overline{v}_{\lambda_0}u_{I_{j_2}}} \|_{L^2_{x,t}}\|{( \partial_x^2u_{I_{j_1}})\overline{u}_{-\lambda_0+I_{j_3}}} \|_{L^2_{x,t}}\nonumber\\
\lesssim&T^{\varepsilon/2}\sum_{\lambda _{0}\ll 0,\:j_1\lesssim \mathrm{ln} \left \langle \lambda _{0} \right \rangle,\ j_2,j_3}\left \langle \lambda _{0} \right \rangle^{1/2}\cdot \left [( \langle\lambda_0 \rangle+2^{j_2})\cdot \langle\lambda_0 \rangle^2 \right ] ^{-1/2+\varepsilon }2^{j_2(-1/q)}  \nonumber\\
&\quad\cdot 2^{2j_1} \left [ ( \langle\lambda_0 \rangle+2^{j_3})^3\right ] ^{-1/2+\varepsilon }2^{j_1(-1/q)+j_3(-1/q)}\|{v_{\lambda_0}}\|_{V^2_S}\| u \|^3_{X^{1/2}_{q,S}}  \nonumber\\
\lesssim&T^{\varepsilon/2}\sum_{\lambda _{0}\ll 0} \left \langle \lambda _{0} \right \rangle^{-1/2-1/q+6\varepsilon}\|{v_{\lambda_0}}\|_{V^2_S}\| u \|^3_{X^{1/2}_{q,S}} \nonumber\\
\lesssim& T^{\varepsilon/2}\|v\|_{Y^0_{q',S}}\| u \|^3_{X^{1/2}_{q,S}}.
\end{align*}
Here we first choose $\left \langle \lambda _{0}\right \rangle$ from $\left ( \left \langle \lambda _{0}\right \rangle+2^{j_k} \right ), (k=2,3)$, then handle it like \eqref{pm4}.

{\it Case $a_+l_+a_+$.} We collect $\lambda_k$ in the following dyadic version:
\begin{align*}
&\lambda_1\in [-\frac{\lambda_0}{2},\infty) = \bigcup_{j_1\gtrsim  \mathrm{ln} \left \langle \lambda _{0}  \right \rangle}I_{j_1};\ \ \lambda_2\in [0,-\frac{\lambda_0}{2})= \bigcup_{j_2\geqslant 0}I_{j_2};\\
&\lambda_3 \in -\lambda_0+[-\frac{\lambda_0}{2}-C,\infty )= \bigcup_{j_3\gtrsim  \mathrm{ln} \left \langle \lambda _{0}  \right \rangle}(-\lambda_0+I_{j_3}).
\end{align*}

Notice that the frequency of $(\lambda_0,\lambda_1$) and $(\lambda_2,\lambda_3)$ have transversality. In view of FCC \eqref{FCC}, we see that $j_3\approx (j_1\vee j_2)\gtrsim \mathrm{ln} \left \langle \lambda _{0} \right \rangle$. It follows from the bilinear estimate \eqref{bilinear4} and Lemma \ref{V2toX} that
\begin{align}
&\mathscr{X}^{-}_{l_+a_+a_+} (u,v):= \sum_{\lambda _{0}\ll 0,\:j_3\approx (j_1\vee j_2)\gtrsim \mathrm{ln} \left \langle \lambda _{0} \right \rangle} \left \langle \lambda _{0} \right \rangle^{1/2}  \int_{\mathbb{R}\times \left [ 0,T \right ]  }\left |\overline{v}_{\lambda _{0}} (\partial _{x}^{2}u_{I_{j_1}})u_{I_{j_2}}\overline{u}_{-\lambda _{0}+I_{j_3}} \right |dxdt \nonumber\\
\lesssim&\sum_{\lambda _{0}\ll 0,\:j_3\approx (j_1\vee j_2)\gtrsim \mathrm{ln} \left \langle \lambda _{0} \right \rangle} \left \langle \lambda _{0} \right \rangle^{1/2}\|{\overline{v}_{\lambda_0}\partial_x^2u_{I_{j_1}}} \|_{L^2_{x,t}}\|{u_{I_{j_2}} \overline{u}_{-\lambda_0+I_{j_3}}} \|_{L^2_{x,t}}\nonumber\\
\lesssim&T^{\varepsilon/2}\sum_{\lambda _{0}\ll 0,\:j_3\approx (j_1\vee j_2)\gtrsim \mathrm{ln} \left \langle \lambda _{0} \right \rangle}\left \langle \lambda _{0} \right \rangle^{1/2}\cdot 2^{2j_1}\cdot \left [( \langle\lambda_0 \rangle+2^{j_1})\cdot 2^{2j_1} \right ] ^{-1/2+\varepsilon }2^{j_1(-1/q)}  \nonumber\\
&\quad \cdot\left [ \langle\lambda_0 \rangle( \langle\lambda_0 \rangle+2^{j_3})^2\right ] ^{-1/2+\varepsilon }2^{j_2(-1/q)+j_3(1/2-1/q)}( \langle\lambda_0 \rangle+2^{j_3})^{-1/2}\|{v_{\lambda_0}}\|_{V^2_S}\| u \|^3_{X^{1/2}_{q,S}}  \nonumber\\
\lesssim&T^{\varepsilon/2}\sum_{\lambda _{0}\ll 0} \left \langle \lambda _{0} \right \rangle^{-1/2-2/q+6\varepsilon}\|{v_{\lambda_0}}\|_{V^2_S}\| u \|^3_{X^{1/2}_{q,S}} \nonumber\\
\lesssim& T^{\varepsilon/2}\|v\|_{Y^0_{q',S}}\| u \|^3_{X^{1/2}_{q,S}}.
\end{align}
where we choose $2^{j_k}$ from $\left ( \left \langle \lambda _{0}\right \rangle+2^{j_k} \right ),k=2,3$, then perform the summation on $j_2, j_1$ and $ j_3$ one by one, and finally apply H\"older's inequality on $\lambda_0$.

{\it Case $a_+a_+a_+$.} We collect $\lambda_k$ in the following dyadic version:
\begin{align*}
&\lambda_k\in [-\frac{\lambda_0}{2},\infty) = \bigcup_{j_k\gtrsim  \mathrm{ln} \left \langle \lambda _{0}  \right \rangle}I_{j_k},\, k=1,2;
&\lambda_3 \in -\lambda_0+[-\lambda_0-C,\infty )= \bigcup_{j_3\gtrsim  \mathrm{ln} \left \langle \lambda _{0}  \right \rangle}(-\lambda_0+I_{j_3}).
\end{align*}

Notice that the frequency of $(\lambda_0,\lambda_3)$ has transversality and $j_3\approx (j_1\vee j_2)\gtrsim  \mathrm{ln} \left \langle \lambda _{0}  \right \rangle $ by FCC \eqref{FCC}. It follows from the bilinear estimate \eqref{bilinear4}, $L^4$ estimates \eqref{lebesgue4a}and Lemma \ref{V2toX} that
\begin{align}
&\mathscr{X}^{-}_{a_+a_+a_+} (u,v):=\sum_{\lambda _{0}\ll 0,\:j_3\approx (j_1\vee j_2)\gtrsim  \mathrm{ln} \left \langle \lambda _{0}  \right \rangle }\left \langle \lambda _{0} \right \rangle^{1/2}  \int_{\mathbb{R}\times \left [ 0,T \right ]  }\left |\overline{v}_{\lambda _{0}} (\partial _{x}^{2}u_{I_{j_1}})u_{I_{j_2}}\overline{u}_{-\lambda _{0}+I_{j_3}} \right |dxdt \nonumber\\
&\lesssim \sum_{\lambda _{0}\ll 0,\:j_3\approx (j_1\vee j_2)\gtrsim  \mathrm{ln} \left \langle \lambda _{0}  \right \rangle } \langle\lambda_0 \rangle^{1/2}\|{\overline{v}_{\lambda_0}u_{-\lambda _{0}+I_{j_3}}} \|_{L^2_{x,t}}\|{\partial_x^2 u_{I_{j_1}}} \|_{L^4_{x,t}}\|{u_{I_{j_2}}} \|_{L^4_{x,t}}  \nonumber\\
&\lesssim  T^{3\varepsilon/4}\sum_{\lambda _{0}\ll 0,\:j_3\approx (j_1\vee j_2)\gtrsim  \mathrm{ln} \left \langle \lambda _{0}  \right \rangle } \langle\lambda_0 \rangle^{1/2}\cdot \left [ ( \langle\lambda_0 \rangle+2^{j_3})^3 \right ]^{-1/2+\varepsilon }\cdot 2^{j_3(1/2-1/q)}( \langle\lambda_0 \rangle+2^{j_3})^{-1/2}  \nonumber\\
&\quad\quad\cdot 2^{2j_1}\cdot 2^{(j_1+j_2)(-1/2-1/q+\varepsilon)}\|{v_{\lambda_0}} \|_{V^2_S}\| u \|^3_{X^{1/2}_{q,S}} \nonumber\\
&\lesssim T^{3\varepsilon/4}\sum_{\lambda _{0}\ll 0} \langle\lambda_0 \rangle^{-3/q+5\varepsilon }\|{v_{\lambda_0}}\|_{V^2_S}\| u \|^3_{X^{1/2}_{q,S}}  \nonumber\\
&\lesssim T^{3\varepsilon/4}\|v\|_{Y^0_{q',S}}\| u \|^3_{X^{1/2}_{q,S}}
\end{align}
Here we first choose $2^{j_3}$ from $\left ( \left \langle \lambda _{0}\right \rangle+2^{j_3} \right )$, then make the summation on $j_1, j_2$ and $ j_3$  by $j_1 \leq j_3+1$ and $j_2,j_3 \gtrsim  \mathrm{ln} \left \langle \lambda _{0} \right \rangle$, finally apply H\"older's inequality on $\lambda_0$.

$\mathbf{Step\:2.}$ $\lambda_0$ is the secondly minimal integer in $\lambda_0,\dots ,\lambda_3$.

$\mathbf{Step\:2.1.}$ We consider the case $\lambda _{1}=\underset{0\leqslant k\leqslant 3}{\mathrm{min} }\lambda _{k}$ and $\lambda _{0}\gg 0$. We need to consider the following cases by dividing the range of $\lambda _1$ as in Table \ref{step2.1}.

\begin{table}[h]
\begin{center}

\begin{tabular}{|c|c|c|c|}
\hline
${\rm Case}$   &  $\lambda_1\in $ & $\lambda_2\in $ & $\lambda_3\in $    \\
\hline
$a_-a_+a_+$  & $ (-\infty,0)$  & $ [2\lambda_0-C,\infty)$  &  $[\lambda_0,\infty)$  \\
\hline
$l_+a_+a_+$ & $ [0,\lambda_0/2 )$  & $ [3\lambda_0/2-C,\infty)$ & $ [\lambda_0, \infty)$  \\
\hline
$h_+a_+a_+$ & $ [\lambda_0/2,\lambda_0)$  & $ [\lambda_0,\infty )$  & $ [\lambda_0,\infty)$\\
\hline
\end{tabular}
\end{center}
\caption{$\lambda _1\leqslant \lambda _0 \leqslant \lambda _3 \leqslant \lambda _2,\ \ \lambda_0\gg  0$.}\label{step2.1}
\end{table}

{\it Case $a_-a_+a_+$.} We decompose $\lambda_k$ in the following way:
\begin{align*}
& \lambda_1\in (-\infty,0 ) = \bigcup_{j_1\geqslant 0} -I_{j_1};\ \ \lambda_2 \in \lambda_0+[\lambda_0-C,\infty ) = \bigcup_{j_2\gtrsim  \mathrm{ln} \left \langle \lambda _{0}  \right \rangle}(\lambda_0+I_{j_2});\\
&\lambda_3 \in [\lambda_0,\infty )= \bigcup_{j_3 \gtrsim \mathrm{ln} \left \langle \lambda _{0}  \right \rangle}I_{j_3}.
\end{align*}
Notice that the frequency of $(\lambda_0,\lambda_2$) and ($\lambda_1,\lambda_3$) have  transversality. In view of FCC \eqref{FCC}, we see that $j_2\approx (j_1\vee j_3)$. It follows from the bilinear estimate \eqref{bilinear4} and Lemma \ref{V2toX} that
\begin{align}\label{aaa}
\mathscr{H}^{+}_{a_-a_+a_+} (u,v):=& \sum_{\lambda _{0}\gg  0,\:j_2\approx (j_1\vee j_3)} \left \langle \lambda _{0} \right \rangle^{1/2}  \int_{\mathbb{R}\times \left [ 0,T \right ]  }\left |\overline{v}_{\lambda _{0}} (\partial _{x}^{2}u_{-I_{j_1}})u_{\lambda _{0}+I_{j_2}}\overline{u}_{I_{j_3}} \right |dxdt\nonumber\\
\lesssim&\sum_{\lambda _{0}\gg  0,\:j_2\approx (j_1\vee j_3)}\langle\lambda_0 \rangle^{1/2}\|{\overline{v}_{\lambda_0}u_{\lambda_0+I_{j_2}}} \|_{L^2_{x,t}}\|{ (\partial_x^2u_{-I_{j_1}})\overline{u}_{I_{j_3}}} \|_{L^2_{x,t}} \nonumber\\
\lesssim& T^{\varepsilon/2}\sum_{\lambda _{0}\gg  0,\:j_2\approx (j_1\vee j_3)}\langle\lambda_0 \rangle^{1/2}\cdot \big(2^{3j_2}\big)^{-1/2+\varepsilon }2^{j_2(1/2-1/q)}(\langle\lambda_0 \rangle+2^{j_2})^{-1/2}\nonumber\\
&\qquad\qquad \cdot 2^{2j_1}\left ( 2^{3j_1} \right )^{-1/2+\varepsilon }\cdot 2^{(j_1+j_3)(-1/q)}\|{v_{\lambda_0}}\|_{V^2_S}\| u \|^3_{X^{1/2}_{q,S}}\nonumber\\
\lesssim&T^{\varepsilon/2}\sum_{\lambda _{0}\gg 0} \langle\lambda_0 \rangle^{-1/2-2/q+6\varepsilon }\|{v_{\lambda_0}}\|_{V^2_S}\| u \|^3_{X^{1/2}_{q,S}}\nonumber\\
\lesssim& T^{\varepsilon/2}\|v\|_{Y^0_{q',S}}\| u \|^3_{X^{1/2}_{q,S}}.
\end{align}
Here the summations are made on $j_1, j_2$ and $ j_3$ one by one by $j_1 \leq j_2+1$ and $j_2\gtrsim  \mathrm{ln} \left \langle \lambda _{0} \right \rangle$.

{\it Case $l_+a_+a_+$.} We decompose $\lambda_k$ in the following way:
\begin{align*}
&\lambda_1\in [0,\frac{\lambda_0}{2}  ) = \bigcup_{j_1\geqslant 0}I_{j_1};\ \ \lambda_2 \in \lambda_0+[\frac{\lambda_0}{2} -C,\infty ) = \bigcup_{j_2\gtrsim \mathrm{ln} \left \langle \lambda _{0}  \right \rangle}(\lambda_0+I_{j_2});\\
&\lambda_3 \in [\lambda_0,\infty )= \bigcup_{j_3\gtrsim \mathrm{ln} \left \langle \lambda _{0}  \right \rangle}I_{j_3}.
\end{align*}
Notice that the frequency of $(\lambda_0,\lambda_2)$ and $(\lambda_1,\lambda_3)$ have transversality and $j_1\leq j_2\approx j_3$ by FCC \eqref{FCC}, thus we can use the same way as in \eqref{aaa} to obtain the same estimate and omit the details.

{\it Case $h_+a_+a_+$.} We decompose $\lambda_k$ in the following way:
$$\lambda_1\in \lambda_0+[-\frac{\lambda_0}{2},0  ) = \bigcup_{j_1\geqslant 0}(\lambda_0-I_{j_1});\ \ \lambda_k \in \lambda_0+[0,\infty ) = \bigcup_{j_k\geqslant 0}(\lambda_0+I_{j_k}),\ k=2,3.$$
From FCC \eqref{FCC} it follows that $j_2\approx (j_1\vee j_3)$. We need to estimate
\begin{align*}
\mathscr{H}^{+}_{h_+a_+a_+} (u,v):=&\sum_{\lambda _{0}\gg  0,\:j_2\approx (j_1\vee j_3)} \left \langle \lambda _{0} \right \rangle^{1/2}  \int_{\mathbb{R}\times \left [ 0,T \right ]  }\left |\overline{v}_{\lambda _{0}} (\partial _{x}^{2}u_{\lambda _{0}-I_{j_1}})u_{\lambda _{0}+I_{j_2}}\overline{u}_{\lambda _{0}+I_{j_3}} \right |dxdt \nonumber\\
\lesssim&\left ( \sum_{\lambda _{0}\gg  0,\:j_2\approx (j_1\vee j_3)\lesssim 1} +\sum_{\lambda _{0}\gg  0,\:j_2\approx (j_1\vee j_3)\gg  1}\right )\left \langle \lambda _{0} \right \rangle^{1/2} \nonumber\\
&\quad\quad \int_{\mathbb{R}\times \left [ 0,T \right ]  }\left |\overline{v}_{\lambda _{0}} (\partial _{x}^{2}u_{\lambda _{0}-I_{j_1}})u_{\lambda _{0}+I_{j_2}}\overline{u}_{\lambda _{0}+I_{j_3}} \right |dxdt \nonumber\\
=:&\mathscr{H}^{+,l}_{h_+a_+a_+} (u,v)+\mathscr{H}^{+,h}_{h_+a_+a_+} (u,v)
\end{align*}

In $\mathscr{H}^{+,l}_{h_+a_+a_+} (u,v)$, we have $\lambda _0\approx \lambda _1\approx \lambda _2\approx \lambda _3$. Using the same way as in $\mathscr{L}^+_{l}(u,v)$, we can get the result as desired.

In $\mathscr{H}^{+,h}_{h_+a_+a_+} (u,v)$, we can use bilinear estimate \eqref{bilinear4}, H\"older's inequality and Lemma \ref{V2toX} to obtain that
\begin{align*}
\mathscr{H}^{+,h}_{h_+a_+a_+} (u,v)
&\lesssim \sum_{j_2\approx (j_1\vee j_3)\gg  1}\langle\lambda_0 \rangle^{1/2}\|{\overline{v}_{\lambda_0}u_{\lambda_0+I_{j_2}}} \|_{L^2_{x,t}}\|\partial_x^2u_{\lambda_0-I_{j_1}}\|_{L^4_{x,t}}\|u_{\lambda_0+I_{j_3}} \|_{L^4_{x,t}}  \nonumber\\
&\lesssim T^{3\varepsilon/4}\sum_{j_2\approx (j_1\vee j_3)\gg  1}\langle\lambda_0 \rangle^{1/2}\cdot \left [ 2^{j_2}\langle\lambda_0 \rangle^2 \right ]^{-1/2+\varepsilon }2^{j_2(1/2-1/q)}(\langle\lambda_0 \rangle+2^{j_2})^{-1/2}\nonumber\\
&\qquad\qquad \cdot \langle\lambda_0 \rangle^22^{(j_1+j_3)(1/4-1/q+\varepsilon )}\cdot \langle\lambda_0 \rangle^{-3/4-3/4}\|{v_{\lambda_0}}\|_{V^2_S}\| u \|^2_{X^{1/2}_{q,S}}\| u \|_{X^{1/2}_{q,S}(\lambda_0+I_{j_2})}
\end{align*}
Here we can made the summation on $j_1, j_3$ by $j_1,j_3\leq j_2+1$, and use $\left ( \left \langle \lambda _{0}\right \rangle+2^{j_2} \right )^{-1/2}$ $\lesssim$ $2^{j_2(-1/2+2\varepsilon )}\left \langle \lambda _{0}\right \rangle^{-2\varepsilon }$ to absorb the positive power of $\langle \lambda _{0}\rangle$. Thus one can apply H\"older's inequality on $\lambda_0$, and then make the summation on $j_2$ for $0<\varepsilon <3/5q$,
\begin{align*}\mathscr{H}^{+,h}_{h_+a_+a_+} (u,v)&\lesssim T^{3\varepsilon/4}\sum_{\lambda _{0}\gg  0,\:j_2\gg  1}2^{j_2(-3/q+5\varepsilon )}\|{v_{\lambda_0}}\|_{V^2_S}\| u \|_{X^{1/2}_{q,S}(\lambda_0+I_{j_2})}\| u \|^2_{X^{1/2}_{q,S}}\nonumber\\
&\lesssim T^{3\varepsilon/4}\|v\|_{Y^0_{q',S}}\| u \|^3_{X^{1/2}_{q,S}}.
\end{align*}

$\mathbf{Step\:2.2.}$ We consider the case $\lambda _{1}=\underset{0\leqslant k\leqslant 3}{\mathrm{min} }\lambda _{k}$ and $\lambda _{0}\ll 0$. We need to consider the following cases by dividing the range of $\lambda _2$ as in Table \ref{step2.2}.

\begin{table}[h]
\begin{center}

\begin{tabular}{|c|c|c|c|}
\hline
${\rm Case}$   &  $\lambda_1\in $ & $\lambda_2\in $ & $\lambda_3\in $    \\
\hline
$1$  & $ [5\lambda_0/4-C,\lambda_0)$  & $ [\lambda_0,3\lambda_0/4)$  &  $[\lambda_0,3\lambda_0/4+C)$  \\
\hline
$2$ & $ [2\lambda_0-C,\lambda_0)$  & $ [3\lambda_0/4,0)$ & $ [\lambda_0,3\lambda_0/4)$  \\
\hline
$3$ & $ [7\lambda_0/4-C,\lambda_0)$  & $ [3\lambda_0/4,0)$  & $ [3\lambda_0/4,C)$\\
\hline
$4$  & $ [11\lambda_0/4-C,\lambda_0)$  & $ [0,-3\lambda_0/4)$  &  $[\lambda_0,0)$  \\
\hline
$5$ & $ [7\lambda_0/4-C,\lambda_0)$  & $ [0,-3\lambda_0/4)$ & $ [0,-3\lambda_0/4+C)$  \\
\hline
$6$ & $ (-\infty ,7\lambda_0/4+C)$  & $ [-3\lambda_0/4,\infty)$  & $ [\lambda_0,0)$\\
\hline
$7$ & $ (-\infty ,\lambda_0)$  & $ [-3\lambda_0/4,\infty)$  & $ [0,\infty)$\\
\hline
\end{tabular}
\end{center}
\caption{$\lambda _1\leqslant \lambda _0 \leqslant \lambda _3 \leqslant \lambda _2 ,\ \ \lambda_0\ll  0$.}\label{step2.2}
\end{table}

{\it Case $1$.} In this case, we have $\langle \lambda _{1}\rangle\sim \langle \lambda _{2}\rangle\sim \langle \lambda _{3}\rangle\sim \langle \lambda _{0}\rangle$, so we can use the similar way as  {\it Case} $h_-h_-h_-$ in Step 1.2 to obtain the result and omit the details.

{\it Case $3$} and {\it Case $5$.} Similar to {\it Case} $h_-l_- l_{l\pm} $ as in Step 1.2, so we omit the details.

{\it Case $2$} and {\it Case $4$.} We decompose $\lambda_k$ in Case 2 and Case 4 as follows:
\begin{equation*}
\left\{\begin{array}{l}
	\lambda_1\in 2\lambda_0+[-C,-\lambda_0) = \bigcup_{j_1\geqslant -1}(2\lambda_0+I_{j_1});\ \ \lambda_2\in [\frac{3\lambda_0}{4},0 ) = \bigcup_{j_2\geqslant 0}(-I_{j_2});\\
\lambda_3\in \lambda_0+[0,-\frac{\lambda_0}{4}) = \bigcup_{j_3\geqslant 0}(\lambda_0+I_{j_3}).
\end{array}\right.
\end{equation*}
\begin{equation*}
\left\{\begin{array}{l}
\lambda_1\in \lambda_0+[\frac{7\lambda_0}{4}-C,0) = \bigcup_{j_1\geqslant 0}(\lambda_0-I_{j_1});\ \ \lambda_2\in [0,-\frac{3\lambda_0}{4}) = \bigcup_{j_2\geqslant 0}I_{j_2};\\
\lambda_3\in [\lambda_0,0) = \bigcup_{j_3\geqslant 0}(-I_{j_3}).
\end{array}\right.
\end{equation*}
In these two cases, we have $2^{j_1}\approx 2^{j_2}+2^ {j_3}$, hence $|\lambda_1-\lambda_3|\gtrsim \left \langle \lambda _{0} \right \rangle$. In fact, $j_1,j_2,j_3\lesssim \ln  \left \langle \lambda _{0}  \right \rangle $ and the frequency of $(\lambda_0,\lambda_2)$ and $(\lambda_1,\lambda_3)$ have transversality, so we can also use the similar way as in $h_-l_- l_{l\pm} $ in Step 1.2 to obtain the result.

{\it Case $7$.} Similar to {\it Case} $a_-a_+a_+$ as in Step 2.1, so we omit the details.

{\it Case $6$.} We decompose $\lambda_k$ in the following way:
\begin{align*}
&\lambda_1\in \frac{3\lambda_0}{4}+(-\infty,\lambda_0+C) = \bigcup_{j_1\gtrsim  \ln  \left \langle \lambda _{0}  \right \rangle}(\frac{3\lambda_0}{4}-I_{j_1});\ \ \lambda_2\in \frac{\lambda_0}{4}+[-\lambda_0,\infty ) = \bigcup_{j_2\gtrsim  \ln  \left \langle \lambda _{0}  \right \rangle}(\frac{\lambda_0}{4}+I_{j_2});\\
&\lambda_3\in [\lambda_0,0) = \bigcup_{j_3\geqslant 0}(-I_{j_3}).
\end{align*}
By FCC \eqref{FCC}, we have $j_1\approx j_2\geqslant j_3$. We see that the frequency of $(\lambda_0,\lambda_2)$ and $(\lambda_1,\lambda_3)$ have transversality. We can use bilinear estimate \eqref{bilinear4}, H\"older's inequality and Lemma \ref{V2toX} to obtain that
\begin{align}
&\mathscr{H}^{-}_{6} (u,v):=\sum_{\lambda _{0}\ll  0,\:j_3\lesssim \ln  \left \langle \lambda _{0}  \right \rangle \lesssim j_1\approx j_2}\left \langle \lambda _{0} \right \rangle^{1/2}  \int_{\mathbb{R}\times \left [ 0,T \right ]  }\left |\overline{v}_{\lambda _{0}} (\partial _{x}^{2}u_{3\lambda _{0}/4-I_{j_1}})u_{\lambda _{0}/4+I_{j_2}}\overline{u}_{-I_{j_3}} \right |dxdt\nonumber\\
&\lesssim \sum_{\lambda _{0}\ll  0,\:j_3\lesssim \ln  \left \langle \lambda _{0}  \right \rangle \lesssim j_1\approx j_2}\langle\lambda_0 \rangle^{1/2}\|{\overline{v}_{\lambda_0}u_{\lambda_0/4+I_{j_2}}} \|_{L^2_{x,t}}\|(\partial_x^2u_{3\lambda_0/4-I_{j_1}})\overline{u} _{-I_{j_3}}\|_{L^2_{x,t}}\nonumber\\
&\lesssim T^{\varepsilon/2}\sum_{\lambda _{0}\ll  0,\:j_3\lesssim \ln  \left \langle \lambda _{0}  \right \rangle \lesssim j_1\approx j_2}\langle\lambda_0 \rangle^{1/2}\cdot \left [ (\langle\lambda_0 \rangle+2^{j_2})^3 \right ]^{-1/2+\varepsilon }2^{j_2(1/2-1/q)}(\langle\lambda_0 \rangle+2^{j_2})^{-1/2}\nonumber\\
&\quad\quad \cdot (\langle\lambda_0 \rangle+2^{j_1})^2\left [ (\langle\lambda_0 \rangle+2^{j_1})^3 \right ]^{-1/2+\varepsilon }\cdot 2^{j_1(1/2-1/q)+j_3(-1/q)}(\langle\lambda_0 \rangle+2^{j_1})^{-1/2}\|{v_{\lambda_0}}\|_{V^2_S}\| u \|^3_{X^{1/2}_{q,S}}\nonumber\\
&\lesssim T^{\varepsilon/2}\sum_{\lambda _{0}\ll 0} \langle\lambda_0 \rangle^{-1/2-2/q+6\varepsilon }\|{v_{\lambda_0}}\|_{V^2_S}\| u \|^3_{X^{1/2}_{q,S}}\nonumber\\
&\lesssim T^{\varepsilon/2}\|v\|_{Y^0_{q',S}}\| u \|^3_{X^{1/2}_{q,S}}.
\end{align}
Here we choose $2^{j_1},2^{j_2}$ from $\left ( \left \langle \lambda _{0}\right \rangle+2^{j_1} \right )$ and $\left ( \left \langle \lambda _{0}\right \rangle+2^{j_2} \right )$, then make the summation on $j_3,j_1,j_2$ by using $ j_1\approx j_2\gtrsim \ln \langle \lambda _{0}\rangle$, and finally apply H\"older's inequality on $\lambda_0$.

$\mathbf{Step\:2.3.}$ We consider the case $\lambda _{2}=\underset{0\leqslant k\leqslant 3}{\mathrm{min} }\lambda _{k}$ and $\lambda _{0}\gg 0$. We need to consider the following cases by dividing the range of $\lambda _2$ as in Table \ref{step2.3}.

\begin{table}[h]
\begin{center}

\begin{tabular}{|c|c|c|c|}
\hline
${\rm Case}$   &  $\lambda_1\in $ & $\lambda_2\in $ & $\lambda_3\in $    \\
\hline
$a_+a_-a_+$  & $ [2\lambda_0-C,\infty)$  & $ [-\infty,0)$  &  $[\lambda_0,\infty)$  \\
\hline
$a_+l_+a_+$ & $ [3\lambda_0/2-C,\infty)$  & $ [0,\lambda_0/2)$ & $ [\lambda_0,\infty)$  \\
\hline
$h_+h_+h_+$ & $ [\lambda_0,2\lambda_0)$  & $ [\lambda_0/2,\lambda_0)$  & $ [\lambda_0,2\lambda_0+C)$\\
\hline
$a_+h_+a_+$  & $ [2\lambda_0,\infty)$  & $ [\lambda_0/2,\lambda_0)$  &  $[3\lambda_0/2-C,\infty)$  \\
\hline
\end{tabular}
\end{center}
\caption{$\lambda _2\leqslant \lambda _0 \leqslant \lambda _3 \leqslant \lambda _1,\ \ \lambda_0\gg  0$.}\label{step2.3}
\end{table}

{\it Case $a_+a_-a_+$.} We decompose $\lambda_k$ in the following way:
\begin{align*}
&\lambda_1\in \lambda_0+[\lambda_0-C,\infty ) = \bigcup_{j_1\gtrsim  \ln  \left \langle \lambda _{0}  \right \rangle}(\lambda_0+I_{j_1});\ \ \lambda_2\in (-\infty,0) = \bigcup_{j_2\geqslant 0}(-I_{j_2});\\
&\lambda_3\in [\lambda_0,\infty ) = \bigcup_{j_3\gtrsim  \ln  \left \langle \lambda _{0}  \right \rangle}I_{j_3}.
\end{align*}
From FCC \eqref{FCC} it follows that $j_1\approx (j_2\vee j_3)$. We see that the frequency of $(\lambda_0,\lambda_1)$ and $(\lambda_2,\lambda_3)$ have  transversality. We can use bilinear estimate \eqref{bilinear4}, H\"older's inequality and Lemma \ref{V2toX} to obtain that
\begin{align*}
&\mathscr{T}^{+}_{a_+a_-a_+} (u,v):=\sum_{\lambda _{0}\gg  0,\:j_1\approx (j_2\vee j_3)\gtrsim \ln  \left \langle \lambda _{0}  \right \rangle} \left \langle \lambda _{0} \right \rangle^{1/2}  \int_{\mathbb{R}\times \left [ 0,T \right ]  }\left |\overline{v}_{\lambda _{0}} (\partial _{x}^{2}u_{\lambda _{0}+I_{j_1}})u_{-I_{j_2}}\overline{u}_{I_{j_3}} \right |dxdt \nonumber\\
&\lesssim \sum_{\lambda _{0}\gg  0,\:j_1\approx (j_2\vee j_3)\gtrsim \ln  \left \langle \lambda _{0}  \right \rangle}\langle\lambda_0 \rangle^{1/2}\|{\overline{v}_{\lambda_0}\partial_x^2u_{\lambda_0+I_{j_1}}} \|_{L^2_{x,t}}\|u_{-I_{j_2}}\overline{u} _{I_{j_3}}\|_{L^2_{x,t}}\nonumber\\
&\lesssim T^{\varepsilon/2}\sum_{\lambda _{0}\gg  0,\:j_1\approx (j_2\vee j_3)\gtrsim \ln  \left \langle \lambda _{0}  \right \rangle}\langle\lambda_0 \rangle^{1/2}\cdot(\langle\lambda_0 \rangle+2^{j_1})^2\cdot \left (2^{3j_1}\right )^{-1/2+\varepsilon }2^{j_1(1/2-1/q)}\nonumber\\
&\quad\quad\cdot(\langle\lambda_0 \rangle+2^{j_1})^{-1/2}\cdot \left [2^{3(j_2\vee j_3)}\right ]^{-1/2+\varepsilon }\cdot 2^{(j_2+j_3)(-1/q)}\|{v_{\lambda_0}}\|_{V^2_S}\| u \|^3_{X^{1/2}_{q,S}}\nonumber\\
&\lesssim T^{\varepsilon/2}\sum_{\lambda _{0}\gg 0} \langle\lambda_0 \rangle^{-1/2-2/q+6\varepsilon }\|{v_{\lambda_0}}\|_{V^2_S}\| u \|^3_{X^{1/2}_{q,S}}\nonumber\\
&\lesssim T^{\varepsilon/2}\|v\|_{Y^0_{q',S}}\| u \|^3_{X^{1/2}_{q,S}}.
\end{align*}

{\it Case $a_+l_+a_+$.} We decompose $\lambda_k$ in the similar way as in last case, and have $j_3\approx j_1\geq j_2$. Thus using the same way as in last case, we can get the desired result.

{\it Case $h_+h_+h_+$.} In this case, we have $\langle \lambda _{1}\rangle\sim \langle \lambda _{2}\rangle\sim \langle \lambda _{3}\rangle\sim \langle \lambda _{0}\rangle$, so we can use the similar way as  {\it Case} $h_-h_-h_-$ in Step 1.2 to obtain the result and omit the details.

{\it Case $a_+h_+a_+$.} Similar to  Case $a_+l_+a_+$, we have $j_3\approx j_1\geq j_2$. Thus we could use the above approach to obtain the estimate as desired.

$\mathbf{Step\:2.4.}$ We consider the case $\lambda _{2}=\underset{0\leqslant k\leqslant 3}{\mathrm{min} }\lambda _{k}$ and $\lambda _{0}\ll 0$. The subcases are the same as in Table \ref{step2.2}, except that the positions of $\lambda _{1}$ and $\lambda _{2}$ need to be exchanged. {\it Case $1$}--{\it Case $6$} are similar to {\it Case $1$}--{\it Case $6$} in Step 2.2, {\it Case $7$} is similar to {\it Case $a_+a_-a_+$} in Step 2.3, so we omit the details.

$\mathbf{Step\:2.5.}$ We consider the case $\lambda _{3}=\underset{0\leqslant k\leqslant 3}{\mathrm{min} }\lambda _{k}$, it follows from FCC \eqref{FCC} we have $\lambda _0\approx  \lambda _1 \approx \lambda _2 \approx \lambda _3$, which is similar to $\mathscr{L}^{+}_{l} (u,v)$ in Step 1.1.

$\mathbf{Step\:3.}$ We consider the case $\left | \lambda _0 \right | \lesssim 1$. It follows that $\left \langle \lambda _0 \right \rangle ^{1/2}\sim 1$, which implies that we now have gained half order derivative. Besides, we can use biharmonic estimates instead of Strichartz estimates in the low frequency. As a result, this case becomes easier to manage, and we can skip the details.
\end{proof}

\subsection{Quintic estimates}

\begin{lem}
    Let $q\geq1$ and $0<\varepsilon\ll 1$, for any $u\in X_{q,S}^{\frac{1}{2}}$, we have
    \begin{align}
\left \| \int_{0}^{t}e^{i(t-\tau )\partial _{x}^{4} }\left ( u\left | u \right |^{4}    \right )(\tau )d\tau \right \|_{X_{q,S}^{\frac{1}{2}}}
\lesssim T^{\varepsilon }\left \| u \right \|_{X_{q,S}^{\frac{1}{2}}}^{5}.
\end{align}
\end{lem}

\begin{proof}

In view of (\ref{vie1}), it suffices to show that
\begin{align}
\left | \int_{\mathbb{R}\times \left [ 0,T \right ]  }\overline{v}u\left | u \right |^4 dxdt  \right |
\lesssim T^{\varepsilon }\left \| u  \right \|_{X_{q,S}^{\frac{1}{2}}}^{5}\left \| v \right \|_{Y_{q',S}^{-\frac{1}{2}}}.\label{5}
\end{align}
We perform a uniform decomposition with $u,v$ in the left-hand side of (\ref{5}), it suffice to prove that
\begin{align}
\left | \sum_{\lambda _{0},\dots,\lambda _{5}  } \left \langle \lambda _{0} \right \rangle^{1/2}  \int_{\mathbb{R}\times \left [ 0,T \right ]  }\overline{v}_{\lambda _{0}} u_{\lambda _{1}}\overline{u}_{\lambda _{2}} u_{\lambda _{3}}\overline{u}_{\lambda _{4}}u_{\lambda _{5}}(x,t)dxdt  \right |
\lesssim T^{\varepsilon }\left \| u \right \|_{X_{q,S}^{1/2}}^{5}\left \| v \right \|_{Y_{q',S}^{0}}
\end{align}

In order to keep the left-hand side of the inequality above nonzero, we have the new frequency constraint condition (FCC)
\begin{align}
\lambda _{0}+\lambda _{2}+\lambda _{4}\approx \lambda _{1}+\lambda _{3}+\lambda _{5}. \label{FCC2}
\end{align}
We can assume that $\lambda _{0}\geqslant 0$, because if it is not the case, we can replace $\lambda _{0},\dots ,\lambda _{5}$ with $-\lambda _{0},\dots ,-\lambda _{5}$.

{\bf Step 1.} Let us assume that $\lambda_0 = \max_{0\leq k\leq 5} |\lambda_k|$.
For short, considering the higher and lower frequency of $\lambda_k$,  we use the following notations:
$$
\left\{
\begin{array}{l}
\lambda_k \in h \Leftrightarrow \lambda_k\in [\lambda_0/8, \lambda_0] \\
\lambda_k \in h_- \Leftrightarrow \lambda_k\in [-\lambda_0, -\lambda_0/8]\\
\lambda_k \in l \Leftrightarrow \lambda_k\in [0, \lambda_0/8]\\
\lambda_k \in l_- \Leftrightarrow \lambda_k\in [-\lambda_0/8, 0]
\end{array}
\right.
$$
Combining FCC \eqref{FCC2}, we know that $\lambda_1,...,\lambda_5$ cannot all belong to the frequency intervals $l_{\pm}$.

{\bf Step 1.1.} We consider the case at least two higher frequency in $\lambda_1,...,\lambda_5$.
Without loss of generality, we can assume $\lambda _{1},\lambda _{2}$ belong to higher frequency intervals $h_{\pm}$. We decompose $\lambda_1,...,\lambda_5$ in a dyadic way:
\begin{align}
\lambda _{k}  \in \bigcup_{j_k\geqslant  0}(\lambda_0-  I_{j_k}), \ \  k=1,...,5. \label{decompose}
\end{align}

We can use H\"older's inequality and $L^4$ estimate in Lemma \ref{L4} to obtain that
\begin{align*}
&\mathscr{L}_{h_{\pm }h_{\pm }l_{\pm }l_{\pm }l_{\pm }}:=\sum_{\lambda _{0}\geqslant 0, j_{k} \lesssim \ln \langle \lambda _{0} \rangle ,k=1,...,5 } \left \langle \lambda _{0} \right \rangle^{1/2}  \int_{\mathbb{R}\times \left [ 0,T \right ]  }\overline{v}_{\lambda _{0}} u_{\lambda _{0}-I_{j_1}}\overline{u}_{\lambda _{0}-I_{j_2}} u_{\lambda _{0}-I_{j_3}}\overline{u}_{\lambda _{0}-I_{j_4}}u_{\lambda _{0}-I_{j_5}}dxdt \nonumber\\
&\lesssim \sum_{\lambda _{0}\geqslant  0,j_{k} \lesssim \ln \langle \lambda _{0} \rangle   ,k=1,...,5 } \left \langle \lambda _{0} \right \rangle^{1/2} \|v_{\lambda _0} \|_{L^\infty _{x,t}}\prod_{k=1,2}\|u_{\lambda _{0}-I_{j_k}} \|_{L^4_{x,t}}\prod_{k=3,4}\|u_{\lambda _{0}-I_{j_k}} \|_{L^4_{x,t}}\| u_{\lambda _{0}-I_{j_5}} \|_{L^\infty_{x,t}}\nonumber\\
&\lesssim T^\eps \sum_{\lambda _{0}\geqslant  0, j_{k} \lesssim \ln \langle \lambda _{0} \rangle ,k=1,...,5 } \left \langle \lambda _{0} \right \rangle^{1/2}\|v_{\lambda _0} \|_{V^2 _S}\cdot \left \langle \lambda _0 \right \rangle ^{-3/2}2^{(j_3+j_4)\varepsilon+(j_1+j_2)(1/4-1/q+\varepsilon )}2^{j_5/2} \|u\|^5_{X^{1/2}_{q,S}}\nonumber\\
&\lesssim T^\eps  \sum_{\lambda _{0}\geqslant  0}\langle \lambda _0 \rangle ^{-2/q+4\varepsilon }\|v_{\lambda _0} \|_{V^2 _S}\|u\|^5_{X^{1/2}_{q,S}}\nonumber\\
&\lesssim T^\eps  \left \| v \right \|_{Y_{q',S}^{0}} \left \| u \right \|_{X_{q,S}^{1/2}}^{5},
\end{align*}
where we use $\left \| v_{\lambda_0} \right \|_{L^\infty _{x,t}}\lesssim \left \| v_{\lambda_0} \right \|_{L^\infty _tL^2_x}\lesssim \left \| v_{\lambda_0} \right \|_{V^2_S}$ and the embedding
\begin{align*}
    &\|{u_{\lambda _{0}-I_{j_k}}}\|_{X^{-1/4}_{4,S}}\lesssim \| u_{\lambda _{0}-I_{j_k}} \|_{X^{1/2}_{q,S}}\,,(k=3,4);\quad
    &\|{u_{\lambda _{0}-I_{j_k}}}\|_{V^2_S}\lesssim \| u_{\lambda _{0}-I_{j_k}} \|_{X^{1/2}_{q,S}}\,,(k=5).
\end{align*}

% Next, we consider the case that there are three or more higher frequency in $\lambda_1,...,\lambda_5$. This situation is similar to the above situation and will be simpler because there will be more decay.

{\bf Step 1.2.} Now we study the case that there is only one higher frequency in $\lambda_1,...,\lambda_5$.
Without loss of generality, we can assume $\lambda _{1}$ belong to higher frequency intervals $h_{\pm}$ and $\lambda_2,...,\lambda_5$ belong to frequency intervals $l_{\pm}$. We decompose $\lambda_1,...,\lambda_5$ in a dyadic way as \eqref{decompose}.

We see that the frequency of $(\lambda_0,\lambda_2)$ has transversality. We can use bilinear estimate \eqref{bilinear3}, $L^4$ estimate \eqref{lebesgue4a} and H\"older's inequality to obtain that
\begin{align*}
&\mathscr{L}_{h_{\pm }l_{\pm }l_{\pm }l_{\pm }l_{\pm }}:=\sum_{\lambda _{0}\geqslant 0, j_{k} \lesssim \ln \langle \lambda _{0} \rangle ,k=1,...,5 } \left \langle \lambda _{0} \right \rangle^{1/2}  \int_{\mathbb{R}\times \left [ 0,T \right ]  }\overline{v}_{\lambda _{0}} u_{\lambda _{0}-I_{j_1}}\overline{u}_{\lambda _{0}-I_{j_2}} u_{\lambda _{0}-I_{j_3}}\overline{u}_{\lambda _{0}-I_{j_4}}u_{\lambda _{0}-I_{j_5}}dxdt \nonumber\\
&\lesssim \sum_{\lambda _{0}\geqslant 0, j_{k} \lesssim \ln \langle \lambda _{0} \rangle ,k=1,...,5 } \left \langle \lambda _{0} \right \rangle^{1/2} \|v_{\lambda _0}u_{\lambda _{0}-I_{j_2}} \|_{L^2_{x,t}}\|u_{\lambda _{0}-I_{j_1}} \|_{L^4_{x,t}}\|u_{\lambda _{0}-I_{j_3}} \|_{L^4_{x,t}}\prod_{k=4,5}^{} \| u_{\lambda _{0}-I_{j_k}} \|_{L^\infty_{x,t}}\nonumber\\
&\lesssim T^{3\varepsilon/4 }\sum_{\lambda _{0}\geqslant 0, j_{k} \lesssim \ln \langle \lambda _{0} \rangle,k=1,...,5 } \left \langle \lambda _{0} \right \rangle^{1/2}\cdot \left \langle \lambda _0 \right \rangle ^{3(-1/2+\varepsilon )}\cdot \left \langle \lambda _0 \right \rangle ^{-3/4}2^{j_1(1/4-1/q+\varepsilon )+j_3\varepsilon}\cdot  2^{(j_4+j_5)(1/2-1/q)}\nonumber\\
&\qquad\qquad\qquad\qquad\qquad\qquad\cdot \|v_{\lambda _0} \|_{V^2 _S} \|u\|^5_{X^{1/2}_{q,S}}\nonumber\\
&\lesssim T^{3\varepsilon/4 }\sum_{\lambda _{0}\geqslant 0}\left \langle \lambda _{0} \right \rangle^{-1/2-3/q+6\varepsilon }\|v_{\lambda _0} \|_{V^2 _S}\|u\|^5_{X^{1/2}_{q,S}}\nonumber\\
&\lesssim T^{3\varepsilon/4 }\left \| v \right \|_{Y_{q',S}^{0}} \left \| u \right \|_{X_{q,S}^{1/2}}^{5}
\end{align*}
Here we use $\|{u_{\lambda _{0}-I_{j_3}}}\|_{L^4_{x,t}}\lesssim 2^{j_3\varepsilon}\| u \|_{X^{1/2}_{q,S}}$, $\|{u_{\lambda _{0}-I_{j_k}}}\|_{L^\infty _{x,t}}\lesssim 2^{j_k/2-1/q}\| u \|_{X^{1/2}_{q,S}}$ $(k=4,5)$.

{\bf Step 2.} Let us assume that $ \lambda_0 $ is the second largest one in $|\lambda_0|,...,|\lambda_5|$. There exists $i\in \left \{ 1,\dots ,5 \right \}$ such that $\left | \lambda_i \right | = \max_{0\leq k\leq 5} |\lambda_k|$.
If $|\lambda_i|$ is the largest one in $|\lambda_0|, ..., |\lambda_5|$, we have from the frequency constraint condition \eqref{FCC2}
$$
\lambda_0 \leq  |\lambda_i | \leq 10\lambda_0.
$$
At this point, we can know that $\left \langle \lambda_i \right \rangle \sim \left \langle \lambda_0 \right \rangle$. That is to say, there exists a high frequency in $\lambda_1,...,\lambda_5$. We see that this case is quite similar to that of $\lambda_0$ to be the largest one as in  Step 1. Therefore the details are omitted.

%{\bf Step 2.1.} We study at least two higher frequency in $\lambda_1,...,\lambda_5$. We can use the same method as Step 1.1. Apart from $|\lambda_i|$, we decompose $\lambda_k$ in a dyadic way as \eqref{decompose}, and
%$$
%\lambda_i\in \bigcup_{j_i\geqslant  0}\pm(\lambda_0+I_{j_i}).
%$$
%Taking $L^\infty_{x,t}$ norms to $v_{\lambda _0}$ and $u$ of lower frequency, other parts are given to $L^4_{x,t}$ norms as Step 1.1, so much so that we can get the same result.

%{\bf Step 2.2.} We study the case that there is only one higher frequency in $\lambda_1,...,\lambda_5$, namely $\lambda_i$. Using the same dyadic decompose method as Step 2.1 and the same handling method as Step 1.2, we can reach the same conclusion, therefore omitting the proof details.

{\bf Step 3.} We assume that $\lambda_0$ is the third largest one in $|\lambda_0|,...,|\lambda_5|$. Namely, there is a bijection $\pi :\left \{ 1,2,3,4,5 \right \}\to \left \{ 1,2,3,4,5 \right \}$ such that
$$
|\lambda_{\pi(1)}| \geq  |\lambda_{\pi(2)}| \geq \lambda_0 \geq  |\lambda_{\pi(3)}| \geq   |\lambda_{\pi(4)}| \geq |\lambda_{\pi(5)}|
$$
We decompose $\lambda_{\pi(1)},...,\lambda_{\pi(5)}$ in a dyadic way:
\begin{align}\label{deco}
&\lambda_{\pi(k)} \in \bigcup_{j _{{\pi(k)}}\geqslant  0}\pm (\lambda_0+I_{j_{\pi(k)}}), \ \  k=1,2;
&\lambda_{\pi(k)} \in \bigcup_{j _{{\pi(k)}}\lesssim \mathrm{ln} \left \langle \lambda _{0}  \right \rangle}\pm I_{j_{\pi(k)}}, \ \  k=3,4,5.
\end{align}
We can use H\"older's inequality and $L^4$ estimate in Lemma\eqref{L4} to obtain that
\begin{align}
\mathscr{L}(u,v)&\lesssim \sum_{\lambda _{0},j _{{\pi(k)}}\geqslant 0 ,k=1,2;\,j _{{\pi(k)}} \lesssim \ln \langle \lambda _{0} \rangle ,k=3,4,5 }\left \langle \lambda _{0} \right \rangle^{1/2} \|v_{\lambda _0} \|_{L^\infty _{x,t}}\prod_{k=1,2}\|u_{\pm (\lambda _{0}+ I_{j_{\pi(k)}})} \|_{L^4_{x,t}}\nonumber\\
&\quad\quad \cdot \prod_{k=3,4}\|u_{\pm I_{j_{\pi(k)}}} \|_{L^4_{x,t}}\| u_{\pm I_{j_{\pi(5)}}} \|_{L^\infty_{x,t}}\nonumber\\
&\lesssim T^\eps \sum_{\lambda _{0},j _{{\pi(k)}}\geqslant 0 ,k=1,2;\,j _{{\pi(k)}} \lesssim \ln \langle \lambda _{0} \rangle ,k=3,4,5 }\left \langle \lambda _{0} \right \rangle^{1/2}\|v_{\lambda _0} \|_{V^2 _S}\nonumber\\
&\quad\quad \cdot \prod_{k=1,2}2^{j_{\pi(k)}(1/4-1/q+\varepsilon )}\left \langle \lambda _0+2^{j_{\pi(k)}} \right \rangle ^{-3/4}\|u\|_{X^{1/2}_{q,S}(\pm (\lambda_0+I_{j_{\pi(1)}}))}\|u\|_{X^{1/2}_{q,S}}\nonumber\\
&\quad\quad \cdot \prod_{k=3,4} 2^{j_{\pi(k)}(-1/2-1/q+\varepsilon )}\cdot  \|u\|^2_{X^{1/2}_{q,S}}\cdot 2^{j_{\pi(5)}(1/2-1/q)}\|u\|_{X^{1/2}_{q,S}}\nonumber\\
&\lesssim T^\eps \sum_{\lambda _{0},j _{{\pi(1)}}\geqslant 0}2^{j_{\pi(1)}(-1/2-1/q+\varepsilon )}
\|v_{\lambda _0} \|_{V^2 _S}\|u\|_{X^{1/2}_{q,S}(\pm (\lambda_0+I_{j_{\pi(1)}}))}\|u\|^4_{X^{1/2}_{q,S}}\nonumber\\
&\lesssim T^\eps \left \| v \right \|_{Y_{q',S}^{0}} \left \| u \right \|_{X_{q,S}^{1/2}}^{5},
\end{align}
where we first use
\begin{align*}
    &\| u_{\pm I_{j_{\pi(5)}}}\|_{L^\infty _{x,t}}\lesssim 2^{j_{\pi(5)}/2}\| u_{\pm I_{j_{\pi(5)}}}\|_{V^2_S} \quad\text{and}\quad    &\left \langle \lambda _0+2^{j_{\pi(2)}} \right \rangle ^{-3/4}\lesssim  \langle\lambda _0\rangle^{-1/2}2^{j_{\pi(2)}(-1/4)},
\end{align*}
then take the summation on $j_{\pi(k)}\,(k=2,\dots ,5)$ and apply H\"older's inequality on $\lambda_0$. Finally we take the summation on $j_{\pi(1)}$ for $0<\varepsilon <1/q$.

{\bf Step 4.} We assume that $\lambda_0$ is the fourth, or five largest one, or the minimal one in $|\lambda_0|,...,|\lambda_5|$. We can apply the similar decomposition to \eqref{deco}, then take $L^\infty_{x,t}$ norms for $v_{\lambda _0}$ and the minimal $u_{\lambda_{\pi(5)}}$, take $L^4_{x,t}$ norms for the remaining four items. It is similar to Step 3, so we omit the details. Now the proof is completed.
\end{proof}

\section{Global well-posedness}
%%%%%%%%%%%%%%%%%%%%%%%%%%%%%%%%%%%%%%%%%%%%%%%%%%%%%%%%%%%%%%%%%%%%%%%%%%%%%%%%%%%%%%%%%%%%%%%%%%%%%%%%%%%%%%%%%%%%%%%%%%%%%%
In this section, we exploit the conservation laws and a priori estimates in modulation spaces of the equation \eqref{4NLS} to extend the local well-posedness established in the previous section to global well-posedness.
\subsection{Conservation laws}
To guarantee the convergence of the series $\alpha(\kappa;u(t))$ and so that we can differentiate it term by term, we estimate its leading term and generator in the following lemma, which was mentioned in paper \cite{OhWang20} (Lemma A.7) without a proof.
\begin{lem}\label{le1}	For $\kappa \in \mathbb{C}$ with $\mathrm{Re}\, \kappa >0$ and $u\in \mathcal{S} \left (\mathbb{R}   \right ) $, we have
	  \begin{align}
	  \mathrm{Re}\,\mathrm{tr}\left \{ \left ( \kappa -\partial _{x}  \right )^{-1}u \left ( \kappa +\partial _{x}  \right )^{-1}\overline{u}   \right \}&=\int \frac{2\left (\mathrm{Re}\,\kappa   \right )\left | \widehat{u}\left ( \xi +2\mathrm{Im}\,\kappa  \right )   \right |^{2} }{4\left (\mathrm{Re}\,\kappa   \right )^{2}+\xi ^{2}  }  d\xi  ;\label{k1}\\
	  \left \| \left ( \kappa -\partial _{x}  \right )^{-\frac{1}{2} }u \left ( \kappa +\partial _{x}  \right )^{-\frac{1}{2}} \right \|_{\mathfrak{I}_{2}\left ( \mathbb{R}  \right )   }^{2}&\sim \int\mathrm{log}\left ( 4+\frac{\xi ^{2} }{\left (\mathrm{Re}\,\kappa   \right )^{2}}  \right )   \frac{\left | \widehat{u}\left ( \xi +2\mathrm{Im}\,\kappa  \right )   \right |^{2} }{\sqrt{4\left (\mathrm{Re}\,\kappa   \right )^{2}+\xi ^{2}  } } d\xi  \label{k2}.
	  \end{align}
\end{lem}

\begin{proof}	From Proposition \ref{lemchen}, we know that
	  \begin{align}
\mathrm{tr}\left \{ \left ( \kappa -\partial _{x}  \right )^{-1}u \left ( \kappa +\partial _{x}  \right )^{-1}\overline{u}   \right \}=&\frac{1}{2\pi }\iint\frac{\left | \hat{u}(\xi -\eta ) \right |^{2}  }{(\kappa -i\xi )(\kappa+i\eta )}d\xi d \eta\\\nonumber
	  =&\frac{1}{2\pi }\iint\frac{\left | \hat{u}(\xi -\eta ) \right |^{2}d\xi d \eta }{\left [\mathrm{Re}\kappa-i\left ( \xi -\mathrm{Im}\kappa \right )   \right ]\cdot \left [\mathrm{Re}\kappa+i\left (\eta +\mathrm{Im}\kappa \right )   \right ] }.
	  \end{align}
Multiplying the numerator and denominator by the conjugate of the denominator, we can obtain that
\begin{align}
\mathrm{LHS}\,\eqref{k1}
	  =\frac{1}{2\pi }\iint\frac{\left [ \left ( \mathrm{Re}\kappa \right )^{2}+(\xi -\mathrm{Im}\kappa)(\eta +\mathrm{Im}\kappa)\right ]\cdot \left | \hat{u}(\xi -\eta ) \right |^{2} }{\left [\left ( \mathrm{Re}\kappa \right )^{2}+ (\xi -\mathrm{Im}\kappa)^{2}  \right ]\cdot \left [\left ( \mathrm{Re}\kappa \right )^{2}+ (\eta  +\mathrm{Im}\kappa)^{2}  \right ] }d\xi d \eta ,
	  \end{align}
then by choosing $\xi-\mathrm{Im}\kappa=(\mathrm{Re}\kappa ) ( x+y )$ and $\eta +\mathrm{Im}\kappa=(\mathrm{Re}\kappa )( x-y )$, we have
	  \begin{align}
	  \mathrm{LHS}\,\eqref{k1}=&\frac{1}{\pi }\iint\frac{(1+x^{2}-y^{2} )\left | \hat{u}\left (2\mathrm{Im}\kappa+2(\mathrm{Re}\kappa)y  \right )  \right |^{2}  }{\left [ 1+(x+y)^{2}  \right ]\cdot \left [ 1+(x-y)^{2}\right ]  }dxdy \\\nonumber
	  =&\int \frac{\left | \hat{u}\left (2\mathrm{Im}\kappa+2(\mathrm{Re}\kappa)y  \right )  \right |^{2} }{1+y^{2} }dy,
  \end{align}
here we used the fact that
	  \begin{align}
	  \int_{\mathbb{R} }\frac{1+x^{2}-y^{2}}{\left [ 1+(x+y)^{2}  \right ]\cdot \left [ 1+(x-y)^{2}\right ]} dx=\frac{\pi }{1+y^{2} },
	  \end{align}
which can be calculated by using the residue theorem. Finally, we make the change of variables $2(\mathrm{Re}\kappa)y=\xi$ and get the conclusion \eqref{k1}.

	For \eqref{k2}, by applying (\ref{A}) and cycling the trace, we have
	  \begin{align*}
	\mathrm{LHS}\,\eqref{k2}=&\mathrm{tr}\left \{ \left ( \kappa -\partial _{x}  \right )^{-\frac{1}{2} } u \left ( \kappa +\partial _{x}  \right )^{-\frac{1}{2}}\left ( \overline{\kappa}  -\partial _{x}  \right )^{-\frac{1}{2} }\overline{u}  \left ( \overline{\kappa}  +\partial _{x}  \right )^{-\frac{1}{2}}\right  \}\\
	  =&\mathrm{tr}\left \{\left [ \left | \kappa  \right | ^{2}-\partial _{x}^{2} +2i(\mathrm{Im}\kappa)\partial _{x} \right ]^{-\frac{1}{2} } u\left [\left | \kappa  \right | ^{2}-\partial _{x}^{2} -2i(\mathrm{Im}\kappa)\partial _{x}  \right ]^{-\frac{1}{2} }\overline{u}\right \},
	  \end{align*}
Therefore, from Proposition \ref{lemchen} and taking the same change of variables as before, it follows that
	  \begin{align*}
	\mathrm{LHS}\,\eqref{k2}=&\frac{1}{2\pi }\iint\frac{\left | \hat{u}(\xi -\eta ) \right |^{2}d\xi d\eta}{\sqrt{\left [\left ( \mathrm{Re}\kappa \right )^{2}+ (\xi -\mathrm{Im}\kappa)^{2}  \right ]\cdot \left [\left ( \mathrm{Re}\kappa \right )^{2}+ (\eta  +\mathrm{Im}\kappa)^{2}  \right ]} }  \\
	  =&\frac{1}{\pi }\iint\frac{\left | \hat{u}(2(\mathrm{Re}\kappa)y +2\mathrm{Im}\kappa ) \right |^{2}}{\sqrt{(1+(x+y)^{2} )(1+(x-y)^{2} )} }dxdy\\
	  \sim &\int\left | \hat{u}(2(\mathrm{Re}\kappa)y +2\mathrm{Im}\kappa ) \right |^{2}\frac{\mathrm{log}(4+4y^{2} ) }{\sqrt{1+y^{2}}}dy\\
	  =&\int\mathrm{log}\left ( 4+\frac{\xi ^{2} }{\left (\mathrm{Re}\,\kappa   \right )^{2}}  \right )   \frac{\left | \widehat{u}\left ( \xi +2\mathrm{Im}\,\kappa  \right )   \right |^{2} }{\sqrt{4\left (\mathrm{Re}\,\kappa   \right )^{2}+\xi ^{2}  } }d\xi ,
	  \end{align*}
where we used the fact that
 $$\int \frac{1}{\sqrt{(1+(x+y)^{2} )(1+(x-y)^{2} )}}dx\sim\frac{\mathrm{log}(4+4y^{2} ) }{\sqrt{1+y^{2}}},$$
which can be obtained by breaking the region of integration into the pieces $|x|>3|y|/2$, $|x|\leq|y|/2$, $|x-y|\leq|y|/2$ and $|x+y|\leq|y|/2$.
\end{proof}

We are now ready to give the conservation of the perturbation determinant $\alpha(\kappa;u(t))$.

\begin{prop}[Conservation of $\alpha(\kappa;u)$]\label{conser0}
	Let $u(t)$ be a Schwartz solution to 4NLS \eqref{4NLS}. Then
	\begin{equation*}
	\ddt \alpha(\kappa;u(t)) = 0
	\end{equation*}
as soon as $\mathrm{Re}\, \kappa >0$ is sufficiently large such that
	\begin{equation}\label{error}
	\int_{\mathbb{R} } \mathrm{log}\left ( 4+\frac{\xi ^{2} }{\left (\mathrm{Re}\,\kappa   \right )^{2}}  \right )   \frac{\left | \widehat{u}\left ( \xi +2\mathrm{Im}\,\kappa  \right )   \right |^{2} d\xi   }{\sqrt{4\left (\mathrm{Re}\,\kappa   \right )^{2}+\xi ^{2}  } }\le c
	\end{equation}
holds for some absolute constant $c>0$.
\end{prop}

\begin{proof}
	Since \eqref{k2} and the condition \eqref{error} hold, the series of $\alpha(\kappa;u(t))$ converges and can be differenetiated term by term, so from \eqref{4NLS} we have
	\begin{equation*}
	\begin{split}
      \frac{d}{dt} \alpha(\kappa;u(t)) ={\rm Re}\sum_{\ell=1}^{\infty }{\rm tr }&\left \{\left [(\kappa +\partial_x)^{-1}\overline{u}(\kappa -\partial_x)^{-1}u \right]^{\ell-1}
      \right.\\ &\left.
      \left [\left (\kappa +\partial_x   \right )^{-1} \overline{u}_t\left (\kappa -\partial_x   \right )^{-1} u+\left (\kappa +\partial_x \right )^{-1}\overline{u}\left (\kappa -\partial_x  \right )^{-1} u_t \right ]  \right \}\\
      ={\rm Im}\sum_{\ell=1}^{\infty }{\rm tr }&\left \{ \left [(\kappa +\partial_x)^{-1}\overline{u}(\kappa -\partial_x)^{-1}u \right]^{\ell-1}(\kappa +\partial_{x}  )^{-1}
  \right.\\ & \left.
  (\overline{u}_{xxxx}+8\overline{u}_{xx}\left | u \right |^{2}+2u_{xx}\overline{u}^{2}+6\overline{u }_{x}^{2}u+4\overline{u}\left | u_{x}  \right |^{2}+6\overline{u}\left | u \right |^{4})(\kappa -\partial _{x})^{-1}u
  \right.\\ & \left.
  -\left [(\kappa +\partial_x)^{-1}\overline{u}(\kappa -\partial_x)^{-1}u \right]^{\ell-1}(\kappa +\partial _{x} )^{-1}\overline{u}(\kappa -\partial _{x} )^{-1}
  \right.\\ & \left.
  (u_{xxxx}+8u_{xx}\left | u \right |^{2}+2\overline{u}_{xx}{u}^{2}+6{u}_{x}^{2}\overline{u}+4u\left | u_{x}  \right |^{2}+6u\left | u \right |^{4})\right \}.
	\end{split}
	\end{equation*}
For convenience, denote $A(u):= (\kappa +\partial_{x})^{-1}\overline{u}(\kappa-\partial_{x})^{-1}u$, it suffices to show that
\begin{align}\label{cancel}
{\rm tr}\left \{(\kappa +\partial _{x})^{-1}\overline{u}_{xxxx}(\kappa -\partial _{x})^{-1}u
-(\kappa +\partial _{x})^{-1}\overline{u}(\kappa -\partial _{x})^{-1}{u}_{xxxx}  \right \}=0,
\end{align}
\begin{align}\label{cancel2}
&{\rm tr}\left \{ A(u)\Big[(\kappa +\partial _{x} )^{-1}\overline{u}_{xxxx}(\kappa -\partial _{x} )^{-1}u-(\kappa +\partial _{x} )^{-1}\overline{u}(\kappa -\partial _{x} )^{-1}u_{xxxx}\Big]    \right \}\nonumber\\
+&{\rm tr}\Bigl \{ (\kappa +\partial_{x}  )^{-1}(8\overline{u}_{xx}\left | u \right |^{2} +2u_{xx}\overline{u}^{2}+6\overline{u}_{x}^{2}u+4\overline{u}\left | u_{x}  \right |^{2})(\kappa -\partial_{x}  )^{-1}u \nonumber\\
&\quad\quad-(\kappa +\partial_{x}  )^{-1}\overline{u} (\kappa -\partial_{x}  )^{-1}(8u_{xx}\left | u \right |^{2} +2\overline{u}_{xx}u^{2}+6u_{x}^{2}\overline{u}+4u\left|u_{x}\right|^{2})\Bigr \}=0,
\end{align}
and
\begin{align}\label{cancel1}
&{\rm tr}\left \{ A(u)^\ell\Big[(\kappa +\partial _{x} )^{-1}\overline{u}_{xxxx}(\kappa -\partial _{x} )^{-1}u-(\kappa +\partial _{x} )^{-1}\overline{u}(\kappa -\partial _{x} )^{-1}u_{xxxx}\Big]    \right \}\nonumber\\
+&{\rm tr}\Bigl \{ A(u)^{\ell-1}\Bigl[(\kappa +\partial_{x}  )^{-1}(8\overline{u}_{xx}\left | u \right |^{2} +2u_{xx}\overline{u}^{2}+6\overline{u}_{x}^{2}u+4\overline{u}\left | u_{x}  \right |^{2})(\kappa -\partial_{x}  )^{-1}u \nonumber\\
&\quad\quad-(\kappa +\partial_{x}  )^{-1}\overline{u} (\kappa -\partial_{x}  )^{-1}(8u_{xx}\left | u \right |^{2} +2\overline{u}_{xx}u^{2}+6u_{x}^{2}\overline{u}+4u\left|u_{x}\right|^{2})\Bigr]\Bigr \}\nonumber\\
+&6{\rm tr}\Bigl \{ A(u)^{\ell-2} \Bigl[(\kappa +\partial_{x}  )^{-1}\overline{u}\left | u \right |^{4}(\kappa -\partial_{x}  )^{-1}u-(\kappa +\partial_{x}  )^{-1}\overline{u} (\kappa -\partial_{x}  )^{-1}u\left | u \right |^{4}\Bigr]\Bigr \}=0
\end{align}
for all $\ell\geq2$.

For \eqref{cancel}, we use Lemma \ref{lemchen} and obtain that
\begin{align*}
\text{LHS\ of\ } \eqref{cancel}=\iint\frac{\big((\xi_1-\xi_2)^4-(\xi_2-\xi_1)^4\big)|\hat{u}(\xi_1-\xi_2)|^2}{(\kappa+i\xi_1)
(\kappa-i\xi_2)}d\xi_1d\xi_2=0.
\end{align*}

To demonstrate the validity of \eqref{cancel1} and absorb $(\kappa +\partial _{x} )^{-1}$ and $(\kappa -\partial _{x} )^{-1}$ from either side, we rewrite
\begin{align*}
u_{xxxx}&=\partial ^{4} _{x}u-4\partial ^{3}_{x} u\partial _{x}+6\partial ^{2}_{x}u\partial ^{2}_{x}-4\partial_{x}u\partial ^{3}_{x}+u\partial ^{4}_{x} \\
&=(\partial ^{4} _{x}+4\kappa\partial ^{3}_{x}+6\kappa^{2}\partial ^{2}_{x}+4\kappa ^{3}\partial _{x} -7\kappa^{4})u+u(\partial ^{4} _{x}-4\kappa\partial ^{3}_{x}+6\kappa^{2}\partial ^{2}_{x}-4\kappa ^{3}\partial _{x} -7\kappa^{4})\\
&\quad+(\kappa -\partial _{x} )[4\partial ^{2}_{x}u-6\partial _{x}u\partial _{x}+4u\partial ^{2} _{x}+10\kappa u_{x}+14\kappa ^{2}u](\kappa +\partial _{x} ),
\end{align*}
\begin{align*}
\overline{u}_{xxxx}&=\partial ^{4} _{x}\overline{u}-4\partial ^{3}_{x}\overline{u}\partial _{x}+6\partial ^{2}_{x}\overline{u}\partial ^{2}_{x}-4\partial_{x}\overline{u}\partial ^{3}_{x}+\overline{u}\partial ^{4}_{x} \\
&=\overline{u}(\partial ^{4} _{x}+4\kappa\partial ^{3}_{x}+6\kappa^{2}\partial ^{2}_{x}+4\kappa ^{3}\partial _{x} -7\kappa^{4})+(\partial ^{4} _{x}-4\kappa\partial ^{3}_{x}+6\kappa^{2}\partial ^{2}_{x}-4\kappa ^{3}\partial _{x} -7\kappa^{4})\overline{u}\\
&\quad+(\kappa +\partial _{x} )[4\partial ^{2} _{x}\overline{u}-6\partial _{x}\overline{u}\partial _{x}+4\overline{u}\partial ^{2}_{x}-10\kappa \overline{u}_{x}+14\kappa ^{2}\overline{u}](\kappa -\partial _{x} ).
\end{align*}
The second identity here follows from the first by replacing $u$ with $\overline{u}$ and $\kappa$ with $-\kappa$.
When the first two terms from each identity are inserted into \eqref{cancel1}, their contributions cancel each other out, then we have 	
\begin{align}\label{cancel3}
&{\rm tr}\left \{ A(u)^\ell\Big[(\kappa +\partial _{x} )^{-1}\overline{u}_{xxxx}(\kappa -\partial _{x} )^{-1}u-(\kappa +\partial _{x} )^{-1}\overline{u}(\kappa -\partial _{x} )^{-1}u_{xxxx}\Big]    \right \}\nonumber\\
=&{\rm tr}\left \{ A(u)^{\ell-1}\Big[(\kappa +\partial _{x} )^{-1}\overline{u}(\kappa -\partial _{x} )^{-1}T_{1}(u)-(\kappa +\partial _{x} )^{-1}T_{2}(u)(\kappa -\partial _{x} )^{-1}u\Big]    \right \},
\end{align}
where
\begin{align*}	
T_{1}(u)= 4u\partial ^{2} _{x}| u |^{2}-6u\partial _{x}\overline{u} \partial _{x}u+4| u|^{2}   \partial ^{2}_{x}u-10\kappa u^{2}\overline{u}_{x}   +14\kappa ^{2}| u |^{2}u,\\
T_{2}(u)= 4\overline{u} \partial ^{2}_{x} | u |^{2}  -6\overline{u}\partial _{x}u \partial _{x}\overline{u}+4| u  |^{2}\partial ^{2} _{x}\overline{u} +10\kappa \overline{u}^{2}  u_{x}   +14\kappa ^{2}| u  |^{2}\overline{u}.
\end{align*}
Notice that
\begin{align*}	
u\partial ^{2} _{x}| u |^{2}&=\partial ^{2}_{x}u| u |^{2}-2\partial _{x}u_{x}|u|^{2}+u_{xx}| u |^{2},\\
| u|^{2} \partial ^{2}_{x}u&=| u |^{2}u_{xx}+2|u|^{2}u_{x}\partial _{x}+| u |^{2}u\partial ^{2}_{x},\\
u\partial _{x}\overline{u} \partial _{x}u&=(\partial _{x}u-u_{x})\overline{u}(u_{x}+u\partial _{x})=\partial _{x}| u |^{2}u\partial _{x}+(| u |^{2}u_{x} )_{x}-u_{x}^{2}\overline{u},
\end{align*}
hence we obtain
\begin{align}\label{cancel4}	
T_{1}(u)&=8u_{xx}| u |^{2}+6u_{x}^{2}\overline{u}+ 4\partial ^{2}_{x}| u |^{2}u+4| u |^{2}u\partial ^{2}_{x}-14(| u |^{2}u_{x} )_{x}-6\partial _{x}| u |^{2}u\partial _{x}\nonumber\\
&\quad-10\kappa u^{2}\overline{u}_{x}+14\kappa ^{2}| u |^{2}u.
\end{align}
Similarly,
\begin{align}\label{cancel5}
T_{2}(u)&=8\overline{u}_{xx}| u |^{2}+6\overline{u}_{x}^{2}u+4\partial ^{2}_{x}| u|^{2}\overline{u}+4| u |^{2}\overline{u}\partial ^{2}_{x}-14(| u |^{2}\overline{u}_{x} )_{x}-6\partial _{x}| u |^{2}\overline{u}\partial _{x}\nonumber\\
&\quad+10\kappa \overline{u} ^{2}u_{x}+14\kappa ^{2}| u |^{2}\overline{u}.
\end{align}
The first two terms in \eqref{cancel4} and \eqref{cancel5} are exactly the second derivative terms that appear in \eqref{cancel1}. For the remaining two second derivative terms in \eqref{cancel1}, note that
\begin{align}\label{cc}
2\overline{u}_{xx}u^{2}+4u| u_{x} |^{2}=2(u^{2}\overline{u} _{x})_{x}, \quad\quad  2u_{xx}\overline{u}^{2}+4\overline{u} | u_{x}  |^{2}=2(\overline{u}^{2}u_{x})_{x},
\end{align}
we rewrite
\begin{align*}
2(| u |^{2}u_{x} )_{x}&=(| u |^{2}u )_{xx}-(u^2\overline{u}_x)_x,\nonumber\\
\quad{\rm and}\quad(| u |^{2}u )_{xx}&=\partial ^{2}_{x}| u |^{2}u-2\partial _{x}| u |^{2}\overline{u}\partial _{x}+| u |^{2}u\partial ^{2}_{x},
\end{align*}
then we have
\begin{align}\label{cancel6}	
T_{1}(u)&=8u_{xx}| u |^{2}+6u_{x}^{2}\overline{u}+6(u^2\overline{u}_x)_x-2\partial ^{2}_{x}| u |^{2}u-2| u |^{2}u\partial ^{2}_{x}-2(| u |^{2}u_{x} )_{x}+6\partial _{x}| u |^{2}u\partial _{x}\nonumber\\
&\quad-10\kappa u^{2}\overline{u}_{x}+14\kappa ^{2}| u |^{2}u.
\end{align}
Similarly,
\begin{align}\label{cancel7}
T_{2}(u)&=8\overline{u}_{xx}| u |^{2}+6\overline{u}_{x}^{2}u+6(\overline{u}^{2}u_{x})_{x}-2\partial ^{2}_{x}| u|^{2}\overline{u}-2| u |^{2}\overline{u}\partial ^{2}_{x}-2(| u |^{2}\overline{u}_{x} )_{x}+6\partial _{x}| u |^{2}\overline{u}\partial _{x}\nonumber\\
&\quad+10\kappa \overline{u} ^{2}u_{x}+14\kappa ^{2}| u |^{2}\overline{u}.
\end{align}
In order to absorb $(\kappa +\partial _{x} )^{-1}$ and $(\kappa -\partial _{x} )^{-1}$ from either side again, we rewrite
\begin{align}
\partial _{x}| u | ^{2}u\partial _{x}&=-(\kappa -\partial _{x})| u | ^{2}u(\kappa +\partial _{x})+\kappa^{2} | u  | ^{2}u-\kappa (| u  | ^{2}u)_{x}, \label{cancel8}\\
\partial _{x}| u | ^{2}\overline{u}\partial _{x}&=-(\kappa +\partial _{x})| u | ^{2}\overline{u}(\kappa -\partial _{x})+\kappa^{2}| u | ^{2}\overline{u}+\kappa (| u | ^{2}\overline{u})_{x}.  \label{cancel9}
\end{align}
Therefore, inserting \eqref{cc}-\eqref{cancel9} into \eqref{cancel1} and \eqref{cancel3}, we can obtain that
\begin{align}\label{cancel10}
&{\rm tr}\left \{ A(u)^{\ell-1}\Big[(\kappa +\partial _{x} )^{-1}\overline{u}(\kappa -\partial _{x} )^{-1}T_{1}(u)-(\kappa +\partial _{x} )^{-1}T_{2}(u)(\kappa -\partial _{x} )^{-1}u\Big]    \right \},\nonumber\\
+&{\rm tr}\Bigl \{ A(u)^{\ell-1}\Bigl[(\kappa +\partial_{x}  )^{-1}(8\overline{u}_{xx}\left | u \right |^{2} +2u_{xx}\overline{u}^{2}+6\overline{u}_{x}^{2}u+4\overline{u}\left | u_{x}  \right |^{2})(\kappa -\partial_{x}  )^{-1}u \nonumber\\
&\quad\quad-(\kappa +\partial_{x}  )^{-1}\overline{u} (\kappa -\partial_{x}  )^{-1}(8u_{xx}\left | u \right |^{2} +2\overline{u}_{xx}u^{2}+6u_{x}^{2}\overline{u}+4u\left|u_{x}\right|^{2})\Bigr]\Bigr \}\nonumber\\
+&6{\rm tr}\Bigl \{ A(u)^{\ell-2} \Bigl[(\kappa +\partial_{x}  )^{-1}\overline{u}\left | u \right |^{4}(\kappa -\partial_{x}  )^{-1}u-(\kappa +\partial_{x}  )^{-1}\overline{u} (\kappa -\partial_{x}  )^{-1}u\left | u \right |^{4}\Bigr]\Bigr \}\nonumber\\
=&2{\rm tr}\left \{ A(u)^{\ell-1}\Big[(\kappa +\partial _{x} )^{-1}\overline{u}(\kappa -\partial _{x} )^{-1}H_{1}(u)-(\kappa +\partial _{x} )^{-1}H_{2}(u)(\kappa -\partial _{x} )^{-1}u \Big]\right \},
\end{align}
where
\begin{align*}	
H_{1}(u)= 2(u^2\overline{u}_x)_x-\partial ^{2}_{x}| u |^{2}u-| u |^{2}u\partial ^{2}_{x}-(| u |^{2}u_{x} )_{x}-5\kappa u^{2}\overline{u}_{x}+10\kappa ^{2}| u |^{2}u-3\kappa (| u  | ^{2}u)_{x},\\
H_{2}(u)=2(\overline{u}^{2}u_{x})_{x}-\partial ^{2}_{x}| u|^{2}\overline{u}-| u |^{2}\overline{u}\partial ^{2}_{x}-(| u |^{2}\overline{u}_{x} )_{x}+5\kappa \overline{u} ^{2}u_{x}+10\kappa ^{2}| u |^{2}\overline{u}+3\kappa (| u | ^{2}\overline{u})_{x}.
\end{align*}
Thus, it is enough to prove that the right-hand side of \eqref{cancel10} equals zero. This is based on the following identities:
\begin{align}	
(u^2\overline{u}_x)_x-\partial ^{2}_{x}| u |^{2}u-| u |^{2}u\partial ^{2}_{x}&=\partial _{x}\overline{u}_xu^2-u^2\overline{u}_x\partial _{x}-\partial ^{2}_{x}\overline{u}u^2-u^2\overline{u}\partial ^{2}_{x}=-\partial _{x}\overline{u}\partial _{x}u^2-u^2\partial _{x}\overline{u}\partial _{x}\nonumber\\
&=(\kappa -\partial _{x})\overline{u}\partial _{x}u^{2} -u^{2}\partial _{x}\overline{u}(\kappa +\partial _{x})-\kappa (| u  | ^{2}u)_{x}+2\kappa u^{2}\overline{u}_{x},	\nonumber\\
(\overline{u}^{2}u_{x})_{x}-\partial ^{2}_{x}| u|^{2}\overline{u}-| u |^{2}\overline{u}\partial ^{2}_{x}&=\overline{u}^{2} \partial _{x}u(\kappa -\partial _{x})-(\kappa +\partial _{x})u\partial _{x}\overline{u}^{2}+\kappa (| u  | ^{2}\overline{u})_{x}-2\kappa \overline{u}^{2}u_{x},\nonumber\\
(u^2\overline{u}_x)_x-(| u |^{2}u_{x} )_{x}&=-(\kappa -\partial _{x})\big(u^2\overline{u}_x-| u |^{2}u_{x}\big) -\big(u^2\overline{u}_x-| u |^{2}u_{x}\big)(\kappa +\partial _{x})\nonumber\\
&\quad+2\kappa \big(u^2\overline{u}_x-| u |^{2}u_{x}\big),	\nonumber\\
(\overline{u}^{2}u_{x})_{x}-(| u |^{2}\overline{u}_{x} )_{x}&=\big(\overline{u}^{2}u_{x}-| u |^{2}\overline{u}_{x}\big) (\kappa -\partial _{x}) +(\kappa +\partial _{x})\big(\overline{u}^{2}u_{x}-| u |^{2}\overline{u}_{x}\big)\nonumber\\
&\quad-2\kappa \big(\overline{u}^{2}u_{x}-| u |^{2}\overline{u}_{x}\big).\nonumber
\end{align}
By substituting these identities into the right-hand side of equation \eqref{cancel10}, we observe that the combined effect of the first four terms is zero, as well as that of the pair of second terms. Therefore, notice that $u^2\overline{u}_x+2| u |^{2}u_{x}=(| u |^{2}u)_x$, we have
\begin{align*}
{\rm RHS\ of\ }\eqref{cancel10}=10{\rm tr}\Bigl\{ &A(u)^{\ell-1}\Bigl[(\kappa +\partial _{x} )^{-1}\overline{u}(\kappa -\partial _{x} )^{-1}\big(2\kappa ^{2}| u |^{2}u-\kappa (| u  | ^{2}u)_{x}\big)\\
&-(\kappa +\partial _{x} )^{-1}\big(2\kappa ^{2}| u |^{2}\overline{u}+\kappa (| u  | ^{2}\overline{u})_{x}\big)(\kappa -\partial _{x} )^{-1}u \Bigr]\Bigr\}=0.
\end{align*}
This is because the net contribution of the following two terms is zero:
\begin{align*}
2\kappa ^{2}| u |^{2}u-\kappa (|u|^{2}u)_{x}=\kappa | u |^{2}u(\kappa +\partial _{x})+\kappa(\kappa -\partial _{x} ) |u|^{2}u,\\
2\kappa ^{2} | u |^{2}\overline{u}+\kappa (| u |^{2}\overline{u})_{x}=\kappa (\kappa +\partial _{x} )| u |^{2}\overline{u}+\kappa | u |^{2}\overline{u}(\kappa -\partial _{x}).
\end{align*}
The proof of \eqref{cancel1} is now finished. For \eqref{cancel2}, let $\ell=1$ in the preceding calculation, and we only need to observe that when plugging \eqref{cancel8} and  \eqref{cancel9} into  \eqref{cancel2}, the contribution of the two first terms is zero. The proof of Proposition \ref{conser0} is concluded.
\end{proof}

\subsection{Global bounds in Modulation spaces}

The size of the leading term in the series $\alpha(\kappa;u(t))$ can be related to the norm of modulation space, as stated in the following lemma.

\begin{lem}(Equivalent norms)\label{equi} Fix $0\leq s<1-\frac{1}{q}$, $2\leq q <\infty$, and $\kappa_0\geq1$, define
\begin{align}
\| f \|_{Z_{\kappa_0,q}^{s}} := \bigg(\sum_{n\in\mathbb{Z}} \langle n\rangle^{sq} \bigg(\int\frac{\kappa_0^2|\hat{f}(\xi+n)|^2}{4\kappa_0^2+\xi^2}d \xi\bigg)^{q/2}\bigg)^{1/q},
\end{align}
Then, we have
	$$\|f\|_{M^{s}_{2, q} }\lesssim \| f \|_{Z_{\kappa_0,q}^{s}}\lesssim \kappa_0 \|f\|_{M^{s}_{2, q} }.$$
\end{lem}

\begin{proof}By a simple calculation, we know that
$$\big\| \left \langle n \right \rangle ^{s}\|\left \langle \xi -n \right \rangle^{-1 }   \hat{f}  \|_{L_\xi ^2} \big\|_{\ell ^q_n}\lesssim\| f \|_{Z_{\kappa_0,q}^{s}}\lesssim \kappa_0 \big\| \left \langle n \right \rangle ^{s}\|\left \langle \xi -n \right \rangle^{-1}   \hat{f}  \|_{L_\xi ^2} \big\|_{\ell ^q_n}.$$
Since $\psi \in \mathcal{S}(\mathbb{R})$, it follows that $\psi(\xi -n)\lesssim \left \langle {\xi - n}\right \rangle^{-1}$, which yields that
$$\|f\|_{M^{s}_{2, q} }\lesssim \big\| \left \langle n \right \rangle ^{s}\|\left \langle \xi -n \right \rangle^{-1 }   \hat{f}  \|_{L_\xi ^2} \big\|_{\ell ^q_n}\lesssim \| f \|_{Z_{\kappa_0,q}^{s}}.$$

In the opposite direction, it follows from $s\geq0$ that $\langle n \rangle ^{2s}\lesssim \langle k-n \rangle ^{2s}+\langle k \rangle ^{2s}$,
	  \begin{align*}
&\big\| \left \langle n \right \rangle ^{s}\|\left \langle \xi -n \right \rangle^{-1}   \hat{f}  \|_{L_\xi ^2} \big\|^2_{\ell ^q_n}\\
\lesssim &\left\| \langle n \rangle ^{2s} \sum_{k\in \mathbb{Z}}\langle k - n\rangle^{-2} \|\hat{f}\|^2_{L_\xi ^2(I_{k})}\right\|_{\ell ^{\frac q2}_n}\\
\lesssim &\left\|\sum_{k\in \mathbb{Z}}\langle k - n\rangle^{2s-2} \|\hat{f}\|^2_{L_\xi ^2(I_{k})}\right\|_{\ell ^{\frac q2}_n}+\left\|\sum_{k\in \mathbb{Z}}\langle k - n\rangle^{-2} \langle k \rangle ^{2s}\|\hat{f}\|^2_{L_\xi ^2(I_{k})}\right\|_{\ell ^{\frac q2}_n}:={I_1}+{I_2}.
	  \end{align*}
By applying Young's inequality, we obtain that
$${I_2}\lesssim \bigg(\sum_{n\in \mathbb{Z}}\langle n \rangle^{-2}\bigg)\cdot
\big\|\langle n \rangle ^{s}\|\hat{f}\|_{L_\xi ^2(I_{n})}\big\|^2_{\ell ^{q}_n}\lesssim \|f\|^2_{M^{s}_{2, q} }.$$
When $s=0$, $I_1$ is the same as $I_2$.  When $s>0$, from Young's and H\"older's  inequalities, $1+\frac{2}{q} =\frac{1}{r}+\big(\frac{1}{r'}+\frac{2}{q}\big)$, it turns out that
\begin{align}\label{YH}
{I_1}\lesssim\|\langle n \rangle ^{2s-2} \|_{\ell _{n}^{r}}\cdot\|\langle n \rangle ^{-2s} \|_{\ell _{n}^{r'}}\cdot\big\|\langle n \rangle ^{s}\|\hat{f}\|_{L_\xi ^2(I_{n})}\big\|^2_{\ell ^{q}_n}\lesssim \|f\|^2_{M^{s}_{2, q} }
\end{align}
on the conditions that $2-2s>\frac1{r}$,$2s>\frac{1}{r'}$, and $\frac{1}{r'}+\frac{2}{q}\leq1$. In order to obtain the upper bound of $s$, we choose $r=\frac{q}{2}$, and obtain that the conditions are satisfied for $\frac{1}{2}-\frac{1}{q}<s<1-\frac{1}{q}$. For $0<s\leq \frac{1}{2}-\frac{1}{q}$, the conditions are satisfied as long as $r$ is chosen to satisfy $\frac{1}{r'}<2s$ (indeed we can choose $r=1$). Therefore, we conclude that the estimate \eqref{YH} holds for $s<1-\frac{1}{q}$. Thus the required conclusion follows.
\end{proof}

\noindent{\bf Proof of Theorem \ref{theorem2}.} It is well known that $M^{s}_{2, 2}=B^{s}_{2, 2}=H^s$, so the results in \eqref{Result} for the case $q=2$ and $-\frac12<s<1$ can be obtained in the same way as for the NLS and 5th-mKdV equations. We refer the readers to \cite{Chen23,KVZ18} for details and omit the proof here.

Next we consider the case $2< q< \infty $ and $0\leq s<1-\frac{1}{q}$.  In the conserved perturbation determinant $\alpha (\kappa;u(t))$, we choose $\kappa=\kappa_n:=\kappa_0+i\frac{n}{2}$ $(\kappa_0>1)$, then  from Lemma \ref{le1} and
$$\left \|( \kappa -\partial _{x} )^{-\frac{1}{2} }u ( \kappa +\partial _{x} )^{-\frac{1}{2}} \right \|_{\mathfrak{I}_{2}( \mathbb{R})   }=\left \|( \kappa +\partial _{x} )^{-\frac{1}{2} }\overline{u} ( \kappa -\partial _{x} )^{-\frac{1}{2}} \right \|_{\mathfrak{I}_{2}( \mathbb{R})},$$
we have for any $\delta>0$ that
\begin{align}
	  \bigg|\alpha(\kappa;u(t))-\int_{\mathbb{R}}\frac{2\kappa_0| \widehat{u}(\xi+n,t)|^{2} }{4\kappa_0^2+\xi ^{2}  }d\xi  \bigg|
 \leq&\sum_{l=2}^{\infty}\frac1l\left \|( \kappa -\partial _{x} )^{-\frac{1}{2} }u ( \kappa +\partial _{x} )^{-\frac{1}{2}} \right \|_{\mathfrak{I}_{2}}^{2l}\\ \nonumber
	 \lesssim&\sum_{l=2}^{\infty}\bigg( \int_{\mathbb{R} } \frac{| \widehat{u}( \xi,t )|^{2}  }{\big( 4\kappa_0^2+(\xi-n)^{2}\big)^{\frac{1}{2} -\delta }  }d\xi \bigg) ^{l}.
	  \end{align}
In order to ensure the convergence of the infinite sum above, we make the following estimate from Young's inequality that
\begin{align}
\int_{\mathbb{R} } \frac{| \widehat{u}( \xi,t )|^{2}  }{\big( 4\kappa_0^2+(\xi-n)^{2}\big)^{\frac{1}{2} -\delta }  }d\xi
&\lesssim \kappa_0^{-2\delta} \sum_{k \in \Z}
 \frac{1}{\langle k-n\rangle^{1-4\delta}}
\| \hat{u}(t, \xi)\|_{L^2_\xi(I_k)}^2\\ \nonumber
&\lesssim \kappa_0^{-2\delta} \|\langle k \rangle ^{4\delta-1} \|_{\ell _{k}^{\frac{q}{q-2}}}\|u\|^2_{M^{s}_{2, q} }\lesssim \kappa_0^{-2\delta} \|u\|^2_{M^{s}_{2, q} }
\end{align}
provided that $\delta=\delta(q)>0$ is sufficiently small(here we need $\delta<1/2q$). Therefore, we can choose
\begin{align}\label{Kappa0}
\kappa_0 \geq \max\{1, C \| u(0) \|^{\frac{1}{\delta}}_{M^{s}_{2, q}}\}
\end{align}
with some large absolute constant $C$, such that the series of $\alpha(\kappa;u(t))$ converges and Proposition  \ref{conser0} holds uniformly for any $t\in I$, where $I$ is a small closed interval containing $0$. Therefore, for any $t\in I$
\begin{align}
	  \bigg|\alpha(\kappa;u(t))-\int_{\mathbb{R}}\frac{2\kappa_0| \widehat{u}(\xi+n,t)|^{2} }{4\kappa_0^2+\xi ^{2}  }d\xi  \bigg|
\lesssim \kappa_0^{-4\delta} \bigg(\sum_{k \in \Z}
 \frac{1}{\langle k-n\rangle^{1-4\delta}}
\| \hat{u}(t, \xi)\|_{L^2_\xi(I_k)}^2 \bigg) ^{2}.
	  \end{align}
Multiply by $\langle n  \rangle ^{2s}$ and compute the $\ell^{\frac q2}_n$-norm, we have
\begin{align*}
	  &\bigg\| \langle n  \rangle ^{2s} \alpha(\kappa;u(t))-\langle n \rangle ^{2s}\int_{\mathbb{R}}\frac{2\kappa_0| \widehat{u}(\xi+n,t)|^{2} }{4\kappa_0^2+\xi ^{2}  }d\xi  \bigg\|_{\ell^\frac{q}{2}_n} \nonumber\\
 \lesssim& \kappa_0^{-4\delta}\bigg\| \sum_{k \in \Z}
 \frac{\langle n  \rangle ^{s}}{\langle k-n\rangle^{1-4\delta}}
\| \hat{u}(t, \xi)\|_{L^2_\xi(I_k)}^2 \bigg\|_{\ell^q_n}^2 \nonumber\\
 \lesssim& \kappa_0^{-4\delta}\bigg\| \sum_{k \in \Z}
 \frac{1}{\langle k-n\rangle^{1-4\delta-s}}
\| \hat{u}(t, \xi)\|_{L^2_\xi(I_k)}^2 \bigg\|_{\ell^q_n}^2 + \kappa_0^{-4\delta}\bigg\| \sum_{k \in \Z}
 \frac{\langle k  \rangle ^{s}}{\langle k-n\rangle^{1-4\delta}}
\| \hat{u}(t, \xi)\|_{L^2_\xi(I_k)}^2 \bigg\|_{\ell^q_n}^2.
	  \end{align*}
By choosing  $\delta=\delta(q)>0$ sufficiently small (indeed $\delta<1/4q$) and applying Young's inequality, we can get the second term above is controlled by $\kappa_0^{-4\delta} \|u\|^4_{M^{s}_{2, q} }$. For the first term, when $s=0$ it is the same as the second term. When $s>0$, proceeding as in Lemma \ref{equi}, from Young's and H\"older's  inequalities, $1+\frac{1}{q} =\frac{1}{l}+\big(\frac{1}{r}+\frac{2}{q}\big)$, we have
\begin{align}\label{YH2}
&\bigg\| \sum_{k \in \Z}
 \frac{1}{\langle k-n\rangle^{1-4\delta-s}}
\| \hat{u}(t, \xi)\|_{L^2_\xi(I_k)}^2 \bigg\|_{\ell^q_n}\nonumber\\
\lesssim&\|\langle n \rangle ^{s-1+4\delta} \|_{\ell _{n}^{l}}\cdot\|\langle n \rangle ^{-2s} \|_{\ell _{n}^{r}}\cdot\big\|\langle n \rangle ^{s}\|\hat{u}\|_{L_\xi ^2(I_{n})}\big\|^2_{\ell ^{q}_n}\lesssim \|u\|^2_{M^{s}_{2, q} }
\end{align}
on the conditions that $1-4\delta-s>\frac1{l}$, $2s>\frac{1}{r}$, and $\frac{1}{r}+\frac{2}{q}\leq1$. When $s\leq \frac{1}{2}-\frac{1}{q}$, we can choose $\frac{1}{r}=2s-\epsilon$ (as long as $0<\epsilon<\frac1q-4\delta $) to satisfy the conditions.  When $\frac{1}{2}-\frac{1}{q}<s<1-\frac{1}{q}$, we can choose $l=q,\ \frac{1}{r}=1-\frac2q$ and $\delta>0$ sufficiently small (indeed $\delta<\frac{1-s}{4}-\frac1{4q}$) to satisfy the conditions. Therefore, we conclude that
\begin{align*}
	  \bigg\| \langle n  \rangle ^{2s} \alpha(\kappa_n;u(t))-\langle n \rangle ^{2s}\int_{\mathbb{R}}\frac{2\kappa_0| \widehat{u}(\xi+n,t)|^{2} }{4\kappa_0^2+\xi ^{2}  }d\xi  \bigg\|_{\ell^\frac{q}{2}_n}
 \lesssim \kappa_0^{-4\delta}\|u\|^4_{M^{s}_{2, q} }
	  \end{align*}
for any $t\in I$. Then from Lemma \ref{equi} and the conservation of $\alpha(\kappa;u(t))$, we have
\begin{align*}
\| u(t)\|^2_{Z_{\kappa_0,q}^{s}}&\lesssim
 \kappa_0^{2} \big(\|\alpha(\kappa_n;u(t))\|_{\l^\frac{q}{2}_n}+\kappa_0^{-4\delta}\|u(t)\|^4_{M^{s}_{2, q} }\big)\nonumber\\
 &\lesssim\kappa_0^{2} \big(\kappa_0^{-2}\| u(0)\|^2_{Z_{\kappa_0,q}^{s}}+\kappa_0^{-4\delta}\|u(0)\|^4_{M^{s}_{2, q} }+\kappa_0^{-4\delta}\|u(t)\|^4_{M^{s}_{2, q} }\big)\nonumber\\
 &\lesssim \|u(0)\|^2_{Z_{\kappa_0,q}^{s}} +\kappa_0^{2-4\delta}\|u(0)\|^4_{Z_{\kappa_0,q}^{s} }+\kappa_0^{2-4\delta}\|u(t)\|^4_{Z_{\kappa_0,q}^{s}},
\end{align*}
it follows that for any $t\in I$,
\begin{align*}
\kappa_0^{2-4\delta}\| u(t)\|^2_{Z_{\kappa_0,q}^{s}}\lesssim
\kappa_0^{2-4\delta}\|u(0)\|^2_{Z_{\kappa_0,q}^{s}} +(\kappa_0^{2-4\delta}\|u(0)\|^2_{Z_{\kappa_0,q}^{s}})^2 +(\kappa_0^{2-4\delta}\|u(t)\|^2_{Z_{\kappa_0,q}^{s}})^2,
\end{align*}
Since the power of $\kappa_0$ is positive, we can not obtain smallness by choosing $\kappa_0$ sufficiently large, so the smallness of $\|u(0)\|_{Z_{\kappa_0,q}^{s}}$ is needed. Thus, without loss of generality, we fix $\kappa_0=C$. Suppose that
\begin{equation}\label{small}
\|u(0)\|_{Z_{\kappa_0,q}^{s}}\lesssim \|u(0)\|_{M^{s}_{2, q}} \leq  \varepsilon \ll 1,
\end{equation}
by a standard  continuity argument, we can get
$$\| u(t)\|_{Z_{\kappa_0,q}^{s}}\lesssim
\|u(0)\|_{Z_{\kappa_0,q}^{s}}$$
for all $t\in \mathbb{R}$, which implies that
$$\|u(t)\|_{M^{s}_{2, q} }\lesssim\| u(t)\|_{Z_{\kappa_0,q}^{s}}\lesssim
\|u(0)\|_{Z_{\kappa_0,q}^{s}}\lesssim \|u(0)\|_{M^{s}_{2, q} }.$$
This completes the proof for initial data with small $M^{s}_{2, q}$ norm.  For the general case, we use the scaling symmetry. For any solution $u(x,t)$ to \eqref{4NLS} with initial data $u(x,0)$, $u_{\lambda } (x,t)=\lambda u(\lambda x,\lambda^{4}t)$ with initial data $u_{\lambda } (x,0)$ is also a solution to \eqref{4NLS}. From \eqref{Scaling1}, we can choose sufficiently small $\lambda>0$ such that 	\begin{align}
	\| u_{\lambda }(0) \|_{M^{s}_{2,q}}\le C\lambda^{\frac{1}{q} }\| u(0)\|_{M^{s}_{2,q}}\le \varepsilon \ll 1, \label{eee15}
	\end{align}
which implies that we can obtain
\begin{align}\| u_{\lambda}(x,t)\| _{M^{s}_{2,q}}\lesssim \| u_{\lambda}(0)\| _{M^{s}_{2,q}}\label{global-t}\end{align}
for all $t\in \mathbb{R}$ by the small initial data results above. Indeed, we may choose
\begin{align}\lambda\sim \| u(0)\|_{M^{2,q}_{s} } ^{-q}.\label{global-t2}\end{align}
Therefore, from \eqref{Scaling2}, and \eqref{eee15}-\eqref{global-t2}, we have
\begin{align}
	\| u(t)\|_{M^{s}_{2,q}}\lesssim\lambda^{-s-\frac{1}{2}} \| u_{\lambda }(\lambda ^{-4}t) \|_{M^{s}_{2,q}}\lesssim\lambda^{-s-\frac{1}{2}+\frac{1}{q} }\| u(0)\|_{M^{s}_{2,q}}\lesssim C_{s,q}(\| u(0)\|_{M^{s}_{2,q}}).
	\end{align}
The theorem is now proved. $\hfill\Box$
\\

{\bf Acknowledgments.}

The first author M.J. Chen is supported in part by the National Natural Science Foundation of China under Grant No. 12001236, and Natural Science Foundation of Guangdong Province  under Grant No. 2020A1515110494.

\end{document}